\documentclass[11pt]{article}
\usepackage[dvips,letterpaper,margin=1.0in]{geometry}
\usepackage{sty}

\date{}

\title{Sharp Phase Transitions in Estimation with Low-Degree Polynomials}

\usepackage{authblk}
\author[1]{Youngtak Sohn\thanks{Email: \textit{youngtak\textunderscore sohn@brown.edu}.}}
\author[2]{Alexander S.\ Wein\thanks{Email: \textit{aswein@ucdavis.edu}. Partially supported by an Alfred P.\ Sloan Research Fellowship and NSF CAREER Award CCF-2338091.}}

\affil[1]{Division of Applied Mathematics, Brown University}
\affil[2]{Department of Mathematics, UC Davis}

\begin{document}

\maketitle

\begin{abstract}
High-dimensional planted problems, such as finding a hidden dense subgraph within a random graph, often exhibit a gap between statistical and computational feasibility.  While recovering the hidden structure may be statistically possible, it is conjectured to be computationally intractable in certain parameter regimes.  A powerful approach to understanding this hardness involves proving lower bounds on the efficacy of low-degree polynomial algorithms.  We introduce new techniques for establishing such lower bounds, leading to novel results across diverse settings: planted submatrix, planted dense subgraph, the spiked Wigner model, and the stochastic block model. Notably, our results address the estimation task --- whereas most prior work is limited to hypothesis testing --- and capture sharp phase transitions such as the ``BBP'' transition in the spiked Wigner model (named for Baik, Ben Arous, and P\'{e}ch\'{e}) and the Kesten--Stigum threshold in the stochastic block model. Existing work on estimation either falls short of achieving these sharp thresholds or is limited to polynomials of very low (constant or logarithmic) degree. In contrast, our results rule out estimation with polynomials of degree $n^{\delta}$ where $n$ is the dimension and $\delta > 0$ is a constant, and in some cases we pin down the optimal constant $\delta$. Our work resolves open problems posed by Hopkins \& Steurer (2017) and Schramm \& Wein (2022), and  provides rigorous support within the low-degree framework for conjectures by Abbe \& Sandon (2018) and Lelarge \& Miolane (2019). 

\end{abstract}

\tableofcontents

\newpage

\section{Introduction}
\label{sec:intro}

The task of discovering a hidden ``signal'' of interest buried in a large noisy dataset, is central to modern statistics and data science. In addition to the statistical question of discerning the weakest possible signal, the high-dimensionality of these problems also poses the computational challenge of finding an algorithm of practical runtime. As a testbed for studying the fundamental limitations of what is achievable in these settings, we focus on the following four canonical models for planted signal in a random matrix or random graph.

\begin{itemize}
    \item {\bf Planted Submatrix}: For a sparsity parameter $\rho \in [0,1]$ and a signal-to-noise parameter $\lambda \ge 0$, we observe the $n \times n$ matrix $Y = \lambda \theta\theta^\top + Z$ where $\theta \in \RR^n$ has i.i.d.\ $\Ber(\rho)$ entries and $Z \in \RR^{n \times n}$ is a symmetric matrix with $\{Z_{ij}\}_{i \le j}$ i.i.d.\ $\cN(0,1)$. Given $Y$, the goal is to estimate $\theta$.
    \item {\bf Planted Dense Subgraph}: For a sparsity parameter $\rho \in [0,1]$ and edge probabilities $0 \le p_0 \le p_1 \le 1$, we observe a random graph on $n$ vertices whose adjacency matrix $Y = (Y_{ij})_{1 \le i < j \le n} \in \{0,1\}^{\binom{n}{2}}$ is generated as follows. First draw $\theta \in \RR^n$ with i.i.d.\ $\Ber(\rho)$ entries. Conditioned on $\theta$, draw $Y_{ij} \sim \Ber(p_0 + (p_1-p_0) \theta_i \theta_j)$ independently for each $i < j$. Thus, edges within the planted subgraph have probability $p_1$ while others have probability $p_0$. Given $Y$, the goal is to estimate $\theta$.
    \item {\bf Spiked Wigner Model}: This is a canonical model for a low-rank matrix corrupted by additive Gaussian noise. Let $U \in \RR^{n \times m}$ have entries i.i.d.\ from some prior $\pi$ with mean 0 and variance 1. We observe the $n \times n$ matrix $Y = X + Z$ where $Z \in \RR^{n \times n}$ is symmetric with $\{Z_{ij}\}_{i \le j}$ i.i.d.\ $\cN(0,1)$, and $X = \sqrt{\lambda/n} \, UU^\top$ for a signal-to-noise parameter $\lambda \ge 0$. Given $Y$, the goal is to estimate $X$. Unlike most prior work, we allow $m$ to potentially grow with $n$.
    \item {\bf Stochastic Block Model (SBM)}: This is a canonical model for community detection in random graphs. To generate an $n$-vertex graph, first each vertex $i \in [n]$ is independently assigned a community label $\sigma_i^\star$, drawn uniformly from $[q]$ where $q$ is the number of communities. For edge probabilities $0 \le p_0 \le p_1 \le 1$, we observe the graph with adjacency matrix $Y = (Y_{ij})_{1 \le i < j \le n} \in \{0,1\}^{\binom{n}{2}}$ generated as follows. Independently for each $i < j$, draw $Y_{ij} \sim \Ber(p_0 + (p_1 - p_0)\one_{\sigma_i^\star = \sigma_j^\star})$. Thus, within-community edges have probability $p_1$ while cross-community edges have probability $p_0$. Given $Y$, the goal is to estimate whether two given vertices (say, vertices $1,2$) are in the same community. We will also consider the more general SBM where the communities can have different relative sizes and each pair of communities can have a different connection probability.
\end{itemize}
\noindent In each of these models we will be assuming an asymptotic regime where $n \to \infty$, and the other parameters (such as $\rho, p_0, p_1, m$) may scale with $n$ in some prescribed way or may be designated as fixed ``constants'' (such as $\pi,q$) that do not depend on $n$. Asymptotic notation such as $O(\cdot)$, $o(\cdot)$, $\Omega(\cdot)$, $\omega(\cdot)$, $\Theta(\cdot)$ will always pertain to this limit. The objective will be to \emph{estimate} (a.k.a.\ \emph{recover}) the planted signal (such as $\theta)$ to some desired accuracy, which may be measured in terms of mean squared error, or in terms of achieving some success metric \emph{with high probability}, i.e., success probability $1-o(1)$ as $n \to \infty$. We assume the parameters of the model (such as $\lambda,\rho,\pi$) are known to the statistician, but not the latent variables such as $\theta, U$.

The models defined above are all well studied, and we defer a thorough literature review to Section~\ref{sec:main-results}. Some prior work focuses on determining the \emph{statistical limits}, that is, for what values of the parameters is it possible versus impossible to succeed, with no restrictions on the estimator. Other work focuses on finding estimators that can be computed \emph{efficiently} (say, in polynomial time). Notably, all the above models appear to exhibit \emph{statistical-computational gaps}, meaning there is a ``possible-but-hard'' regime of parameters where some estimator is known to succeed via ``brute-force'' search, yet no polynomial-time algorithm is known to succeed. In such cases, it is desirable to understand whether this hardness is inherent: does there really not exist a poly-time algorithm, or do we just need to work harder to find one? Our focus in this work will be on identifying this transition between ``easy'' (poly-time solvable) and ``hard.'' Notably, the best known algorithms for the above models have ``sharp'' phase transitions in their behavior, where the problem abruptly becomes easy once a signal-to-noise parameter passes a certain threshold. Our aim will be to prove matching lower bounds, showing hardness below this precise threshold.

A key challenge in this endeavor is that classical notions of complexity such as NP-hardness are not applicable here, since we are dealing with \emph{average-case} problems where the goal is to succeed for (not all but) ``most'' random inputs from a particular distribution. Instead, a popular and versatile approach for rigorously vindicating the type of phase transitions we are interested in, is to study the behavior of \emph{low-degree polynomial estimators}~\cite{SW-estimation} (see~\cite{ld-survey} for a survey). To this end, we consider the task of estimating a particular scalar quantity, denoted $x$. For instance, in the planted submatrix problem, a natural choice is $x := \theta_1$, the first entry of the signal vector. The class of algorithms we will consider are multivariate polynomials $f \in \RR[Y]$ in the input variables $Y_{ij}$, of degree (at most) some parameter $D$ (which may scale with $n$). The best performance over all such algorithms is measured by the \emph{degree-$D$ minimum mean squared error},
\begin{equation}\label{eq:def:mmse}
\MMSE_{\le D} := \inf_{\substack{f \in \RR[Y] \\ \deg(f) \le D}} \EE[(f(Y)-x)^2],
\end{equation}
where the expectation is over the joint distribution of $(x,Y)$ as specified by the model. This scalar MMSE can generally be directly related to its vector analogue; see Section~\ref{sec:main-subm}. If the quantity $\MMSE_{\le D}$ is ``small'' (appropriately defined), we will say degree-$D$ polynomials succeed at the estimation task, and if $\MMSE_{\le D}$ is ``large'' then we will say degree-$D$ polynomials fail. The degree $D$ serves as a measure of an algorithm's complexity, and the above framework allows us to quantify the difficulty of an estimation task by the degree that is required to solve it. Our main results will establish that in the conjectured ``hard'' regimes for each of the four models, $\MMSE_{\leq D}$ does not even beat the ``trivial'' estimator $f(Y) \equiv \EE[x]$, that is, $\MMSE_{\leq D}\geq (1-o(1))\MMSE_{\leq 0}$, for some $D$ scaling as $n^\delta$ for a constant $\delta > 0$.

The above notion of ``degree complexity'' has conjectural connections to the more traditional notion of time complexity (runtime). For the style of problems we are considering, polynomials of degree $O(\log n)$ tend to be powerful enough to capture the best known poly-time algorithms, so if we can prove failure of super-logarithmic polynomials, this is considered an indication that there is no poly-time algorithm for the task (see~\cite[Conjecture~2.2.4]{hopkins-thesis}). For this reason, ``low'' degree typically means $O(\log n)$, unless stated otherwise. More generally, polynomials of degree $D$ are expected to have the same power as algorithms of runtime $n^{\tilde{O}(D)}$ where $\tilde{O}(\cdot)$ hides a $\mathrm{polylog}(n)$ factor (see~\cite[Hypothesis~2.1.5]{hopkins-thesis}), so for example, degree $n^\delta$ corresponds to time $\exp(n^{\delta \pm o(1)})$ for fixed $\delta > 0$. We refer the reader to~\cite{ld-survey} and references therein for further discussion of this framework. We emphasize that while these heuristics have a good track record for a particular style of problems, they certainly do not hold for arbitrary distributions; see, for instance,~\cite{ld-false,fails-subspace}. For the four models studied in this work, we will give low-degree \emph{upper bounds}, confirming that degree-$O(\log n)$ polynomials succeed at estimation in the ``easy'' regime where poly-time algorithms are known. This shows that low-degree estimators are a meaningful class of algorithms to consider here.

It is worth noting that we will depart from most prior work on low-degree complexity: the original work on this topic considered the setting of \emph{hypothesis testing} (a.k.a.\ \emph{detection}) rather than estimation~\cite{HS-bayesian,sos-detect} (see~\cite{hopkins-thesis,ld-notes,ld-survey} for exposition), and there have been numerous follow-up works in this setting, some of which are discussed in Section~\ref{sec:main-results}. The case of estimation that we consider here has also received attention~\cite{SW-estimation,LZ-tensor,KM-tree,tensor-decomp-LD,MW-amp,graphon,dense-cycle,HM-tree,tensor-cumulants,opt-bot} but it is more difficult to analyze and our mathematical toolbox is less complete. While results on testing can sometimes shed light on hardness of estimation, this is not the case for many problems --- including some that we study here --- due to gaps between the testing and estimation thresholds. For this reason it is important to have tools that directly address the estimation task. Existing lower bounds for low-degree estimation either do not reach the type of sharp thresholds that we aim for in this work, or are limited to very low degree (falling short of the super-logarithmic standard discussed above). Some examples of sharp thresholds at very low degree appear in~\cite{SW-estimation,MW-amp,HM-tree}. Lower bounds at very low degree are already interesting but do not necessarily give a reliable prediction for the true computational threshold, as discussed in, e.g.,~\cite{counting-stars}. Our work is the first to simultaneously capture sharp thresholds and rule out super-logarithmic degree in estimation. Specifically, our results will rule out polynomials of degree $n^\delta$ for a constant $\delta > 0$, sometimes for the \emph{optimal} constant $\delta$. Being the first results of this type, our work gives further credibility to the low-degree polynomial framework by demonstrating its ability to capture sharp estimation thresholds that coincide with the suspected computational limits. We note that concurrent and independent work of~\cite{opt-bot} also proves hardness at degree $n^\delta$ up to a sharp threshold, namely the Kesten--Stigum threshold in the broadcast tree model, using techniques that are rather different from ours.

We remark that there are other frameworks for explaining statistical-computational gaps, including average-case reductions (see~\cite{secret-leakage}), the statistical query model (see~\cite{sq-ld}), the sum-of-squares hierarchy (see~\cite{sos-survey}), and the overlap gap property (see~\cite{ogp-survey}). However, none of these appear able to capture the sharp estimation thresholds that we study here, at least with current techniques. One exception is the methods based on statistical physics (see~\cite{ZF-phys-survey,KZ-phys-notes}), which tend to be non-rigorous but do have an extraordinary track record of predicting sharp thresholds in estimation problems, such as the Kesten--Stigum (KS) threshold in the stochastic block model (SBM)~\cite{decelle} which we will also study here. These methods essentially postulate the optimality of specific algorithms such as \emph{belief propagation (BP)} or \emph{approximate message passing (AMP)}, which tend to be captured by the low-degree class (see~\cite{MW-amp,sos-amp-robustly,fourier-iter,alg-universality}). Since physics-style methods are quite different from low-degree methods, it is valuable to corroborate those predictions with rigorous evidence from low-degree approaches. This is especially true because low-degree methods appear to make more reliable predictions for a wider range of problems and scaling regimes, with \emph{tensor PCA}~\cite{RM-tensor-pca} being a prime example of a setting where the physics methods struggle (see~\cite{kikuchi}).

Finally, we note that concurrent and independent work~\cite{sbm-reduction} gives a different form of low-degree evidence for one of the threshold phenomena that we study, namely the KS threshold in the SBM. Rather than our approach of giving unconditional bounds on $\MMSE_{\le D}$, their work assumes a conjecture on a certain optimality of low-degree polynomials for hypothesis testing, and uses this to deduce hardness of non-trivial recovery by $\exp(n^{0.99})$-time algorithms. We refer to this type of argument as a \emph{detection-to-recovery reduction}. We discuss the comparison to our work further in Section~\ref{sec:main-sbm}. Another recent independent work is~\cite{alg-contig}, which gives evidence for sharp recovery thresholds in certain graph matching models, also using a detection-to-recovery reduction. More recently, these types of arguments are refined further by~\cite{alg-contig-2}.

Before presenting the specifics of our contributions, we briefly summarize the key conceptual advances of our paper, particularly given that our results span four distinct models.

\begin{enumerate}

    \item In one of the very first papers introducing the low-degree polynomial framework, Hopkins and Steurer showed a sharp phase transition for low-degree \emph{detection} at the KS threshold in the SBM~\cite[Section~1.4]{HS-bayesian}. They posed the question of showing the same for \emph{estimation}. We resolve this long-standing open problem affirmatively, and show similar sharp thresholds in other classical models.
    
    \item Some of the sharp thresholds we study, including the KS transition, were first predicted by statistical physics methods. We confirm these conjectures in the low-degree estimation framework, which we view as a stronger and more rigorous form of hardness. Furthermore, our results uncover new phase transitions in regimes where the physics predictions do not seem applicable or may not be reliable. These regimes include the spiked Wigner model with growing rank, and the planted dense subgraph model with sublinear subgraph size.

    \item We introduce a general framework for proving low-degree hardness of estimation, as outlined in Section~\ref{sec:pf-techniques}. Our approach generalizes the cumulant expansion technique of~\cite{SW-estimation}. In one model (spiked Wigner), the more basic cumulant expansion will suffice to prove our results, but even here we will need to introduce new ideas for bounding the cumulants. Since the initial posting of our paper, our techniques have already played a key role in a few different follow-up works~\cite{CMSW25,alg-contig,EGV:25}, and we anticipate these techniques will be broadly applicable to other high-dimensional planted problems.
    
\end{enumerate}

\subsection{Our Contributions}

\subsubsection{Example: Planted Submatrix}

We now illustrate our results in more detail, focusing on the planted submatrix problem. To recap the setting, a principal submatrix of size roughly $\rho n \times \rho n$ with elevated mean $\lambda$ is planted in an $n \times n$ symmetric Gaussian matrix. The goal is to estimate $\theta_1$, the indicator for whether vertex $1$ belongs to the planted submatrix. We state (a simplification of) our main result for this setting.

\begin{theorem}[See Theorem~\ref{thm:main-subm}]
Consider the planted submatrix model with $\rho = o(1)$. For any constant $\eps > 0$ there exists a constant $C \equiv C(\eps) > 0$ such that the following holds for all sufficiently large $n$. If
\[ \lambda \le (1-\eps)(\rho\sqrt{en})^{-1} \qquad \text{and} \qquad D \le \lambda^{-2}/C \]
then
\[ \MMSE_{\le D} := \inf_{\substack{f \in \RR[Y] \\ \deg(f) \le D}} \EE[(f(Y)-\theta_1)^2] \ge \rho - C \rho^2. \]
\end{theorem}

To interpret this, we restrict the following discussion to the regime $\Omega(1/\sqrt{n}) \le \rho \le o(1)$ for ease of exposition, with the general case discussed in Section~\ref{sec:main-subm}. First note that the trivial MSE achieved by $f(Y) \equiv \EE[\theta_1] = \rho$ is $\Var(\theta_1) = \rho - \rho^2$, so we have met our stated goal of showing $\MMSE_{\leq D}\geq (1-o(1))\MMSE_{\leq 0}$. 
Next, a poly-time algorithm based on \emph{approximate message passing (AMP)} is known to achieve near-perfect estimation of $\theta$ --- meaning $o(\rho n)$ misclassification errors, with high probability --- when $\lambda \ge (1+\eps) (\rho\sqrt{en})^{-1}$ for an arbitrary constant $\eps > 0$~\cite{submatrix-amp}. Our condition on $\lambda$ holds below this sharp threshold. Finally, below the AMP threshold $\lambda \le (1-\eps)(\rho\sqrt{en})^{-1}$, the best known algorithms have runtime $\exp(\tilde{O}(\lambda^{-2}))$~\cite{subexp-sparse,anytime-sparse} (so long as $\lambda \gg (\rho n)^{-1/2}$, which is information-theoretically necessary~\cite{kolar-info,BIS-info}; we use $\gg$ to hide $\polylog(n)$ factors). Recalling the heuristic correspondence between degree and runtime discussed above, our condition on $D$ coincides with this runtime $\exp(\tilde{O}(\lambda^{-2}))$. Thus, our lower bound suggests that below the sharp AMP threshold, runtime $\exp(\tilde{\Omega}(\lambda^{-2}))$ is required. This shows optimality of the existing algorithms in a strong sense, pinning down both the sharp AMP threshold and also the precise degree $\Omega(\lambda^{-2})$. This resolves two different open questions of~\cite{SW-estimation}, where coarser low-degree lower bounds were shown.

We emphasize that the hardness of estimation established above cannot be deduced from the statistical limits of estimation~\cite{kolar-info,BIS-info} nor the computational limits of hypothesis testing~\cite{MW-reduction,BBH-reduction}, since these have different thresholds; see~\cite{SW-estimation}.

\subsubsection{Other Models}

We now summarize our contributions for the other three models presented above, with more details in Section~\ref{sec:main-results}.

For planted dense subgraph, we give the same lower bound as for planted submatrix but with the substitution $\lambda^2 = (p_1 - p_0)^2/[p_0(1-p_0)]$; see Theorem~\ref{thm:main-pds}. In particular, our lower bound reaches a sharp threshold, and we also show that degree-$O(\log n)$ polynomials succeed above this threshold. However, the matching poly-time algorithm does not seem to be known. Thus, our result suggests a new phase transition phenomenon that may be achievable with an AMP-style algorithm. 

The spiked Wigner model and stochastic block model both have well-established conjectured computational thresholds for the onset of weak recovery, i.e., non-trivial estimation. These thresholds are known as the \emph{BBP transition} (after Baik, Ben Arous, and P\'{e}ch\'{e}~\cite{BBP}) and \emph{Kesten--Stigum (KS) bound}~\cite{KestenStigum:66, decelle}, respectively. Specifically, it was conjectured in \cite[Conjecture 10]{LM-wigner} (based on~\cite{LKZ-sparse,mi-rank-1}) and \cite[Conjecture 2]{AS-acyclic} (based on~\cite{decelle}) that below the BBP transition and the KS bound, respectively, non-trivial estimation is impossible in polynomial time for these models. Existing low-degree lower bounds show hardness of the associated hypothesis testing problem below these thresholds~\cite{HS-bayesian,ld-notes,spectral-planting}, but no such results were known for estimation and this was stated as an open problem by Hopkins and Steurer in one of the first papers on the low-degree framework~\cite{HS-bayesian}. We resolve this, giving estimation lower bounds at degree $n^\delta$ for a constant $\delta > 0$, below the sharp BBP and KS thresholds. We refer to Theorem~\ref{thm:main-wigner} for the spiked Wigner model, and to Theorem~\ref{thm:main-sbm} for the SBM. Altogether, our results provide rigorous evidence towards the conjectures of~\cite{LM-wigner} and~\cite{AS-acyclic} by establishing the corresponding low-degree hardness. (We do not realistically expect to resolve these conjectures outright, as this would imply a solution to the P versus NP problem.) In the spiked Wigner model, we additionally go beyond the setting of the conjecture by allowing the rank $m$ to grow with $n$.

\subsubsection{Low-Degree Upper Bounds}

In the discussion above, we have focused on low-degree \emph{lower bounds}, that is, proving failure of degree-$D$ polynomials. Such results are most meaningful when complemented by the corresponding low-degree \emph{upper bound}, that is, we would like to show $\MMSE_{\le D}$ is ``small'' for some $D = O(\log n)$ in the ``easy'' regime where efficient algorithms are known, in order to rigorously establish a phase transition in the behavior of polynomials.

Indeed, we will show for our four models of interest that degree-$O(\log n)$ polynomials succeed above the sharp thresholds appearing in our lower bounds. More specifically, for planted submatrix and planted dense subgraph we will show that degree-$O(\log n)$ polynomials achieve \emph{strong recovery} above the threshold, meaning $\MMSE_{\le D} = o(\MMSE_{\le 0})$. For spiked Wigner and SBM we will give a similar result but for \emph{weak recovery}, meaning $\MMSE_{\le D} = (1-\Omega(1)) \MMSE_{\le 0}$. The full details, along with further discussion, are presented in Section~\ref{sec:main-results}.

For planted submatrix and planted dense subgraph, our upper bound is proved by constructing a polynomial estimator that is the average of certain ``tree-structured'' monomials with a particular structure. For spiked Wigner, a sharp low-degree upper bound above the BBP transition was previously known in the rank-1 case ($m=1$), based on self-avoiding walks~\cite{HS-bayesian}. We will extend this to larger $m$. Similarly, for the SBM, a sharp low-degree upper bound above the KS threshold is known, also based on self-avoiding walks~\cite{HS-bayesian}. We will include the proof for completeness.

\subsection{Proof Techniques}
\label{sec:pf-techniques}

Here we present the main ideas in the lower bound proofs.

\subsubsection{Overview of the Proof Strategy}
\label{sec:overview-proof}

Consider for now the generic setting where the goal is to estimate some scalar $x$ given $Y$. Instead of $\MMSE_{\le D}$, it will be more convenient to work with the \emph{degree-$D$ maximum correlation},
\[ \Corr_{\le D} := \sup_{\substack{f \in \RR[Y] \\ \deg(f) \le D}} \frac{\EE[f(Y) \cdot x]}{\sqrt{\EE[f(Y)^2] \cdot \EE[x^2]}}=\sup_{\substack{f \in \RR[Y] \\ \deg(f) \le D}} \frac{\big|\EE[f(Y) \cdot x]\big|}{\sqrt{\EE[f(Y)^2] \cdot \EE[x^2]}} , \]
where the last equality holds by considering $-f$, if necessary. We have taken a normalization such that $\Corr_{\le D} \in [0,1]$, which deviates from the one used by~\cite{SW-estimation}. The following fact shows that $\Corr_{\le D}$ and $\MMSE_{\le D}$ are directly related.
\begin{fact}[\cite{SW-estimation}, Fact~1.1]
\label{fact}
$\MMSE_{\le D} = (1 - \Corr_{\le D}^2) \, \EE[x^2]$.
\end{fact}

\noindent We aim to give an upper bound on $\Corr_{\le D}$, and the main difficulty is the factor $\EE[f(Y)^2]$ in the denominator, since $Y$ does not have independent entries. We will make use of the underlying independent random variables from which $Y$ is generated, for instance $Z$ and $\theta$ in the planted submatrix problem.

Choose a basis $\{\phi_\alpha\}_{\alpha \in \cI}$ for $\RR[Y]_{\le D}$ (polynomials in the entries of $Y$ of degree $\le D$) so that any candidate polynomial estimator can be expanded as
\[ f(Y) = \sum_\alpha \hat{f}_\alpha \phi_\alpha(Y) \]
for some vector of (real) coefficients $\hat{f} = (\hat{f}_\alpha)_{\alpha \in \cI}$. In the settings we consider there will always be a vector $W$ of independent random variables from which $x$ and $Y$ are generated. We will need to construct an orthonormal collection $\{\psi_\beta\}_{\beta \in \cJ}$ of polynomials in (the entries of) $W$. The meaning of ``orthonormal'' here is $\langle \psi_\beta, \psi_{\beta'} \rangle := \EE[\psi_\beta(W) \cdot \psi_\beta'(W)] = \one_{\beta = \beta'}$. There is no requirement on the degrees of the polynomials $\psi_\beta$ nor on their span (they need not form a full basis), but it will be advantageous to use a ``rich" collection of polynomials. We emphasize that, unlike $\{\psi_\beta\}$, the basis $\{\phi_\alpha\}$ need not be orthogonal. Working in the Hilbert space of random variables, we can bound the norm of $f(Y)$ using its projections onto the orthonormal directions $\psi_\beta$:
\[ \EE[f(Y)^2] \ge \sum_\beta \EE[f(Y) \cdot \psi_\beta(W)]^2. \]
Equivalently, $\EE[f(Y)^2] \ge \|M \hat f\|^2$ where the matrix $M = (M_{\beta\alpha})_{\beta \in \cJ, \alpha \in \cI}$ is defined by
\[ M_{\beta\alpha} = \EE[\phi_\alpha(Y) \cdot \psi_\beta(W)]. \]
Also define the vector $c = (c_\alpha)_{\alpha \in \cI}$ by
\[ c_\alpha = \EE[\phi_\alpha(Y) \cdot x]. \]
The core of our approach will be to construct a certain vector $u$, which then implies a bound on $\Corr_{\le D}$.

\begin{proposition}
\label{prop:master}
Suppose there is a vector $u = (u_\beta)_{\beta \in \cJ}$ satisfying the linear equations $u^\top M = c^\top$, or equivalently,
\begin{equation}\label{eq:u-cond}
\sum_{\beta \in \cJ} u_\beta M_{\beta\alpha} = c_\alpha
\end{equation}
for all $\alpha \in \cI$. Then
\[ \Corr_{\le D} \le \frac{\|u\|}{\sqrt{\EE[x^2]}} := \sqrt{\frac{\sum_{\beta \in \cJ} u_\beta^2}{\EE[x^2]}}. \]
\end{proposition}

\begin{proof}
Recalling the fact $\EE[f(Y)^2] \ge \|M\hat{f}\|^2$,
\[ \sqrt{\EE[x^2]} \cdot \Corr_{\le D} = \sup_f \frac{\EE[f(Y) \cdot x]}{\sqrt{\EE[f(Y)^2]}} \le \sup_{\hat f} \frac{c^\top \hat f}{\|M \hat f\|} = \sup_{\hat f} \frac{u^\top M \hat f}{\|M \hat f\|} \le \sup_{\hat f} \frac{\|u\| \cdot \|M \hat f\|}{\|M \hat f\|} = \|u\|. \]
\end{proof}

Provided $x$ lies in the span of $\{\psi_\beta\}$, it is not difficult to construct \emph{some} vector $u$ satisfying the constraints $u^\top M = c^\top$, as one solution is $u_\beta = \EE[\psi_\beta \cdot x]$ which can be seen from the expansion $x = \sum_{\be} \E[x \cdot \psi_{\be}] \, \psi_{\be}$. However, this gives only the trivial bound $\Corr_{\le D} \le 1$. Ideally, we would take $u$ to be the minimum-norm solution to the constraints, which can be found using the Moore--Penrose pseudoinverse: $u^\top = c^\top M^+$. The issue here is that this choice of $u$ seems difficult to work with explicitly since we do not have a closed-form expression for $M^+$. Instead, our approach will be to manually construct a more tractable solution $u$ with a small enough norm. As we will see in the subsequent sections, our proof leverages a key geometric insight that ``tree-structured'' polynomials are most informative for estimation in the planted submatrix/subgraph problems whereas ``path-structured'' polynomials are most effective for spiked Wigner and SBM. In Section~\ref{sec:ex-subm}, we illustrate our construction of $u$ for the planted submatrix model. Although this solution is explicit in that model, our approach to the stochastic block model is more complicated, where our construction of $u$ is defined recursively over a certain set of graphs (see Section~\ref{subsec:SBM:u}), necessitating a delicate analysis to control its $\ell_2$ norm.

We are hopeful that the above strategy may be useful to prove new low-degree lower bounds in settings beyond those considered in this work. This new approach has not appeared before in its current form, although it takes some inspiration from the prior works~\cite{SW-estimation,tensor-decomp-LD}\footnote{Additionally, we thank Jonathan Niles-Weed for discussions that helped inspire this approach.} which also crucially use orthogonal polynomials in the underlying independent random variables. Compared to~\cite{tensor-decomp-LD}, our approach is a generalization (the $\psi_\beta$ may not form a basis) and simplification (we only need to construct $u$ solving $u^\top M = c^\top$ rather than a left-inverse for $M$). While~\cite{tensor-decomp-LD} studies a different problem from us (tensor decomposition, which does not appear to have sharp threshold behavior),~\cite{SW-estimation} studies some of the same problems that we do, and we sharpen the bounds in a qualitative way, allowing us to capture sharp thresholds. We will give more details on the comparison to~\cite{SW-estimation} in Section~\ref{sec:ex-subm} below.

We remark that for the spiked Wigner model, our proof utilizes the cumulant upper bound from \cite{SW-estimation} (see Section~\ref{sec:cumulant}) and does not fully exploit Proposition~\ref{prop:master}. Nevertheless, within this cumulant expansion framework, we introduce a novel technique to systematically bound the cumulants. This approach enables us to attain the BBP threshold even when the rank grows. Although this contribution could have been presented separately, we chose to include it here, as it complements our broader set of techniques for handling the four canonical models in high-dimensional statistics. Indeed, shortly after the initial posting of this paper, our techniques developed for the spiked Wigner model were applied to the Gaussian mixture model in~\cite{EGV:25}. Specifically, Appendices~C and~D of that paper closely follow our methods detailed in Lemmas~\ref{lem:spiked:cumulant} and~\ref{lem:f:upper:bound}.

\subsubsection{Removing ``Bad'' Terms}

We will designate some values of $\alpha$ and $\beta$ as ``good,'' denoted $\alpha \in \hat\cI$ or $\beta \in \hat\cJ$, and all others as ``bad'' (denoted, e.g., $\alpha \notin \hat\cI$). The choice of $\hat\cI \subseteq \cI$ and $\hat\cJ \subseteq \cJ$ will be problem-specific and will depend on the estimand $x$, but intuitively the good basis elements represent those that are potentially useful, or ``informative,'' for estimation of $x$ whereas the bad ones represent those that ``uninformative.'' Under certain conditions we will be able to remove the bad terms from consideration. We refer to Sections~\ref{sec:ex-subm} and~\ref{subsec:other} for specific examples.

\begin{lemma}\label{lem:disconn}
Suppose the following conditions are met:
\begin{itemize}
\item $u_\beta = 0$ for all $\beta \notin \hat\cJ$.
\item For each $\alpha \notin \hat\cI$ there exist $\hat\alpha \in \hat\cI$ and $\mu \in \RR$ such that $c_\alpha = \mu c_{\hat\alpha}$ and $M_{\beta\alpha} = \mu M_{\beta\hat\alpha} \;\; \forall \beta \in \hat\cJ$.
\end{itemize}
If~\eqref{eq:u-cond} holds for all $\alpha \in \hat\cI$ then~\eqref{eq:u-cond} holds for all $\alpha \in \cI$.
\end{lemma}

\begin{proof}
Fix $\alpha \notin \hat\cI$ and the corresponding $\hat\alpha, \mu$. We verify~\eqref{eq:u-cond}:
\begin{align*}
c_\alpha &= \mu c_{\hat\alpha} = \mu \sum_\beta u_\beta M_{\beta\hat\alpha} = \mu\sum_{\beta \in \hat\cJ} u_\beta M_{\beta\hat\alpha} = \sum_{\beta \in \hat\cJ} u_\beta M_{\beta\alpha} = \sum_{\beta} u_\beta M_{\beta\alpha}.
\end{align*}
\end{proof}

\begin{remark}
In our applications of this framework, $\cJ$ will be a set consisting of pairs $(\beta,\gamma)$ where $\beta\in \cI$ and $\gamma$ is a certain ``coloring'' of $\beta$ that encodes the signal structure. The analogous construction applies to the sets $\hat{\cI}$ and $\hat{\cJ}$. Thus, the role of $\beta$ in Lemma~\ref{lem:disconn} will be played by a pair $(\beta,\gamma)$, abbreviated as $\beta\gamma$.
\end{remark}

\subsubsection{Notation}
\label{sec:notation}

For indexing polynomials we will often use elements $\alpha \in \{0,1\}^S$ or $\alpha \in \NN^S$ for some finite set $S$, where $\NN := \{0,1,2,\ldots\}$. For such elements $\alpha,\beta$ we define $|\alpha| := \sum_{k \in S} \alpha_k$, $\alpha! := \prod_{k \in S} \alpha_k!$, and $\binom{\alpha}{\beta} := \prod_{k \in S} \binom{\alpha_k}{\beta_k}$. Also, $\beta \le \alpha$ means $\beta_k \le \alpha_k$ for all $k \in S$; and $\beta \lneq \alpha$ means $\beta \le \alpha$ and there exists $k \in S$ where $\beta_k < \alpha_k$. For a vector $X = (X_k)_{k \in S}$, we write $X^\alpha := \prod_{k \in S} (X_k)^{\alpha_k}$.

An element $\alpha \in \{0,1\}^S$ can equivalently be viewed as a subset of $S$, specifically the subset $\{k \in S \,:\, \alpha_k = 1\}$. We may abuse notation and identify $\alpha$ with this subset. For $\alpha, \beta \in \{0,1\}^S$, the meaning of $\beta \le \alpha$ is $\beta \subseteq \alpha$, and the meaning of $\beta \lneq \alpha$ is $\beta \subsetneq \alpha$. An element $\alpha \in \NN^S$ can be viewed as a multiset, containing $\alpha_k$ copies of each $k \in S$.

One common choice of $S$ will be $[n] := \{1,2,\ldots,n\}$. Another will be either $\{(i,j) \,:\, 1 \le i < j \le n\}$ or $\{(i,j) \,:\, 1 \le i \le j \le n\}$, which we will abbreviate as $\bcn$ or $\mcn$ (``multichoose'' notation) respectively. We will also abbreviate $\alpha_{(i,j)}$ as $\alpha_{ij}$ in this case. Note that $\alpha \in \{0,1\}^{\bcn}$ can be viewed as a (simple) graph on vertex set $[n]$, namely the graph that includes edge $(i,j)$ whenever $\alpha_{ij} = 1$. Similarly $\alpha \in \NN^{\mcn}$ can be viewed as a multigraph on vertex set $[n]$ (now with self-loops and parallel edges allowed), namely the multigraph with $\alpha_{ij}$ copies of edge $(i,j)$ for all $i \le j$. With this view, we will often refer to graph-theoretic properties of $\alpha$. We use $V(\alpha) \subseteq [n]$ to denote the vertex set of $\alpha$, by which we mean the set of non-isolated vertices. In other words, $i \in V(\alpha)$ if there exists some edge $(i,j)$, possibly a self-loop $(i,i)$. We call $\alpha$ \emph{connected} if every pair of vertices in $V(\alpha)$ has a path connecting them. The empty graph $\alpha = 0$ is considered to be connected. We write $\al= \bone_{(i,j)}$ for the graph consisting of a single edge between vertices $i$ and $j$. As above, we may abuse notation and identify $\alpha$ with the associated edge set, e.g., writing $\alpha \cap \beta$ for the common edges between two graphs. We may also explicitly write $E(\alpha)$ for the edge set of $\alpha$, that is, the (multi)set of pairs $(i,j)$ with $i \le j$ corresponding to the edges of $\alpha$ (where parallel edges are included multiple times based on their multiplicity).

\subsubsection{Example: Planted Submatrix}
\label{sec:ex-subm}

We now sketch how to apply the above framework to planted submatrix, and how this differs from the prior work~\cite{SW-estimation}. We include here the main steps, with the full details deferred to Section~\ref{sec:subm}. To recap the setup, we have parameters $\lambda \ge 0$ and $\rho \in [0,1]$, and the observed data is $(Y_{ij})_{1 \le i \le j \le n}$ generated as follows: $Y = X + Z$ where $X_{ij} = \lambda \theta_i \theta_j$, $\theta \in \{0,1\}^n$ is i.i.d.\ $\Ber(\rho)$, and $\{Z_{ij}\}_{i \le j}$ are i.i.d.\ $\cN(0,1)$. The goal is to estimate $x := \theta_1$.

To construct orthogonal polynomials, we will make use of the multivariate Hermite polynomials $\{H_\alpha\}$ for $\alpha \in \NN^{\mcn}$ (see~\cite{sze-book}). These are well known to be orthogonal with respect to Gaussian measure, and we use the normalization for which they are orthonormal: $\EE[H_\alpha(Z) \cdot H_\beta(Z)] = \one_{\alpha=\beta}$. As our basis for $\RR[Y]_{\le D}$, choose $\phi_\alpha(Y) := H_\alpha(Y)$ for $\alpha \in \cI := \{\alpha \in \NN^{\mcn} \,:\, |\alpha| \le D\}$. For our orthonormal polynomials in the underlying independent random variables $W = (Z,\theta)$, choose
\[ \psi_{\beta\gamma}(Z,\theta) = H_\beta(Z) \left(\frac{\theta-\rho}{\sqrt{\rho(1-\rho)}}\right)^\gamma \]
for $\beta\gamma \in \cJ := \{(\beta,\gamma) \,:\, \beta \in \NN^{\mcn}, \, |\beta| \le D, \, \gamma \in \{0,1\}^n\}$. We will often view $\gamma$ as a subset $\gamma \subseteq [n]$ as described above. Note that $\{\psi_{\beta\gamma}\}$ is orthonormal as required. With this setup, we will compute (with details deferred to Section~\ref{sec:subm})
\[ c_\alpha = \frac{\lambda^{|\alpha|} \rho^{|V(\alpha) \cup \{1\}|}}{\sqrt{\alpha!}} \qquad \text{and} \qquad M_{\beta\gamma,\alpha} = \one_{\beta \le \alpha} \cdot \one_{\gamma \subseteq V(\alpha-\beta)} \cdot \sqrt{\frac{\beta!}{\alpha!}} \binom{\alpha}{\beta} \lambda^{|\alpha - \beta|} \rho^{|V(\alpha-\beta)|} \left(\frac{1-\rho}{\rho}\right)^{|\gamma|/2}. \]
Towards applying Lemma~\ref{lem:disconn}, we define $\hat\cI$ and $\hat\cJ$ as follows.
\begin{itemize}
    \item By convention, $0 \in \hat\cI$. For $|\alpha| \ge 1$ we include $\alpha$ in $\hat\cI$ if and only if $\alpha$ (when viewed as a multigraph) is connected with $1 \in V(\alpha)$.
    \item We include $\beta\gamma$ in $\hat\cJ$ if and only if $\beta \in \hat\cI$ and $\gamma \subseteq V(\beta) \cup \{1\}$. (The union with $\{1\}$ only matters in the case $\beta = 0$.)
\end{itemize}
\noindent Later, we will check the conditions in Lemma~\ref{lem:disconn}, allowing the ``bad'' terms $\alpha \notin \hat\cI$ and $\beta \notin \hat\cJ$ to be disregarded. This step is analogous to~\cite{SW-estimation}, where disconnected graphs were similarly removed from consideration. Now, combining Proposition~\ref{prop:master} and Lemma~\ref{lem:disconn}, our goal is to choose values $(u_{\beta\gamma})_{\beta\gamma \in \hat\cJ}$ and verify
\[ \sum_{\beta\gamma \in \hat\cJ} u_{\beta\gamma} M_{\beta\gamma,\alpha} = c_\alpha \qquad \forall \alpha \in \hat\cI. \]
Then our final bound will be
\[ \Corr_{\le D} \le \frac{\|u\|}{\sqrt{\EE[x^2]}} = \rho^{-1/2} \sqrt{\sum_{\beta\gamma \in \hat\cJ} u_{\beta\gamma}^2} \]
since we are choosing $u_{\beta\gamma} = 0$ for $\beta\gamma \notin \hat\cJ$.

Using the support structure of $M$ and the fact $M_{\alpha 0,\alpha} = 1$, our constraints can be written as
\begin{equation}\label{eq:u-rec-intro}
u_{\alpha 0} = c_\alpha - \sum_{\beta \lneq \alpha} \left( \sum_{\gamma \subseteq V(\alpha-\beta)} u_{\beta\gamma} M_{\beta\gamma,\alpha}\right) \qquad \forall \alpha \in \hat\cI.
\end{equation}
For all $\alpha\gamma \in \hat\cJ$ with $\gamma \ne 0$, we refer to the values $u_{\alpha\gamma}$ as \emph{free variables}. We are free to choose values for these variables and then the remaining values $u_{\alpha 0}$ are determined by the recurrence~\eqref{eq:u-rec-intro}. Setting all free variables to zero recovers the existing method of~\cite{SW-estimation}. This amounts to using only the orthonormal polynomials $\psi_\beta(Z) = H_\beta(Z)$, which means only the ``noise'' is being used to lower-bound $\EE[f(Y)^2]$. We will improve on this by using both the signal and noise. Specifically, we will set the free variables in such a way to zero out the term in parentheses in~\eqref{eq:u-rec-intro}. As a result, we will have simply $u_{\alpha 0} = c_\alpha$ for all $\alpha \in \hat\cI$. This avoids some difficult-to-control recursive blowup of the $u$ values that is present in~\cite{SW-estimation}. Our specific construction for $u$ is
\[ u_{\alpha\gamma} = \left(-\sqrt{\frac{\rho}{1-\rho}}\right)^{|\gamma|} \cdot c_\alpha \qquad \forall \alpha\gamma \in \hat\cJ. \]
The remaining details, including the computation of $\|u\|$, are deferred to Section~\ref{sec:subm}. The dominant contributions to $\|u\|$ come from $\alpha$ that (viewed as a multigraph) are trees containing vertex $1$.

\subsubsection{Other Models}
\label{subsec:other}
We now summarize how the arguments change for the other models we consider. For the planted dense subgraph problem, the choice for the orthonormal set $\{\psi_{\beta\gamma}\}$ is more subtle. However, we ultimately manage to use a similar construction for $u$ as above. In particular, the set of ``good'' terms and ``bad'' terms are the same as those used for the planted submatrix problem, besides the consideration of simple graphs rather than multigraphs.

The proofs for the spiked Wigner model and SBM are more involved. A key difference lies in the definition of the ``good'' terms. This difference stems, in part, from the dependence of the estimand $x$ on both vertices 1 and 2. Specifically, ``good'' $\al$ now satisfy $1,2\in V(\al)$ and every $v\in V(\al)\setminus \{1,2\}$ has degree at least $2$ (see Definition~\ref{def:good:SBM}). The intuition behind the last condition is that $\al$ that contains a leaf node which is not $1,2$ is ``uninformative'' for the estimation of $x$.

Furthermore, we cannot obtain a closed-form solution for $u$ as in the prior cases.    Consequently, the values of $u=(u_{\be \ga})$ are defined implicitly through a recursive formula: 
$(u_{\al\ga})_{\ga}$ 
is constructed based on the values of $(u_{\be\ga})_{\ga}$ for $|\be|<|\al|$. The construction is designed to minimize $\sum_{\ga} u_{\al\ga}^2$ subject to a linear constraint at each step. We refer to Section~\ref{subsec:SBM:u} for the precise construction. The key aspect of the analysis involves controlling their growth. Crucially, we show that $\sum_{\ga} u_{\alpha\ga}^2$ does not grow exponentially in $|\alpha|$ (as in the bounds from~\cite{SW-estimation}) but exponentially in $|\alpha|-|V(\alpha)|+1$ (which represents edges in excess of a tree). This turns out to be crucial to capture the sharp thresholds that we aim for at degree $n^\delta$.

For planted submatrix and planted dense subgraph, the dominant contributions to $\|u\|$ come from trees containing vertex $1$, where vertex $1$ is distinguished because we are estimating $\theta_1$. The importance of trees here appears sensible, given that the optimal algorithm is AMP, which can essentially be expressed as a sum of such ``tree-structured'' polynomials~\cite{amp-univ,MW-amp,sos-amp-robustly,fourier-iter,alg-universality}. For spiked Wigner and SBM, there are two distinguished vertices $1,2$ since we will be estimating pairwise quantities, and the dominant contributions to $\|u\|$ come from simple paths from vertex $1$ to $2$. This reflects the fact that the optimal algorithms are based on (self-avoiding or non-backtracking) paths~\cite{spectral-redemption,massoulie,nb-spectrum,HS-bayesian}.

\section{Main Results}
\label{sec:main-results}

\subsection{Correlation and Weak Recovery}
\label{sec:corr}

In all the models we consider, the observation will be called $Y$ and the scalar quantity to be estimated will be called $x$. We will state our results in terms of $\Corr_{\le D}$, as defined in Section~\ref{sec:overview-proof}:
\[ \Corr_{\le D} := \sup_{f \in \RR[Y]_{\le D}} \frac{\EE[f(Y) \cdot x]}{\sqrt{\EE[f(Y)^2] \cdot \EE[x^2]}}=\sup_{f \in \RR[Y]_{\le D}} \frac{\big|\EE[f(Y) \cdot x]\big|}{\sqrt{\EE[f(Y)^2] \cdot \EE[x^2]}} \in [0,1]. \]
To show low-degree hardness of estimation, our goal will be to rule out \emph{weak recovery}, defined as $\Corr_{\le D} = \Omega(1)$. That is, we aim to prove $\Corr_{\le D} = o(1)$ for $D=n^{\Omega(1)}$, which by Fact~\ref{fact} implies
\[ \MMSE_{\le D} \ge (1-o(1)) \EE[x^2] \ge (1-o(1)) \Var(x) = (1-o(1)) \MMSE_{\le 0}, \]
meaning degree-$D$ polynomials have no significant advantage over the trivial estimator $f(Y) \equiv \EE[x]$.

For our low-degree upper bounds, we will assume a regime where $\EE[x]^2=o(\EE[x^2])$, implying $\MMSE_{\le 0} = (1-o(1)) \EE[x^2]$. We will aim to establish either \emph{strong recovery}, meaning $\Corr_{\le D} = 1-o(1)$ or equivalently $\MMSE_{\le D} = o(\MMSE_{\le 0})$, or \emph{weak recovery}, meaning $\Corr_{\le D} = \Omega(1)$ or equivalently $\MMSE_{\le D} = (1-\Omega(1)) \MMSE_{\le 0}$.

\subsection{Planted Submatrix}
\label{sec:main-subm}

\begin{definition}[Planted Submatrix Model]
For parameters $\lambda \ge 0$ and $\rho \in [0,1]$, observe the $n\times n$ matrix $Y=\la \theta \theta^{\top}+Z$ where $\theta\in \{0,1\}^n$ is i.i.d.\ $\Ber(\rho)$, and $Z$ has entries $Z_{ij}=Z_{ji}\sim\cN(0,1)$ where $\{Z_{ij}\}$ are independent. The goal is to estimate $x := \theta_1$.
\end{definition}

\begin{theorem}\label{thm:main-subm}
Consider the planted submatrix model. For any constant $\eps > 0$ there exists a constant $C \equiv C(\eps) > 0$ for which the following holds.
\begin{enumerate}
\item[(a)] (Lower Bound) If
\[ \lambda \le (1-\eps)(\rho\sqrt{en})^{-1}\sqrt{1-\rho} \qquad \text{and} \qquad D \le \lambda^{-2}/C \]
then
\[ \Corr_{\le D} \le C \sqrt{\frac{\rho}{1-\rho}}. \]
\item[(b)] (Upper Bound) If
\[ \lambda \ge (1+\eps)(\rho\sqrt{en})^{-1}, \quad \rho = \omega(n^{-1} \log^{12} n), \quad \rho = o(\log^{-12} n) \]
then
\[ \Corr_{\le C \log n} = 1-o(1) \quad \text{as} \quad n \to \infty. \]
\end{enumerate}
\end{theorem}

Provided $\rho = o(1)$, the thresholds for $\lambda$ in the lower and upper bounds match, and the lower bound shows $\Corr_{\le D} = o(1)$ as desired. It is natural to assume $\rho \gtrsim 1/n$ so that the planted submatrix is not empty, and in the upper bound we have assumed $1/n \ll \rho \ll 1$ with additional log factors, for convenience. We have not attempted to optimize the log factors. While the lower bound is non-asymptotic (holds for any $n$), the upper bound is an asymptotic result for $n \to \infty$. Note that since $\eps$ is designated as a constant, it is not allowed to depend on $n$.

As discussed in Section~\ref{sec:corr}, the lower bound implies $\MMSE_{\le D} \ge (1-o(1)) \MMSE_{\le 0}$, so that the degree-$D$ MMSE is essentially no better than the trivial MMSE, which in this case is $\MMSE_{\le 0} = \rho - \rho^2$. As pointed out in~\cite{SW-estimation}, the (scalar) $\MMSE_{\le D}$ that we are working with can be directly related to its vector analogue: using linearity of expectation and symmetry of the model,
\[ \inf_{f_1,\ldots,f_n \in \RR[Y]_{\le D}} \EE \, \frac{1}{n} \sum_{i=1}^n (f_i(Y) - \theta_i)^2 = \inf_{f \in \RR[Y]_{\le D}} \EE  (f(Y) - \theta_1)^2 = \MMSE_{\le D}. \]

\paragraph{Discussion.}

This model has been thoroughly studied. The statistical limits are well understood~\cite{BI-info,kolar-info,BIS-info}, as are the computational limits of hypothesis testing between the planted submatrix model and the ``null'' model $\lambda = 0$~\cite{MW-reduction,BBH-reduction}. Our focus is instead on the computational limits of estimation, which has a different threshold (see~\cite{SW-estimation}).

We will focus on the regime $\rho = n^{\gamma-1}$ for a constant $\gamma \in (0,1)$. When $\gamma > 1/2$, an algorithm based on \emph{approximate message passing (AMP)} achieves near-perfect estimation of $\theta$ --- meaning $o(\rho n)$ misclassification errors, with high probability --- provided $\lambda \ge (1+\eps) (\rho\sqrt{en})^{-1}$~\cite{submatrix-amp}. This sharp threshold represents the best known performance of any poly-time algorithm. Aside from a general belief that AMP algorithms are powerful, there were no convincing lower bounds reaching this sharp threshold prior to our work. Low-degree lower bounds were given in~\cite[Appendix~E]{SW-estimation} showing that polynomials of very low degree (a specific constant times $\log n$) fail to reach the sharp threshold, but the bound falls away from the threshold as the degree increases beyond that.

In the other regime, $\gamma < 1/2$, the AMP threshold can be beaten by a very simple algorithm: thresholding the diagonal entries of the observed matrix gives exact recovery of $\theta$ (with high probability) provided $\lambda \gg 1$ (where throughout we use $\gg$ to hide a $\polylog(n)$ factor). This regime does not have a sharp threshold but rather a smooth tradeoff between signal strength and runtime, with the best known algorithms achieving exact recovery in runtime $\exp(\tilde{O}(\lambda^{-2}))$ for all $\lambda$ in the range $(\rho n)^{-1/2} \ll \lambda \le \polylog(n)$ \cite{subexp-sparse,anytime-sparse}; estimation becomes statistically impossible when $\lambda \ll (\rho n)^{-1/2}$~\cite{kolar-info,BIS-info}. These algorithms also work for $\gamma \ge 1/2$. The low-degree lower bounds of~\cite{SW-estimation} rule out degree-$D$ polynomials when $\lambda \le \Omega(\min\{1,(\rho\sqrt{n})^{-1}\}/D^2)$, which captures the poly-time threshold at $\lambda \approx 1$ but not the specific runtime needed when $\lambda \ll 1$. Prior to our work, the only known lower bounds that capture this specific runtime $\exp(\tilde{O}(\lambda^{-2}))$ were of a rather different nature, ruling out certain Markov chains~\cite{ogp-sparse}.

Theorem~\ref{thm:main-subm} gives a comprehensive low-degree lower bound showing optimality of the algorithms above in a strong sense: if $\lambda$ lies below the AMP threshold, namely $\lambda \le (1-\eps) (\rho\sqrt{en})^{-1}$, then the polynomial degree required for estimation is $\Omega(\lambda^{-2})$, suggesting that runtime $\exp(\tilde\Omega(\lambda^{-2}))$ is necessary. This resolves two different open questions of~\cite{SW-estimation}: pinning down the sharp AMP threshold and also the precise degree $\Omega(\lambda^{-2})$.

The lower and upper bounds together establish an ``all-or-nothing'' phase transition for low-degree polynomials in the regime $\gamma > 1/2$: the low-degree MMSE jumps sharply from near-trivial to near-perfect at the AMP threshold. This can be viewed as a computational analogue of, e.g.,~\cite{all-nothing}. While the upper bound confirms that degree-$O(\log n)$ polynomials succeed above the AMP threshold, it remains open to show that degree-$\tilde{O}(\lambda^{-2})$ polynomials succeed below the AMP threshold for $(\rho n)^{-1/2} \ll \lambda \ll 1$. Success of degree-$O(\log n)$ polynomials for $\lambda \gg 1$ is established by~\cite{SW-estimation}.

The case $\rho = \Theta(1)$ has also been studied~\cite{DM-sparse-amp,LKZ-sparse,submatrix-ogp,MW-amp}. In this case the relevant scaling for $\lambda$ is $\lambda = c/\sqrt{n}$ for a constant $c$. The best known algorithm is again based on AMP~\cite{DM-sparse-amp,LKZ-sparse}, but here its MSE converges to some nontrivial constant depending on $c,\rho$ in contrast to the all-or-nothing behavior above. It has been shown that constant-degree polynomials cannot surpass the precise MSE achieved by AMP~\cite{MW-amp}, and extending this result to higher degree remains an interesting open question. The bound that we prove here applies to higher degree polynomials but does not reach the optimal MSE value. In fact, our bound $\Corr_{\le D} \le C\sqrt{\rho/(1-\rho)}$ becomes vacuous in the regime $\rho = \Theta(1)$ unless $\rho$ is a very small constant.

\subsection{Planted Dense Subgraph}

\begin{definition}[Planted Dense Subgraph Model]
For parameters $\rho \in [0,1]$ and $p_0,p_1\in [0,1]$, we observe $Y=(Y_{ij})\in \{0,1\}^{\bcn}$ generated as follows. 
\begin{itemize}
    \item A planted signal $\theta=(\theta_i)_{1\leq i\leq n}\in \{0,1\}^n$ is drawn with i.i.d.\ $\Ber(\rho)$ entries.
    \item Conditioned on $\theta$, $Y_{ij}\sim \Ber(p_0+(p_1-p_0)\theta_i \theta_j)$ is sampled independently for each $i<j$. 
\end{itemize}
The goal is to estimate $x := \theta_1$.
\end{definition}

\noindent Without loss of generality we will assume $p_0\leq p_1$, since otherwise one can consider the complement graph instead. The main result is as follows. 
\begin{theorem}\label{thm:main-pds}
Consider the planted dense subgraph model with $0<p_0 \le p_1 \le 1$, and define
\begin{equation}\label{eq:lambda-sub}
\lambda := \frac{p_1-p_0}{\sqrt{p_0(1-p_0)}}.
\end{equation}
For any constant $\eps>0$, there exists a constant $C\equiv C(\eps)>0$ for which the following holds.
\begin{enumerate}
\item[(a)] (Lower Bound) If
\[ \lambda \le (1-\eps)(\rho\sqrt{en})^{-1}\sqrt{1-\rho} \qquad \text{and} \qquad D \le \lambda^{-2}/C \]
then
\[ \Corr_{\le D} \le C \sqrt{\frac{\rho}{1-\rho}}. \]
\item[(b)] (Upper Bound) If
\[ \lambda \ge (1+\eps)(\rho\sqrt{en})^{-1}, \quad \rho = \omega(n^{-1} \log^{12} n), \quad \rho = o(\log^{-12} n), \quad p_0 = \omega(n^{-1} \log^{24} n) \]
then
\[ \Corr_{\le C \log n} = 1-o(1) \quad \text{as} \quad n \to \infty. \]
\end{enumerate}
\end{theorem}

\noindent Note that this matches our result for planted submatrix (Theorem~\ref{thm:main-subm}) with the substitution~\eqref{eq:lambda-sub}, and with an additional condition $p_0 \gg 1/n$ in the upper bound, which precludes isolated vertices. (Again, we have not attempted to optimize the log factors.)

\paragraph{Discussion.}

Compared to planted submatrix, the planted dense subgraph problem has a wider variety of different regimes with different behaviors, depending on the scaling of $p_0,p_1$. Many statistical results exist for this model~\cite{AV-info,VA-info,CX-info}, as well as computational lower bounds for testing~\cite{HWX-reduction}, and positive algorithmic results~\cite{log-density,ames-convex,CX-info,one-community-sparse}. Computational limits for estimation were considered by~\cite{SW-estimation}, and our result provides a sharper refinement.

To simplify the discussion, we focus for now on the regime considered by~\cite{HWX-reduction} where $1/n^2 \ll p_0 \ll 1$ and $p_1 = c p_0$ for a constant $c > 1$, and furthermore we restrict our attention to the case $1/\sqrt{n} \ll \rho \ll 1$ which is most relevant for us due to the detection-recovery gap (see~\cite[Conjecture~2.2]{BBH-reduction}). Under these assumptions, the results of~\cite{SW-estimation} already give matching upper and lower bounds for low-degree polynomials on a coarse scale, namely, the threshold occurs at $n \rho^2 p_0 = \tilde\Theta(1)$. This is achieved by a very simple algorithm that selects vertices of large degree. Our result refines this, showing a low-degree phase transition at the sharp threshold $n \rho^2 p_0 = [e(c-1)^2]^{-1}$. To our knowledge, this threshold has not appeared before in the literature, except in one specific scaling regime considered by~\cite{one-community-sparse} where $p_0,p_1 = \Theta(1/n)$, $\rho = \Theta(1)$ and $\rho\to 0$ \emph{after} $n\to\infty$. Notably, a matching poly-time algorithm that achieves this threshold has not appeared in the literature (outside the specific regime of~\cite{one-community-sparse}). Our upper bound implies an algorithm of quasi-polynomial runtime $n^{O(\log n)}$, by directly evaluating the polynomial term-by-term. We expect it should be possible to give a poly-time algorithm by approximately evaluating this ``tree-structured'' polynomial via the ``color coding'' trick~\cite{color-coding,HS-bayesian}, which has been used by~\cite{graph-match-trees,graph-match-otter} to approximate tree-structured polynomials. We also expect that a more practical algorithm should be possible using an AMP approach similar to~\cite{clique-e,submatrix-amp}, and we leave this for future work.

Another aspect in which our result improves on~\cite{SW-estimation} is in the bound on $D$, namely, in the above scaling regime of~\cite{HWX-reduction} we show that degree $D = \Omega(1/p_0)$ is necessary below the sharp threshold. We expect this is essentially optimal: while it has not appeared in the literature, an algorithm of runtime $\exp(\tilde{O}(1/p_0))$ can be obtained for $(\rho n)^{-1} \ll p_0 \ll 1$ by a simple adaptation of the spiked Wigner results in~\cite{subexp-sparse}. The algorithm is based on a brute-force search over size-$\ell$ subsets of vertices, with $\ell = \tilde\Theta(1/p_0)$.

We will discuss a few other popular scaling regimes. While not a focus of our work, a notable special case of the planted dense subgraph model is the well-known \emph{planted clique} problem, where $p_1 = 1$ and $p_0 = 1/2$. Here, the testing and estimation problems coincide in difficulty, both transitioning from hard to easy at $\rho \approx 1/\sqrt{n}$. More precisely, there is an AMP algorithm for exact recovery above $\rho = 1/\sqrt{en}$~\cite{clique-e}, but there are in fact better poly-time algorithms that reach $\rho = \eps/\sqrt{n}$ for an arbitrary constant $\eps > 0$~\cite{alon-clique}; so there is not actually a sharp threshold here but rather a smooth tradeoff between runtime and clique size. Our result, as it should, does not indicate a sharp threshold here, as it becomes limited to constant degree in the planted clique regime. Stronger low-degree lower bounds for planted clique estimation were known previously, reaching the ``correct'' degree $D \approx \log^2 n$~\cite{SW-estimation}, and the analogous result for testing was known even earlier (see~\cite{hopkins-thesis}). However, we note that planting a clique with $1/\sqrt{n}\ll \rho\ll 1$ in a dense graph with $p_0=1-c/(n\rho^2)$, or equivalently planting an independent set of size $n\rho$ with $p_0=c/(n\rho^2)$, exhibits a sharp threshold at $c=1/e$ for estimation by low-degree polynomials, as we establish in Theorem~\ref{thm:main-pds}. To our knowledge, this sharp phase transition has not appeared in the previous literature.

Another interesting regime of planted dense subgraph that we do not focus on is the \emph{log-density} regime where $\rho \ll 1/\sqrt{n}$ and $p_1 = n^{-\alpha}$, $p_0 = n^{-\beta}$ for constants $0 < \alpha < \beta$~\cite{log-density}. We do not expect a sharp threshold here, and the low-degree limits for testing have been characterized using a somewhat delicate conditioning argument~\cite{subhyper}. The known algorithms for estimation match the lower bounds for testing (which is an easier problem), so we expect the limits of estimation and testing to coincide. Still, proving optimal bounds on $\MMSE_{\le D}$ remains a difficult technical challenge, which we leave for future work.

\subsection{Spiked Wigner Model}

\begin{definition}[Spiked Wigner Model]
For a parameter $\lambda \ge 0$, let
\[ X = \sqrt{\frac{\lambda}{n}} UU^\top \]
where $U \in \RR^{n \times m}$ with entries i.i.d.\ from some prior $\pi$. Observe
\[ Y_{ij} = X_{ij} + Z_{ij} \qquad \text{for} \qquad 1 \le i \le j \le n, \]
where $Z_{ij}$ are i.i.d.\ $\cN(0,1)$. The goal is to estimate $x := X_{1,2}$.
\end{definition}

\noindent The estimand $X_{1,2}$ is representative of any off-diagonal entry of $X$, by symmetry, while the diagonal entries have negligible contribution to the matrix mean squared error $\|\hat{X} - X\|_F^2$. It is common to put $\cN(0,2)$ instead of $\cN(0,1)$ on the diagonal of $Z$, which would only make the problem harder, so our lower bound still holds (see~\cite[Claim~A.2]{SW-estimation}). The upper bound below also remains valid because it does not use the diagonal entries.

\begin{theorem}\label{thm:main-wigner}
Consider the spiked Wigner model and assume the prior $\pi$ satisfies $\EE[\pi] = 0$ and $\EE[\pi^2] = 1$.
\begin{enumerate}
\item[(a)] (Lower Bound) Suppose that for some $c,\nu>0$, $\EE |\pi|^k\leq (ck)^{\nu k}$ for any integer $k\geq 1$. There exist constants $\delta,C>0$ depending only on $c, \nu$ such that if $D \le n^\delta$ then
\[ \Corr_{\le D} \le C \sqrt{\frac{m}{n}\sum_{d=1}^{D}\la^d}.
\]
Consequently,
\begin{itemize}
\item if $\la<1$ is fixed and $D\leq n^{\delta}$, then $\Corr_{\le D}\leq C\sqrt{\frac{m\la}{n(1-\la)}}$,
\item if $\la=1$ and $D \le n^\delta$, then $\Corr_{\leq D} \leq C\sqrt{Dm/n}$, and
\item if $\la=O(1)$, $m \le n^{1-\Omega(1)}$, and $D = o(\log n)$, then $\Corr_{\leq D}=o(1)$.
\end{itemize}

\item[(b)] (Upper Bound) Suppose that $K\equiv \EE[\pi^4]<\infty$ and $\lambda \ge 1+\eta$, where $K$ and $\eta \in (0,1)$ are constants (not depending on $n$). Then there exist a constant $c \equiv c(K) > 0$ and a universal constant $C > 0$ such that if $m=m_n$ satisfies $m = o(n)$, then for $D\geq \frac{C}{\eta}\log (n/m)$ and large enough $n$,
\[ \Corr_{\le D} \ge c\eta. \]
\end{enumerate}
\end{theorem}

\paragraph{Discussion.}

Most prior work pertains to the case $m=O(1)$ (and often $m=1$), and we focus on this case for now. We also assume that $\lambda$ and the spike prior $\pi$ are fixed, i.e., not depending on $n$. The spiked Wigner model was first studied in random matrix theory, where a sharp phase transition at $\lambda=1$ was discovered in the eigenvalues and eigenvectors of $Y$~\cite{peche,FP-wigner,maida,CDF-wigner,BN-eigenvec}. This is known as the ``BBP'' transition, by analogy to the similar phase transition discovered by Baik, Ben Arous, and P{\'e}ch{\'e} in the related \emph{spiked Wishart} model~\cite{BBP}. These results immediately imply an algorithm, based on the $m$ leading eigenvectors, for weak recovery when $\lambda > 1$. Here, weak recovery means $\langle \hat{X},X \rangle/(\|\hat{X}\|_F \|X\|_F)$ converges to some positive constant depending on $\lambda$, where the estimator $\hat{X}$ is obtained from $Y$ by truncating the eigendecomposition to the $m$ leading terms. When $\lambda \le 1$, the $m$ leading eigenvectors fail to achieve weak recovery.

Later, a number of works asked the statistical question of whether \emph{any} method can estimate the signal below the BBP threshold~\cite{FR-amp,MMSE,proof-replica,LM-wigner,short-replica} (see~\cite{miolane-survey} for a survey), or even detect its presence~\cite{MRZ,BMVVX,opt-subopt,fund-limits-wigner}. Statistically speaking, the answer turns out to be ``yes'' for some spike priors $\pi$, including sufficiently sparse priors but not $\cN(0,1)$ or $\Unif(\{\pm 1\})$. However, regardless of the prior (as long as it does not depend on $n$), no \emph{poly-time} algorithm is known to achieve weak recovery below $\lambda = 1$. It has been conjectured that this hardness is inherent, on the basis of statistical physics methods, namely AMP and the associated free energy barriers~\cite{MMSE} (see Conjecture~10 in~\cite{LM-wigner}). Low-degree lower bounds show hardness for hypothesis testing against the ``null'' model ($\lambda = 0$) whenever $\lambda < 1$~\cite{ld-notes,spectral-planting}, providing indirect evidence that the seemingly more difficult estimation task should also be hard. We note that our focus is on weak recovery because strong recovery, meaning $\langle \hat{X},X \rangle/(\|\hat{X}\|_F \|X\|_F) \to 1$, is information-theoretically impossible; see~\cite{LM-wigner,pbm24}.

Our lower bound shows low-degree hardness of weak recovery when $\lambda \le 1$, notably covering the critical case $\lambda = 1$. This holds for all priors $\pi$ satisfying a mild moment condition (which, for instance, covers all priors with bounded support). We rule out polynomials of degree $n^\delta$ for a particular constant $\delta > 0$ depending only on $c,\nu$. This resolves the conjecture mentioned above, within the low-degree framework. The testing results~\cite{spectral-planting} suggest that this hardness can be extended to any $\delta < 1$, and this remains an open problem.

Another consequence of our lower bound is that logarithmic degree is required, even in the easy regime. Combined with the upper bound, we know that when $\lambda > 1$ is fixed and $m \le n^{1-\Omega(1)}$, the degree complexity of weak recovery is exactly on the order $D = \Theta(\log n)$.

Our results also extend to the case of growing $m$ (a.k.a.\ symmetric matrix factorization), which has received recent attention~\cite{meso,bun-rotational,extensive-rank,pbm24,semerjian,multiscale-cavity}. We show that the above phenomenon persists, that is, the low-degree threshold for weak recovery remains at $\lambda = 1$ as long as $m \ll n$. To our knowledge, no prior work has explored computational hardness in the growing-$m$ regime, so our results uncover a new phase transition that was not previously substantiated. As $m$ grows, the testing and estimation thresholds separate, so this new phenomenon could not have been uncovered by studying the testing problem.

Our low-degree upper bound is an extension of~\cite{HS-bayesian}, which handles the $m=1$ case using self-avoiding walks. We expect that our estimator can be made into a poly-time algorithm using the ``color coding'' trick~\cite{color-coding}, as in~\cite{HS-bayesian}. Alternatively, poly-time weak recovery based on eigenvectors might be deduced from the spectral analysis by~\cite{meso}. More sophisticated poly-time algorithms that aim to optimize the precise mean squared error are discussed in~\cite{bun-rotational,pbm24,semerjian}.

Our lower bound becomes vacuous once $m = \Omega(n)$, and for good reason: in this regime, the degree-1 estimator $f(Y)=Y_{1,2}$ achieves correlation $\sqrt{\frac{m\la/n}{m\la/n+1}}$, which gives weak recovery as long as $\la$ is of constant order.

\subsection{Stochastic Block Model}
\label{sec:main-sbm}

\begin{definition}[Stochastic Block Model]
\label{def:sbm}
Let $q\geq 2$ be the number of communities. Let $\pi=(\pi_k)_{k\in [q]}\in \R_{> 0}^q$ be a vector whose entries sum to 1, representing a probability distribution over $[q]$. Let $Q\in \R_{>0}^{q\times q}$ be a symmetric matrix with positive entries. Observe $Y \in \{0,1\}^{\bcn}$, the adjacency matrix for a graph, generated as follows. The community labels $\sigma^\star=(\sigma^\star_i)_{1\le i\le n}\in [q]^n$ for the $n$ vertices are drawn as $\sigma^\star_i\iid \pi$. Conditional on $\sigma^\star$, $Y_{ij}\sim \Ber(Q_{\sigma^\star_i,\sigma^\star_j}/n)$ is sampled independently for each $i<j$. The goal is to estimate $x := Q_{\sigma^\star_1,\sigma^\star_2}-\EE[Q_{\sigma^\star_1,\sigma^\star_2}]$.
\end{definition}

\noindent Equivalently, the probability of having an edge between $i$ and $j$, conditioned on $\sigma^\star$, is $Q_{\sigma_i^\star,\sigma_j^\star}/n$. We will mostly focus on the sparse regime with a fixed number of communities, i.e.\ regard $q, \pi,Q$ as fixed and consider the limit $n\to\infty$, although our proof for the low-degree lower bound applies more generally, e.g.\ when $q\ll n^{1/8}$ in the \emph{symmetric} SBM where $\pi_k\equiv 1/q$ and the diagonal (resp.\ off-diagonal) entries of $Q$ are the same (see Remark~\ref{rmk:general:SBM}). In the sparse regime, we cannot hope to achieve strong recovery, even information-theoretically, due to the presence of isolated vertices, and thus we focus on weak recovery.

We are interested in recovering the membership matrix $(Q_{\sigma^\star_i,\sigma^\star_j})_{1 \le i < j \le n}$. We have chosen $x$ so that $\Corr_{\le D}$ is directly related to the MMSE
\begin{equation}\label{eq:sbm:mmse:edge}
\MMSE_{\le D} = \inf_{f \in \RR[Y]_{\leq D}} \EE[(f(Y)-Q_{\sigma^\star_1,\sigma^\star_2})^2],
\end{equation}
as described in Section~\ref{sec:corr}. Here, note that the centering term $\EE[Q_{\sigma^\star_1,\sigma^\star_2}]$ that appears in $x$ has been omitted from~\eqref{eq:sbm:mmse:edge} because it does not affect $\MMSE_{\le D}$. Note that in the \emph{symmetric} SBM, estimating $Q_{\sigma^\star_1,\sigma^\star_2}$ is equivalent to estimating $\one_{\sigma^\star_1=\sigma^\star_2}$. In fact, our lower bound also rules out estimation of $(\one_{\sigma^\star_1=k}-\pi_k)(\one_{\sigma^\star_2=\ell}-\pi_\ell)$ for \emph{any} $k,\ell\in[q]$; see Theorem~\ref{thm:sbm}.

Define $d := \E[Q_{\sigma^\star_1,\sigma^\star_2}]>0$, which is (asymptotically) the average degree of the observed graph. Assume that the average degree of each vertex is the same regardless of its community label:
\begin{equation}\label{eq:degree:condition}
\sum_{\ell=1}^{q}Q_{k,\ell}\pi_\ell=d \qquad\forall k \in [q].
\end{equation}
Violation of condition~\eqref{eq:degree:condition} allows for weak recovery, regardless of $q,\pi, Q$ as long as these are fixed as $n\to\infty$, via simple degree counting (see e.g.~\cite[Proposition~4.1]{mss-weak}).  Thus, we may assume the condition~\eqref{eq:degree:condition} without loss of generality (see also~\cite{BMNN, AS-acyclic, HS-bayesian}). Define the stochastic matrix
\begin{equation}\label{eq:T}
    T:=\frac{1}{d}\diag(\pi)Q
\end{equation}
and let $1=\la_1(T)\geq |\la_2(T)|\geq \ldots\geq |\lambda_q(T)|$ denote the eigenvalues of $T$ in decreasing order of magnitude. A central role will be played by the parameter
\begin{equation*}
    \la := |\la_2(T)|.
\end{equation*}
\begin{theorem}\label{thm:main-sbm}
    Consider the stochastic block model with parameters $q,\pi,Q$ such that Eq.~\eqref{eq:degree:condition} holds.
    \begin{enumerate}
        \item[(a)] (Lower Bound) There exist constants $\delta,C>0$ depending only on $q,\pi,Q$ such that if $D\leq n^\delta$ then
        \[
        \Corr_{\leq D} \leq \sqrt{\frac{C}{n}\sum_{t=1}^{D}(d\la^2)^{t}}.
        \]
        Consequently, if $d\la^2 \le 1$ and $D \le n^\delta$, then $\Corr_{\le D} = o(1)$.
        \item[(b)] (Upper Bound~\cite{AS-acyclic,HS-bayesian}) If $q,\pi,Q$ are fixed with $d\la^2>1$, then for large enough $n$ we have $\Corr_{\leq C \log n} \ge \eta$ for some constants $C \equiv C(q,\pi,Q)>0$ and $\eta \equiv \eta(q,\pi,Q)>0$.
    \end{enumerate}
\end{theorem}
\begin{remark}
\label{rmk:general:SBM}
    Although Theorem~\ref{thm:main-sbm} is stated only for the sparse SBM with a constant number of communities, our proof reveals that if Condition~\eqref{eq:degree:condition} is satisfied, the low-degree lower bound remains valid when $\|Q\|_{\infty}\leq n^{1-\eps}$ and $\left(\frac{q}{\pi_{\min}}\right)^4\frac{d}{Q_{\min}}\leq n^{1-\eps}$ for a constant $\eps>0$
    where $\pi_{\min}:=\min_{k \in [q]}\pi_k$ and $Q_{\min}:=\min_{k,\ell\in [q]} Q_{k,\ell}$. Namely, if $d\la^2\leq 1-\eta$ for a constant $\eta>0$ then $\Corr_{\leq n^{\delta}}=o(1)$ holds for some $\delta=\delta(\eps,\eta)>0$.
\end{remark}

Note that in Definition~\ref{def:sbm}, we assumed $\pi\in \R_{> 0}^{q}$ and $Q\in \R_{>0}^{q\times q}$. We may assume the former without loss of generality since otherwise we can remove the empty communities. On the other hand, the latter condition that the connectivity matrix $Q$ has positive entries is likely a proof artifact, and we leave the question of resolving the low-degree hardness when some entries of $Q$ are $0$ as an open problem.

The upper bound (b) for $d\la^2>1$ follows from the upper bounds derived in the works~\cite{AS-acyclic,HS-bayesian}, which we include for completeness. The main contribution of this work is the lower bound (a) for $d\la^2\leq 1$.

\paragraph{Discussion.}
The stochastic block model (SBM) is a special case of 
inhomogeneous random graphs~\cite{BoJaRi:07} that has been extensively studied as a model for communities in statistics and social sciences, see e.g.~\cite{HoLaLe:83, SnijdersNowicki:97, BC:09, RCY:11}, and for analyzing clustering algorithms in computer science, see e.g.~\cite{DyerFrieze:89,JerrumSorkin:98,CondonKarp:01,McSherry:01,CojaOghlan:10}. See \cite{moore-survey-sbm,abbe-survey-sbm} for survey articles. 

Our focus is on the sparse regime where the edge probabilities are proportional to $1/n$, the number of communities $q$ is held constant, and the objective is weak recovery. In the sparse regime, the landmark work~\cite{decelle} first predicted a sharp computational phase transition at the so-called \emph{Kesten--Stigum (KS) threshold} $d\la^2=1$ based on a heuristic analysis of the \emph{belief propagation (BP)} algorithm. First identified by Kesten and Stigum \cite{KestenStigum:66} in the context of multi-type branching processes, the KS threshold $d\la^2=1$ has since played an important role in other areas, including phylogenetic reconstructions~\cite{DaMoRo:11,MoRoSl:11,RochSly:17}.

To simplify the discussion, we focus on the symmetric SBM where the diagonal (resp.\ off-diagonal) entries of $Q$ are the same. A sequence of works has established that poly-time algorithms achieve weak recovery above the KS threshold $d\la^2>1$ for $q=2$~\cite{massoulie,mns-alg,nb-spectrum} and for general $q\geq 3$~\cite{AS-acyclic}. Below the KS threshold $d\la^2\leq 1$, weak recovery is information-theoretically impossible for $q=2$ \cite{mns-impossible} or $q=3,4$ and $d$ large enough \cite{exact-phase-sbm}. For $q \ge 5$, a statistical-computational gap appears: no known poly-time algorithm succeeds below the KS threshold, yet it is information-theoretically possible to do so~\cite{AS-crossing,BMNN, CKPZ:18}. However, regardless of the prior $\pi$ or the probability matrix $Q$, it has been conjectured by~\cite{AS-acyclic} based on the prediction by~\cite{decelle} that no poly-time algorithm can achieve weak recovery below $d\la^2<1$. Low-degree lower bounds for hypothesis testing support the presumed hardness below the KS threshold~\cite{HS-bayesian,spectral-planting}, and the concurrent work~\cite{sbm-reduction} makes this precise with a detection-to-recovery reduction. However, proving bounds on $\MMSE_{\le D}$ remained a difficult technical challenge that was posed as an open question by Hopkins and Steurer in one of the first papers on the low-degree framework~\cite{HS-bayesian}. We resolve this, proving that degree-$n^\delta$ polynomials fail to achieve weak recovery below the sharp KS threshold $d\la^2\leq 1$. We show this for a particular constant $\delta > 0$ but we expect the result to hold for any constant $\delta < 1$, as suggested by~\cite{sbm-reduction}, and we leave this as an open problem.

Prior to our work, coarser low-degree lower bounds for estimation were obtained by~\cite{graphon}, who additionally studied the case of a growing number of communities and the related task of graphon estimation. Our proof of Theorem~\ref{thm:main-sbm}(a) reveals that, in the symmetric SBM where the average degree $d$ is of order constant, community detection is low-degree hard below the KS threshold as long as $q \ll n^{1/8}$ (see Remark~\ref{rmk:general:SBM}). The follow-up work~\cite{CMSW25} extends our techniques to show that the KS bound remains the threshold for low-degree recovery as long as $q\ll n^{1/2}$, and this is tight in the sense that the KS threshold can be surpassed in polynomial time when $q\gg n^{1/2}$. See also~\cite{sbm-above-sqn} for a lower bound in the regime $q\gg n^{1/2}$. The concurrent work~\cite{sbm-reduction} also argues hardness of weak recovery below the KS bound for slowly growing $q=n^{o(1)}$, leveraging a connection to testing, and the subsequent work~\cite{alg-contig-2} further pushes this approach to $q \ll n^{1/8}$.

\subsection*{Organization}

The remaining sections are devoted to the proofs of our main results. Section~\ref{sec:subm} considers the planted submatrix model, Section~\ref{sec:pds} considers the planted dense subgraph model, Section~\ref{sec:wigner} considers the spiked Wigner model, and Section~\ref{sec:sbm} considers the stochastic block model. All the low-degree upper bounds are in Section~\ref{sec:upper}.

\section{Planted Submatrix}
\label{sec:subm}

This section is devoted to the proof of Theorem~\ref{thm:main-subm}(a). We will follow the strategy from Section~\ref{sec:pf-techniques}. The main steps were already outlined in Section~\ref{sec:ex-subm} and we fill in the remaining details here.

\subsection{Orthonormal Basis}

Let $\{H_\alpha\}$ denote the multivariate Hermite polynomials, normalized so that $\EE[H_\alpha(Z) \cdot H_\beta(Z)] = \one_{\alpha=\beta}$. Here $\alpha \in \NN^{\mcn}$, where we use ${\mcn}$ to abberviate ${\{(i,j) \,:\, 1 \le i \le j \le n\}}$ (see Section~\ref{sec:notation}). We will use $\{h_k\}_{k \in \NN}$ for the corresponding univariate Hermite polynomials. As our basis for $\RR[Y]_{\le D}$, choose $\phi_\alpha(Y) := H_\alpha(Y)$ for $\alpha \in \cI := \{\alpha \in \NN^{\mcn} \,:\, |\alpha| \le D\}$. For our orthonormal polynomials in the underlying independent random variables $W = (Z,\theta)$, choose
\[ \psi_{\beta\gamma}(Z,\theta) := H_\beta(Z) \left(\frac{\theta-\rho}{\sqrt{\rho(1-\rho)}}\right)^\gamma \]
for $\beta\gamma \in \cJ := \{(\beta,\gamma) \,:\, \beta \in \NN^{\mcn}, \, |\beta| \le D, \, \gamma \in \{0,1\}^n\}$. We will often view $\gamma$ as a subset $\gamma \subseteq [n]$ as described in Section~\ref{sec:notation}. Note that $\{\psi_{\beta\gamma}\}$ is orthonormal as required.

Apply the standard expansion (see~\cite[Proposition~3.1]{SW-estimation})
\begin{align*}
H_\alpha(Y) &= \prod_{i \le j} h_{\alpha_{ij}}(X_{ij} + Z_{ij}) \\
&= \prod_{i \le j} \sum_{k = 0}^{\alpha_{ij}} \sqrt{\frac{k!}{\alpha_{ij}!}} \binom{\alpha_{ij}}{k} X_{ij}^{\alpha_{ij} - k} h_k(Z_{ij}) \\
&= \sum_{0 \le \beta \le \alpha} \sqrt{\frac{\beta!}{\alpha!}} \binom{\alpha}{\beta} X^{\alpha-\beta} H_\beta(Z).
\end{align*}

\noindent Compute
\[ c_\alpha = \EE[H_\alpha(Y) \cdot \theta_1] = \frac{1}{\sqrt{\alpha!}} \EE[X^\alpha \cdot \theta_1] = \frac{\lambda^{|\alpha|} \rho^{|V(\alpha) \cup \{1\}|}}{\sqrt{\alpha!}} \]
and
\begin{align*}
M_{\beta\gamma,\alpha} &= \EE[H_\alpha(Y) \cdot \psi_{\beta\gamma}(Z,v)] \\
&= \one_{\beta \le \alpha} \cdot \sqrt{\frac{\beta!}{\alpha!}} \binom{\alpha}{\beta} \EE\left[X^{\alpha-\beta}\left(\frac{\theta-\rho}{\sqrt{\rho(1-\rho)}}\right)^\gamma\right] \\
&= \one_{\beta \le \alpha} \cdot \one_{\gamma \subseteq V(\alpha-\beta)} \cdot \sqrt{\frac{\beta!}{\alpha!}} \binom{\alpha}{\beta} \lambda^{|\alpha - \beta|} \sqrt{\rho(1-\rho)}^{|\gamma|} \rho^{|V(\alpha-\beta) \setminus \gamma|} \\
&= \one_{\beta \le \alpha} \cdot \one_{\gamma \subseteq V(\alpha-\beta)} \cdot \sqrt{\frac{\beta!}{\alpha!}} \binom{\alpha}{\beta} \lambda^{|\alpha - \beta|} \rho^{|V(\alpha-\beta)|} \left(\frac{1-\rho}{\rho}\right)^{|\gamma|/2}.
\end{align*}

\subsection{Removing ``Bad'' Terms}

Towards applying Lemma~\ref{lem:disconn}, we define $\hat\cI$ and $\hat\cJ$ as follows.

\begin{definition}[``Good'' graphs for planted submatrix]
\,
\begin{itemize}
    \item By convention, $0 \in \hat\cI$. For $|\alpha| \ge 1$ we include $\alpha$ in $\hat\cI$ if and only if $\alpha$ (when viewed as a multigraph) is connected with $1 \in V(\alpha)$.
    \item We include $\beta\gamma$ in $\hat\cJ$ if and only if $\beta \in \hat\cI$ and $\gamma \subseteq V(\beta) \cup \{1\}$. (The union with $\{1\}$ only matters in the case $\beta = 0$.)
\end{itemize}
\end{definition}

\noindent We need to check the conditions in Lemma~\ref{lem:disconn}. We will set $u_{\beta\gamma} = 0$ for all $\beta\gamma \notin \hat\cJ$. For the second condition, fix $\alpha \notin \hat\cI$. The corresponding $\hat\alpha$ is the connected component of vertex 1 in $\alpha$. (If vertex 1 is isolated in $\alpha$ then $\hat\alpha = 0$.) The value $\mu$ is equal to $\EE[H_{\alpha - \hat\alpha}(Y)]$. Using independence between components,
\[ c_\alpha = \EE[H_{\hat\alpha}(Y) \cdot H_{\alpha - \hat\alpha}(Y) \cdot x] = \EE[H_{\alpha - \hat\alpha}(Y)] \cdot \EE[H_{\hat\alpha}(Y) \cdot x] = \mu c_{\hat\alpha}. \]
If $\beta\gamma \in \hat\cJ$ is such that $\beta$ shares no vertices with $\alpha-\hat\alpha$ we can again argue by independence,
\[ M_{\beta\gamma,\alpha} = \EE[H_{\hat\alpha}(Y) \cdot H_{\alpha-\hat\alpha}(Y) \cdot \psi_{\beta\gamma}(Z,v)] = \EE[H_{\alpha-\hat\alpha}(Y)] \cdot \EE[H_{\hat\alpha}(Y) \cdot \psi_{\beta\gamma}(Z,v)] = \mu M_{\beta\gamma,\hat\alpha}. \]
Otherwise, $\beta$ shares a vertex with $\alpha-\hat\alpha$. Since $\beta$ is connected with $1 \in V(\beta)$, there must be an edge $(i,j)$ where $\alpha_{ij} = 0$ but $\beta_{ij} \ge 1$. This implies $M_{\beta\gamma,\alpha} = M_{\beta\gamma,\hat\alpha} = 0$, recalling that the expression for $M_{\beta\gamma,\alpha}$ is zero unless $\beta \le \alpha$. This verifies the conditions of Lemma~\ref{lem:disconn}.

\subsection{Constructing $u$}

Combining Proposition~\ref{prop:master} and Lemma~\ref{lem:disconn}, and setting $u_{\beta\gamma} = 0$ for $\beta\gamma \notin \hat\cJ$, our goal is now to choose values $(u_{\beta\gamma})_{\beta\gamma \in \hat\cJ}$ and verify
\[ \sum_{\beta\gamma} u_{\beta\gamma} M_{\beta\gamma,\alpha} = c_\alpha \qquad \forall \alpha \in \hat\cI. \]
Using the support structure of $M$ and the fact $M_{\alpha 0,\alpha} = 1$, our objective can be written as
\begin{equation}\label{eq:u-rec}
u_{\alpha 0} = c_\alpha - \sum_{\beta \lneq \alpha} \left( \sum_{\gamma \subseteq V(\alpha-\beta)} u_{\beta\gamma} M_{\beta\gamma,\alpha}\right) \qquad \forall \alpha \in \hat\cI.
\end{equation}
For all $\alpha\gamma \in \hat\cJ$ with $\gamma \ne 0$, we refer to the values $u_{\alpha\gamma}$ as \emph{free variables}. We are free to choose values for these variables and then the remaining values $u_{\alpha 0}$ are determined by the recurrence~\eqref{eq:u-rec}. Setting all free variables to zero recovers the existing method of~\cite{SW-estimation}. We will instead set the free variables in such a way to zero out the term $\left(\cdots\right)$ in~\eqref{eq:u-rec}. As a result, we will have simply $u_{\alpha 0} = c_\alpha$ for all $\alpha \in \hat\cI$.

\begin{lemma}\label{lem:def-u-subm}
The choice
\[ u_{\alpha\gamma} = \left(-\sqrt{\frac{\rho}{1-\rho}}\right)^{|\gamma|} \cdot c_\alpha \qquad \forall \alpha\gamma \in \hat\cJ \]
and $u_{\alpha\gamma} = 0$ for $\alpha\gamma \notin\hat\cJ$, satisfies~\eqref{eq:u-rec}. 
\end{lemma}

\begin{proof}
Fix $\alpha \in \hat\cI$. To verify~\eqref{eq:u-rec} it suffices to show that the term in parentheses is zero. For $\beta \notin \hat\cI$ this is immediate since each $u_{\beta\gamma}$ is zero, so fix $\beta \in \hat\cI$ with $\beta \lneq \alpha$. From the expression for $M$ we have for any $\gamma \subseteq V(\alpha-\beta)$ the identity $M_{\beta\gamma,\alpha} = \left(\frac{1-\rho}{\rho}\right)^{|\gamma|/2} M_{\beta 0,\alpha}$. Also recall that, since $\beta \in \hat\cI$, we have $\beta\gamma \in \hat\cJ$ if and only if $\gamma \subseteq V(\beta) \cup \{1\}$. Our quantity of interest is now
\begin{align*}
\sum_{\gamma \subseteq V(\alpha-\beta)} u_{\beta\gamma} M_{\beta\gamma,\alpha} &= M_{\beta 0,\alpha} \sum_{\gamma \subseteq V(\alpha-\beta)} u_{\beta\gamma} \left(\frac{1-\rho}{\rho}\right)^{|\gamma|/2} \\
&= M_{\beta 0,\alpha} \sum_{\gamma \subseteq V(\alpha-\beta) \cap (V(\beta) \cup \{1\})} c_\beta \left(-\sqrt{\frac{\rho}{1-\rho}}\right)^{|\gamma|} \left(\frac{1-\rho}{\rho}\right)^{|\gamma|/2} \\
&= c_\beta M_{\beta 0,\alpha} \sum_{\gamma \subseteq V(\alpha-\beta) \cap (V(\beta) \cup \{1\})} (-1)^{|\gamma|}.
\end{align*}
This is zero so long as the set $V(\alpha-\beta) \cap (V(\beta) \cup \{1\})$ is non-empty. And indeed this set is non-empty since $\alpha,\beta \in \hat\cI$ with $\beta \lneq \alpha$.
\end{proof}

\subsection{Putting it Together}

We have
\begin{align*}
\rho \cdot \Corr_{\le D}^2 \le \|u\|^2 &= \sum_{\alpha\gamma \in \hat\cJ} \left(\frac{\rho}{1-\rho}\right)^{|\gamma|} c_\alpha^2 \\
&= \sum_{\alpha \in \hat\cI} c_\alpha^2 \sum_{\gamma \subseteq V(\alpha) \cup \{1\}} \left(\frac{\rho}{1-\rho}\right)^{|\gamma|} \\
&= \sum_{\alpha \in \hat\cI} c_\alpha^2 \left(1 + \frac{\rho}{1-\rho}\right)^{|V(\alpha) \cup \{1\}|} \\
&= \sum_{\alpha \in \hat\cI} \frac{1}{\alpha!} \lambda^{2|\alpha|} \rho^{2|V(\alpha) \cup \{1\}|} \left(1 + \frac{\rho}{1-\rho}\right)^{|V(\alpha) \cup \{1\}|}.
\end{align*}

\begin{lemma}
For $v \ge 1$ and $k \ge 0$, the number of multigraphs $\alpha \in \hat\cI$ with $|V(\alpha) \cup \{1\}| = v$ and $|\alpha| = v-1+k$ is at most
\[ \binom{n-1}{v-1} v^{v-2} \binom{\frac{v(v+1)}{2} + k - 1}{k}. \]
\end{lemma}

\begin{proof}
Recall that $\alpha \in \hat\cI$ means (if $\alpha \ne 0$) that $\alpha$ must be connected and contain vertex $1$. The first factor $\binom{n-1}{v-1}$ counts the possible choices for $V(\alpha) \cup \{1\}$. The next factor $v^{v-2}$ is Cayley's tree formula, the number of spanning trees on $v$ (labeled) vertices. The final factor counts, using the ``stars and bars'' method, the number of ways to add $k$ additional edges that do not span any new vertices. Note that there are $\frac{v(v+1)}{2}$ possible edges (including self-loops) that span the existing $v$ vertices, and we may add multiple copies of each.
\end{proof}

\begin{proof}[Proof of Theorem~\ref{thm:main-subm}(a)]
Letting
\[ \tilde\rho := \rho \sqrt{1+\frac{\rho}{1-\rho}} = \frac{\rho}{\sqrt{1-\rho}} \]
and combining the results above,
\begin{align*}
\Corr_{\le D}^2 &\le \rho^{-1} \sum_{\alpha \in \hat\cI} \lambda^{2|\alpha|} \tilde\rho^{2|V(\alpha) \cup \{1\}|} \\
&\le \rho^{-1} \sum_{v = 1}^{D+1} \sum_{k = 0}^D \binom{n-1}{v-1} v^{v-2} \binom{\frac{v(v+1)}{2} + k - 1}{k} \lambda^{2(v-1+k)} \tilde\rho^{2v} \\
&\le \rho^{-1} \tilde\rho^2 \sum_{v = 1}^{D+1} \binom{n-1}{v-1} v^{v-2} \lambda^{2(v-1)} \tilde\rho^{2(v-1)} \sum_{k = 0}^D \binom{v^2 + k}{k} \lambda^{2k}.
\end{align*}
Our goal is to bound everything aside from the initial factor of $\rho^{-1} \tilde\rho^2 = \rho/(1-\rho)$ by a constant depending only on $\eps$. Focus first on the terms that do not involve $k$. We will use the standard bound $\binom{n}{k} \le \left(\frac{en}{k}\right)^k$ for all $k \ge 1$, as well as the assumption $\lambda \le (1-\eps)(\tilde\rho\sqrt{en})^{-1}$. For $v \ge 2$,
\[ b_v := \binom{n-1}{v-1} v^{v-2} \lambda^{2(v-1)} \tilde\rho^{2(v-1)}
\le \left(\frac{e(n-1)}{v-1}\right)^{v-1} (v\lambda^2\tilde\rho^2)^{v-1} 
\le \left((1-\eps)^2 \cdot \frac{v}{v-1}\right)^{v-1}, \]
and $b_1 = 1$. As a result,
\begin{itemize}
\item $b_v \le 2^{v-1}$ for all $v \ge 1$, and
\item $b_v \le (1-\eps)^{v-1}$ for all sufficiently large $v$ (i.e., $v$ exceeding some $\eps$-dependent constant).
\end{itemize}
Now focus on the sum over $k$:
\begin{align*}
\sum_{k = 0}^D \binom{v^2 + k}{k} \lambda^{2k} &\le 1 + \sum_{k=1}^D \left(\frac{e(v^2+k)}{k}\right)^k \lambda^{2k} \\
&\le 1 + \sum_{1 \le k \le v^2} \left(e\lambda^2 \left(1 + \frac{v^2}{k}\right)\right)^k + \sum_{v^2 \le k \le D} \left(e\lambda^2 \left(1 + \frac{v^2}{k}\right)\right)^k \\
&\le 1 + v^2 \sup_{k \in (0,\infty)} \left(\frac{2e\lambda^2 v^2}{k}\right)^k + \sum_{v^2 \le k \le D} (2e\lambda^2)^k.
\end{align*}
For $D \ge 1$ we have $1 \le v \le D+1 \le 2D \le 2\lambda^{-2}/C$, using the assumption on $D$. Choose $C$ large enough so that $2e\lambda^2 \le 1/2$, and so the final term above is bounded by 1. The supremum over $k \in (0,\infty)$ is attained at $k = 2\lambda^2 v^2$ and has value $\exp(2\lambda^2 v^2) \le \exp(4v/C)$. Putting everything together, we now conclude
\[ \Corr_{\le D}^2 \le \frac{\rho}{1-\rho} \sum_{v=1}^\infty b_v [2 + v^2 \exp(4v/C)]. \]
To complete the proof, we claim that the sum above (excluding the $\rho/(1-\rho)$ factor) is bounded by an $\eps$-dependent constant. To see this, recall $b_v \le (1-\eps)^{v-1}$ for all sufficiently large $v$, and by choosing $C$ large enough we can arrange $2+v^2 \exp(4v/C) \le (1+\eps/2)^v$ for all sufficiently large $v$.
\end{proof}

\section{Planted Dense Subgraph}
\label{sec:pds}

This section is devoted to the proof of Theorem~\ref{thm:main-pds}(a).

\subsection{Orthonormal Basis}

For a graph $\alpha\in \{0,1\}^{\bcn}$, choose
\begin{equation*}
    \phi_{\alpha}(Y)=(Y-p_0)^{\alpha}.
\end{equation*}
Moreover, for $\beta \in \{0,1\}^{\bcn}$ and $\gamma\equiv (\gamma_i) \in\{0,1\}^n$, consider 
\begin{equation*}
    \psi_{\beta \gamma}(Y,\theta)=\prod_{i,j:\theta_i=\theta_j=1}\left(\frac{Y_{ij}-p_1}{\sqrt{p_1(1-p_1)}}\right)^{\beta_{ij}}\prod_{i,j:\theta_i \theta_j=0}\left(\frac{Y_{ij}-p_0}{\sqrt{p_0(1-p_0)}}\right)^{\beta_{ij}}\prod_{i=1}^{n}\left(\frac{\theta_i-\rho}{\sqrt{\rho(1-\rho)}}\right)^{\gamma_i}.
\end{equation*}
\begin{remark}[$p_1=1$]
In the case $p_1=1$, we incur a division by zero above. However, our final result remains true by a continuity argument: for any $f \in \RR[Y]$, the mean squared error $\EE[(f(Y)-x)^2]$ is a continuous function of $p_1$. We thus assume $p_1<1$ for the rest of this section.
\end{remark}
We first prove that $\{\psi_{\beta\gamma}\}$ is orthonormal.
\begin{lemma}
    $\{\psi_{\beta\gamma}\}_{\beta\in\{0,1\}^{\bcn},\gamma\in \{0,1\}^n}$ is orthonormal.
\end{lemma}
\begin{proof}
We argue that $\E[\psi_{\beta\gamma}\psi_{\beta^\prime\gamma^\prime}]=\one(\beta=\beta^\prime, \gamma=\gamma^\prime)$ holds for arbitrary $\beta,\beta^\prime \in \{0,1\}^{\bcn}$ and $\gamma,\gamma^\prime\in \{0,1\}^{n}$. Recall that conditioned on $\theta$, we have independently for each pair $i<j$ that $Y_{ij}\sim \Ber (p_1)$ if $\theta_i=\theta_j=1$ and $Y_{ij}\sim \Ber(p_0)$ if $\theta_i \theta_j=0$. Thus,
\begin{equation*}
\E\Big[\prod_{\theta_i=\theta_j=1}\left(\frac{Y_{ij}-p_1}{\sqrt{p_1(1-p_1)}}\right)^{\beta_{ij}+\beta^\prime_{ij}}\prod_{\theta_i \theta_j=0}\left(\frac{Y_{ij}-p_0}{\sqrt{p_0(1-p_0)}}\right)^{\beta_{ij}+\beta^\prime_{ij}}\Bgiven \theta\Big]=\one(\beta=\beta^\prime). 
\end{equation*}
By the tower property, it follows that
\begin{equation*}
\E[\psi_{\beta\gamma}\psi_{\beta^\prime\gamma^\prime}]=\one(\beta=\beta^\prime)\E\prod_{i=1}^{n}\left(\frac{\theta_i-\rho}{\sqrt{\rho(1-\rho)}}\right)^{\gamma_i+\gamma^\prime_i}=\one(\beta=\beta^\prime,\gamma=\gamma^\prime),
\end{equation*}
where the last equality holds since $(\theta_i)\iid \Ber(\rho)$.
\end{proof}
Let us now compute $c_{\al}=\E[\phi_{\al}\cdot \theta_1]$ and $M_{\beta\gamma, \al}=\E[\psi_{\be \ga}\cdot \phi_{\al}]$. First,
\begin{equation}\label{eq:c:pds}
c_{\al}=\E\left[\theta_1\E\left[\phi_{\al}\bgiven \theta\right]\right]=(p_1-p_0)^{|\al|}\,\P\left(\theta_i=1,~\forall i\in V(\al)\cup \{1\}\right)=\rho^{|V(\alpha) \cup \{1\}|}(p_1-p_0)^{|\al|}.
\end{equation}
To compute $M_{\be\ga,\al}$, note that
\begin{equation*}
\begin{split}
    &\E\bigg[\prod_{\theta_i=\theta_j=1}\left\{\left(\frac{Y_{ij}-p_1}{\sqrt{p_1(1-p_1)}}\right)^{\beta_{ij}}(Y_{ij}-p_0)^{\alpha_{ij}}\right\}\prod_{\theta_i \theta_j=0}\left\{\left(\frac{Y_{ij}-p_0}{\sqrt{p_0(1-p_0)}}\right)^{\beta_{ij}}(Y_{ij}-p_0)^{\alpha_{ij}}\right\}\bbgiven \theta\bigg]\\
    &=\one_{\be \le \al}\cdot \one_{\theta_i=1,\forall i \in V(\al-\be)}\cdot (p_1-p_0)^{|\al|-|\be|}\left(p_1(1-p_1)\right)^{\frac{\ell(\beta;\theta)}{2}}\left(p_0(1-p_0)\right)^{\frac{|\beta|-\ell(\beta;\theta)}{2}},
\end{split}
\end{equation*}
where $\ell(\be;\theta)$ denotes the number of common edges between $\be$ and the planted graph:
\begin{equation*}
    \ell(\be;\theta)=\Big|\big\{(i,j)\in E(\beta) \,:\, \theta_i=\theta_j=1\big\}\Big|.
\end{equation*}
Thus, it follows that
\begin{equation}\label{eq:M:pds}
\begin{split}
    M_{\be\ga,\al}
    &=\one_{\be\le\al}\cdot (p_1-p_0)^{|\al|-|\be|}\left(p_0(1-p_0)\right)^{|\be|/2}\\
    &\quad\quad\quad\quad\quad \times \E\left[\left(\frac{p_1(1-p_1)}{p_0(1-p_0)}\right)^{\ell(\be;\theta)/2}\left(\frac{\theta-\rho}{\sqrt{\rho(1-\rho)}}\right)^{\ga}\one_{\theta_i=1,\forall i \in V(\alpha-\beta)}\right].
    \end{split}
\end{equation}

\subsection{Removing ``Bad'' Terms}
We now apply Lemma~\ref{lem:disconn}. We define $\hat\cI$ and $\hat\cJ$ the same as in the planted submatrix model (except we are now working with simple graphs rather than multigraphs), which we recall for the reader's convenience:
\begin{definition}[``Good'' graphs for planted dense subgraph]
\,
\begin{itemize}
    \item By convention, $0 \in \hat\cI$. For $|\alpha| \ge 1$ we include $\alpha$ in $\hat\cI$ if and only if $\alpha$ (when viewed as a graph) is connected with $1 \in V(\alpha)$.
    \item We include $\beta\gamma$ in $\hat\cJ$ if and only if $\beta \in \hat\cI$ and $\gamma \subseteq V(\beta) \cup \{1\}$. (The union with $\{1\}$ only matters in the case $\beta = 0$.)
\end{itemize}
\end{definition}
\noindent We now check the conditions in Lemma~\ref{lem:disconn}. We will set $u_{\beta\gamma} = 0$ for all $\beta\gamma \notin \hat\cJ$. For the second condition, fix $\alpha \notin \hat\cI$. The corresponding $\hat\alpha$ is the connected component of vertex 1 in $\alpha$. (If vertex 1 is isolated in $\alpha$ then $\hat\alpha = 0$.) The value $\mu$ is equal to $\E[\phi_{\alpha - \hat\alpha}]$. Using independence between components,
\[ c_\alpha = \E[\phi_{\hat\alpha}\cdot \phi_{\alpha - \hat\alpha} \cdot x] = \E[\phi_{\alpha - \hat\alpha}] \cdot \E[\phi_{\hat\alpha}\cdot x] = \mu c_{\hat\alpha}. \]
Here, we used the fact that $V(\al-\hat\al)\cap (V(\hat\al)\cup\{1\})=\emptyset$ in the second equality.

If $\beta\gamma \in \hat\cJ$ is such that $\beta$ shares no vertices with $\alpha-\hat\alpha$ we can again argue by independence,
\[ M_{\beta\gamma,\alpha} = \EE[\phi_{\hat\alpha}\cdot \phi_{\alpha-\hat\alpha} \cdot \psi_{\beta\gamma}] = \EE[\phi_{\alpha-\hat\alpha}] \cdot \EE[\phi_{\hat\alpha} \cdot \psi_{\beta\gamma}] = \mu M_{\beta\gamma,\hat\alpha}. \]
Otherwise, $\beta$ shares a vertex with $\alpha-\hat\alpha$. Since $\beta$ is connected with $1 \in V(\beta)$, there must be an edge $(i,j)$ where $\alpha_{ij} = 0$ but $\beta_{ij} \ge 1$. This implies $M_{\beta\gamma,\alpha} = M_{\beta\gamma,\hat\alpha} = 0$ as seen in Eq.~\eqref{eq:M:pds}. This verifies the conditions of Lemma~\ref{lem:disconn}.

\subsection{Constructing $u$}

We proceed as in the planted submatrix problem. In light of Lemma~\ref{lem:disconn}, we set $u_{\be\ga}=0$ for $\be\ga\notin \hat\cJ$, and our goal is to construct $(u_{\be\ga})_{\be\ga\in \hat\cJ}$ such that
\begin{equation*}
    \sum_{\be\ga \in \hat\cJ} u_{\be\ga}M_{\be\ga,\al}=c_{\al} \qquad \forall\al\in \hat\cI
\end{equation*}
and such that $\|u\|$ is ``small enough.'' Note that the situation for planted dense subgraph is more complicated than planted submatrix since $M_{\be \ga,\al}$ for the planted dense subgraph problem in Eq.~\eqref{eq:M:pds} is not explicit. Nevertheless, the following proposition shows that we can construct $u$ that resembles the one constructed for the planned submatrix problem.
\begin{proposition}\label{prop:u:pds}
    The choice
    \[
    u_{\alpha\gamma} = \left(-\sqrt{\frac{\rho}{1-\rho}}\right)^{|\gamma|} \cdot \frac{c_\alpha}{\left(p_0(1-p_0)\right)^{|\al|/2}} \qquad \forall \alpha\gamma \in \hat\cJ
    \]
    satisfies $\sum_{\be\ga \in \hat\cJ} u_{\be\ga}M_{\be\ga,\al}=c_{\al}$ for all $\al\in \hat\cI$.
\end{proposition}
To prove Proposition~\ref{prop:u:pds}, we prove the following lemma.
\begin{lemma}\label{lem:crucial:identity}
    Suppose $\al,\be\in \hat\cI$. Then, we have
    \[
    \sum_{\ga \subseteq V(\be)\cup\{1\}}\left(-\sqrt{\frac{\rho}{1-\rho}}\right)^{|\gamma|}M_{\be\ga,\al}=\one_{\be=\al}\cdot \left(p_0(1-p_0)\right)^{|\al|/2}.
    \]
\end{lemma}
\noindent We note that the above lemma is trivial when $\beta \nleq \al$ because $M_{\be\ga,\al}=0$ in this case. The non-trivial part is to show that if $\beta\lneq \al$, then $\sum_{\ga \subseteq V(\be)\cup\{1\}}\left(-\sqrt{\frac{\rho}{1-\rho}}\right)^{|\gamma|}M_{\be\ga,\al}=0$. We also note that in the sum $\sum_{\ga \subseteq V(\be)\cup\{1\}}$, the union with $\{1\}$ is only relevant when $\beta=\emptyset$. Although this seems unimportant, it plays a crucial role in the construction of $u$, and this is the reason we included $(\beta \gamma)=(\emptyset, \{1\})$ in $\hat\cJ$ in the first place.
\begin{proof}
We consider $\be\le \al$, since otherwise both sides are trivially zero.
Recalling the expression for $M_{\be\ga,\al}$ in Eq.~\eqref{eq:M:pds}, $\sum_{\ga \subseteq V(\be)\cup\{1\}}\left(-\sqrt{\frac{\rho}{1-\rho}}\right)^{|\gamma|}M_{\be\ga,\al}$ equals 
\[
(p_1-p_0)^{|\al|-|\be|}\left(p_0(1-p_0)\right)^{|\be|/2}\sum_{\gamma\subseteq V(\beta)\cup \{1\}}\E\left[\left(\frac{p_1(1-p_1)}{p_0(1-p_0)}\right)^{\ell(\be;\theta)/2}\left(-\frac{\theta-\rho}{1-\rho}\right)^{\ga}\one_{\theta_i=1,\forall i \in V(\alpha-\beta)}\right].
\]
A crucial observation is that
\[
\sum_{\gamma\subseteq V(\beta)\cup \{1\}}\left(-\frac{\theta-\rho}{1-\rho}\right)^{\ga}=\prod_{i\in V(\beta)\cup\{1\}}\left(1-\frac{\theta_i-\rho}{1-\rho}\right)=(1-\rho)^{-|V(\beta)\cup\{1\}|}\cdot \one_{\theta_i=0,\forall i \in V(\beta)\cup\{1\}}.
\]
Since $\alpha, \beta\in \hat\cI$, note that $V(\al-\be)\cap (V(\be)\cup\{1\})$ is non-empty unless $\al=\be$. Therefore, combining the two displays above yields
\[
\begin{split}
&\sum_{\ga \subseteq V(\be)\cup\{1\}}\left(-\sqrt{\frac{\rho}{1-\rho}}\right)^{|\gamma|}M_{\be\ga,\al}\\
&\qquad=\one_{\al=\be} \left(p_0(1-p_0)\right)^{|\be|/2} (1-\rho)^{-|V(\beta)\cup\{1\}|}\E\left[\left(\frac{p_1(1-p_1)}{p_0(1-p_0)}\right)^{\ell(\be;\theta)/2}\one_{\theta_i=0,\forall i \in V(\beta)\cup\{1\}}\right].
\end{split}
\]
Note that on the event that $\theta_i=0$ holds for all $i\in V(\beta)\cup \{1\}$, $\ell(\beta;\theta)=0$ holds. Thus, the expectation in the display above equals $(1-\rho)^{|V(\beta)\cup \{1\}|}$, which concludes the proof. 
\end{proof}

\begin{proof}[Proof of Proposition~\ref{prop:u:pds}]
For a given $\al, \beta\in \hat\cI$, we have
   \[
   \begin{split}
   \sum_{\ga: \be\ga\in \hat\cJ}u_{\be \ga}M_{\be \ga,\al}
   &=\sum_{\ga: \be\ga\in \hat\cJ}\left(-\sqrt{\frac{\rho}{1-\rho}}\right)^{|\gamma|} \cdot \frac{c_\alpha}{\left(p_0(1-p_0)\right)^{|\al|/2}}\cdot M_{\be\ga,\al}\\
   &=\one_{\al=\be}\cdot c_\al,
   \end{split}
   \]
   where the last equality is by Lemma~\ref{lem:crucial:identity}. Thus, summing over $\beta \in \hat\cI$ concludes the proof.
\end{proof}
\begin{proof}[Proof of Theorem~\ref{thm:main-pds}(a)]
    Note that $u_{\al\ga}$ constructed in Proposition~\ref{prop:u:pds} is given by
    \[
    u_{\al\ga}=\left(-\sqrt{\frac{\rho}{1-\rho}}\right)^{|\gamma|} \left(\frac{p_1-p_0}{\sqrt{p_0(1-p_0)}}\right)^{|\al|}\rho^{|V(\al)\cup\{1\}|} \qquad \al\in \hat\cI
    \]
    where we used the expression~\eqref{eq:c:pds} for $c_{\al}$. Observe that this equals $u_{\al\ga}$ for the planted submatrix problem constructed in Lemma~\ref{lem:def-u-subm}, with the substitution $\la=\frac{p_1-p_0}{\sqrt{p_0(1-p_0)}}$. Since the sets $\hat\cI$ and $\hat\cJ$ are also the same as the planted submatrix problem (in fact, here there are fewer since we restrict to simple graphs) the exact same calculations as in Section~\ref{sec:subm} yield the desired result. 
\end{proof}

\section{Spiked Wigner}
\label{sec:wigner}

This section is devoted to the proof of Theorem~\ref{thm:main-wigner}(a). For this result we will not need to use the full power of our new framework but can instead use the method of~\cite{SW-estimation} as a starting point. In bounding the resulting formula, we will need to use some arguments that are somewhat more delicate than those appearing in~\cite{SW-estimation}, in order to capture the sharp threshold.

\subsection{Cumulant Formula}
\label{sec:cumulant}
Our starting point will be Theorem~2.2 of~\cite{SW-estimation}, which we recap here. This result pertains to the general additive Gaussian noise model $Y = X+Z$ with $Z$ i.i.d.\ $\cN(0,1)$ and $X$ drawn from an arbitrary distribution (independent from $Z$), of which our setting is a special case. Accounting for our difference in normalization, the result is
\[ \EE[x^2] \cdot \Corr_{\le D}^2 \le \sum_{\substack{\alpha \in \NN^{\mcn} \\ 0 \le |\alpha| \le D}} \frac{\kappa_\alpha^2}{\alpha!} \]
where $\kappa_\alpha$ is defined recursively by
\[ \kappa_\alpha = \EE[x X^\alpha] - \sum_{0 \le \beta \lneq \alpha} \binom{\alpha}{\beta} \EE[X^{\alpha-\beta}] \kappa_\beta. \]
Equivalently, $\kappa_\alpha$ is equal to the joint cumulant of $|\alpha|+1$ (dependent) random variables: $x$, along with $\alpha_{ij}$ copies of $X_{ij}$ for each $i \le j$. (These are not independent copies but identical copies of the same random variable.)

\subsection{The Case $m=1$}

We focus for now on the case $m=1$, as the general case will follow easily from this using standard properties of cumulants. Here we have $X = \sqrt{\frac{\lambda}{n}}uu^\top$ where $u \in \RR^n$ has entries i.i.d.\ from $\pi$.

\begin{definition}[``Good'' graphs for spiked Wigner]
\label{def:good:SW}
Say $\alpha \in \NN^{\mcn}$ is ``good'' if either $\alpha = 0$ or if $\alpha$ (viewed as a multigraph) satisfies all of the following:
\begin{itemize}
    \item $1,2 \in V(\alpha)$,
    \item $\bar\alpha := \alpha + \one_{(1,2)}$ is connected,
    \item for every $v \in V(\alpha)$, $\deg_{\bar\alpha}(v) \ge 2$.
\end{itemize}

\end{definition}

\begin{lemma}
If $\alpha$ is ``bad'' (i.e., not good), then $\kappa_\alpha = 0$.
\end{lemma}

\noindent For consistency with the SBM in the next section, we are defining the empty graph $\alpha = 0$ to be ``good,'' even though $\kappa_0 = \EE[x] = 0$.

\begin{proof}
We will use some standard properties of cumulants that are presented in Section~2.5 of~\cite{SW-estimation}. Recall $\kappa_\alpha$ is the joint cumulant of the collection of random variables $X_{ij} = \sqrt{\frac{\lambda}{n}} u_i u_j$ for $(i,j) \in E(\bar\alpha)$. First, $\bar\alpha$ must be connected or else we can partition our collection of random variables into two independent parts and conclude $\kappa_\alpha = 0$ by~\cite[Proposition~2.11]{SW-estimation}. Also, $\bar\alpha$ cannot have a degree-1 vertex $i$, or else $X_{ij}$ for the single edge $(i,j)$ is uncorrelated with any function of the other random variables (due to the independent mean-zero $u_i$), and we can conclude $\kappa_\alpha = 0$ using the combinatorial definition of joint cumulants as a sum over partitions~\cite[Definition~2.10]{SW-estimation}.
\end{proof}

Bound the moments: for any $\alpha \in \NN^{\mcn}$,
\[ |\EE[X^\alpha]| = \left(\frac{\lambda}{n}\right)^{|\alpha|/2} \left|\EE\prod_{(i,j) \in E(\alpha)} u_i u_j\right| = \left(\frac{\lambda}{n}\right)^{|\alpha|/2} \prod_{v \in V(\alpha)} |\EE[\pi^{\deg_\alpha(v)}]| \le \left(\frac{\lambda}{n}\right)^{|\alpha|/2} \prod_{\substack{v \in V(\alpha) \\ \deg_\alpha(v) \ge 3}} \EE|\pi|^{\deg_\alpha(v)}. \]

\noindent For $a,b > 0$ we have $\EE|\pi|^a \EE|\pi|^b \le \EE|\pi|^{a+b}$, since two increasing functions of the same random variable always have nonnegative covariance (in our case, $t \mapsto t^a$ and $t \mapsto t^b$ are increasing functions applied to the random variable $|\pi|$). This gives
\[ |\EE[X^\alpha]| \le \left(\frac{\lambda}{n}\right)^{|\alpha|/2} \EE|\pi|^{\delta(\alpha)} \]
(with the convention $0^0 = 1$, in the case $\delta(\alpha) = 0$) where the ``excess degree'' $\delta$ is defined as
\[ \delta(\alpha) := \sum_{\substack{v \in V(\alpha) \\ \deg_\alpha(v) \ge 3}} \deg_\alpha(v). \]
Now to bound $\kappa_\alpha$ for good $\alpha$: since $\EE[xX^\alpha] = \EE[X^{\bar\alpha}]$ where, recall, $\bar\alpha := \alpha + \one_{(1,2)}$, we have
\begin{align}
|\kappa_\alpha| &\le |\EE[x X^\alpha]| + \sum_{0 \le \beta \lneq \alpha} \binom{\alpha}{\beta} |\EE[X^{\alpha-\beta}]| \cdot |\kappa_\beta| \nonumber \\
&\le \left(\frac{\lambda}{n}\right)^{(|\alpha|+1)/2} \EE|\pi|^{\delta(\bar\alpha)} + \sum_{\substack{0 \lneq \beta \lneq \alpha \\ \beta\text{ good}}} \binom{\alpha}{\beta} \left(\frac{\lambda}{n}\right)^{(|\alpha|-|\beta|)/2} \EE|\pi|^{\delta(\alpha-\beta)} \cdot |\kappa_\beta|,
\label{eq:kappa-recursion}
\end{align}
recalling $\kappa_0 = 0$.

\begin{lemma}\label{lem:spiked:cumulant}
    For any good $\alpha$,
    \[ |\kappa_\alpha| \le \left(\frac{\lambda}{n}\right)^{(|\alpha|+1)/2} f(\alpha) \cdot M(\delta(\bar\alpha)) \qquad\text{where } M(k) := \max_{j \in \{0,1,\ldots,k\}} \EE|\pi|^j, \]
and where $f(\cdot)$, which takes a good $\alpha$ and outputs a natural number, is given by
\begin{equation}\label{eq:def:recursion:f}
f(\alpha)=\sum_{\substack{0 \le \beta\lneq\alpha \\ \beta \textnormal{ good}}}\binom{\alpha}{\beta} f(\beta) \qquad \text{for } \alpha \ne 0,
\end{equation}
with base case $f(0) = 1$.
\end{lemma}

\begin{proof}
Proceed by strong induction on $|\alpha|$, i.e., fix a good $\alpha$ and assume the conclusion holds for all good $\beta \lneq \alpha$. First, focus on one term in~\eqref{eq:kappa-recursion} and apply the induction hypothesis:
\begin{align*}
T_{\alpha,\beta} := \left(\frac{\lambda}{n}\right)^{(|\alpha|-|\beta|)/2} \EE|\pi|^{\delta(\alpha-\beta)} \cdot |\kappa_\beta| &\le \left(\frac{\lambda}{n}\right)^{(|\alpha|-|\beta|)/2} \EE|\pi|^{\delta(\alpha-\beta)} \cdot \left(\frac{\lambda}{n}\right)^{(|\beta|+1)/2} f(\beta) \cdot M(\delta(\bar\beta)) \\
&= \left(\frac{\lambda}{n}\right)^{(|\alpha|+1)/2} \EE|\pi|^{\delta(\alpha-\beta)} \EE|\pi|^k f(\beta)
\end{align*}
for some $k \in \{0,1,\ldots,\delta(\bar\beta)\}$. Note that $(\alpha-\beta) + \bar\beta = \bar\alpha$ and so $\delta(\alpha-\beta) + k \le \delta(\alpha-\beta) + \delta(\bar\beta) \le \delta(\bar\alpha)$. This means
\[ \EE|\pi|^{\delta(\alpha-\beta)} \EE|\pi|^k \le \EE|\pi|^{\delta(\alpha-\beta) + k} \le M(\delta(\bar\alpha)), \]
and we conclude
\[ T_{\alpha,\beta} \le \left(\frac{\lambda}{n}\right)^{(|\alpha|+1)/2} f(\beta) \cdot M(\delta(\bar\alpha)). \]
Putting it together,
\begin{align*}
|\kappa_\alpha| &\le \left(\frac{\lambda}{n}\right)^{(|\alpha|+1)/2} \EE|\pi|^{\delta(\bar\alpha)} + \sum_{\substack{0 \lneq \beta \lneq \alpha \\ \beta\text{ good}}} \binom{\alpha}{\beta} T_{\alpha,\beta} \\
&\le \left(\frac{\lambda}{n}\right)^{(|\alpha|+1)/2} M(\delta(\bar\alpha)) \left[1 + \sum_{\substack{0 \lneq \beta \lneq \alpha \\ \beta\text{ good}}} \binom{\alpha}{\beta} f(\beta) \right],
\end{align*}
and then we recognize the term in brackets $[\cdots]$ as $f(\alpha)$.
\end{proof}

\begin{lemma}\label{lem:f:upper:bound}
For any good $\alpha \ne 0$,
\[
f(\alpha)\leq (2 |\alpha|)^{|\alpha|-|V(\alpha)|+1}.
\]
\end{lemma}

\begin{proof}
From the definition of $f$, we have the following identity for any good $\alpha \ne 0$:
\begin{equation}\label{eq:f-identity}
2f(\alpha)=\sum_{\substack{0 \le \beta \le \alpha \\ \beta \textnormal{ good}}} \binom{\alpha}{\beta}f(\beta) = 1 + \sum_{\substack{0 \lneq \beta \le \alpha \\ \beta \textnormal{ good}}} \binom{\alpha}{\beta}f(\beta).
\end{equation}
Proceed by strong induction on $|\alpha|$. For an edge $e \in E(\alpha)$, define $\alpha^\star(e) \in \NN^{\mcn}$ to be the unique maximal good graph contained in $\alpha - \one_e$. Specifically, $\alpha^\star(e)$ is obtained from $\alpha - \one_e$ by the following 3-step procedure. First, repeatedly remove any degree-1 vertices other than vertices $1,2$ (by deleting the one edge incident to them) until no degree-1 vertices remain except possibly $1,2$. (In other words, this first step removes from $\alpha$ the maximal path containing $e$ whose internal vertices have degree 2 and are not $1,2$.) Next, remove any connected components that are not connected to vertex $1$ or $2$. Finally, check that vertices $1$ and $2$ still have degree at least 1, and otherwise remove the entire graph, i.e., set $\alpha^\star(e) = 0$. The graph $\alpha^\star(e)$ produced by the above procedure is good, and (since we only removed edges that were necessary to remove in order to maintain a good graph) also has the property that if $\beta \le \alpha - \one_e$ is good then $\beta \le \alpha^\star(e)$.

If $\alpha \ne 0$ is good and has the property that the only good $\beta \lneq \alpha$ is $\beta = 0$, then we have $f(\alpha) = 1$ and $|\alpha|-|V(\alpha)|+1 \ge 0$, so we can verify the conclusion of the lemma directly. Otherwise,
\begin{align*}
f(\alpha) &= \sum_{\substack{0 \le \beta\lneq\alpha \\ \beta \textnormal{ good}}}\binom{\alpha}{\beta} f(\beta) \\
&= 1 + \sum_{\substack{0 \lneq \beta\lneq\alpha \\ \beta \textnormal{ good}}}\binom{\alpha}{\beta} f(\beta),
\intertext{and viewing $\sum_\beta \binom{\alpha}{\beta}$ as a sum over subgraphs --- i.e., subsets of the edge (multi)set $E(\alpha)$ --- of the multigraph described by $\alpha$,}
&\leq 1 + \sum_{e\in E(\alpha)}\sum_{\substack{0 \lneq \beta\leq \alpha-\one_e \\ \beta \textnormal{ good}}}\binom{\alpha-\one_e}{\beta}f(\beta),\\
&= 1 + \sum_{\substack{e\in E(\alpha) \\ \alpha^\star(e) \ne 0}}\sum_{\substack{0 \lneq \beta\leq \alpha^\star(e) \\ \beta \textnormal{ good}}}\binom{\alpha^\star(e)}{\beta}f(\beta),
\intertext{where the equality $\binom{\alpha-\one_e}{\beta} = \binom{\alpha^\star(e)}{\beta}$ is because $\alpha-\one_e$ and $\alpha^\star(e)$ only differ on $i,j$ pairs where $(\alpha^\star(e))_{ij} = 0$. Now applying the identity~\eqref{eq:f-identity} to $\alpha^\star(e)$,}
&= 1 + \sum_{\substack{e\in E(\alpha) \\ \alpha^\star(e) \ne 0}} [2 f(\alpha^\star(e)) - 1] \\
&\le \sum_{\substack{e\in E(\alpha) \\ \alpha^\star(e) \ne 0}} 2 f(\alpha^\star(e)),
\end{align*}
recalling that we checked by hand the case where every $e \in E(\alpha)$ has $\alpha^\star(e) = 0$. Next, note that for $\alpha^\star(e) \ne 0$, the process of producing $\alpha^\star(e)$ from $\alpha$ removes more edges than vertices, so $|\alpha^\star(e)| - |V(\alpha^\star(e))| \le |\alpha| - |V(\alpha)| - 1$. Continuing from above and applying the induction hypothesis,
\begin{align*}
f(\alpha) &\le \sum_{\substack{e\in E(\alpha) \\ \alpha^\star(e) \ne 0}} 2 (2 |\alpha^\star(e)|)^{|\alpha^\star(e)| - |V(\alpha^\star(e))| + 1} \\
&\le \sum_{\substack{e\in E(\alpha) \\ \alpha^\star(e) \ne 0}} 2 (2 |\alpha|)^{|\alpha| - |V(\alpha)|} \\
&\le |\alpha| \cdot 2 (2 |\alpha|)^{|\alpha| - |V(\alpha)|} \\
&= (2 |\alpha|)^{|\alpha| - |V(\alpha)| + 1},
\end{align*}
as desired.
\end{proof}

\begin{lemma}\label{lem:number:good:graphs:2}
For $d \ge 1$ and $v \ge 2$, the number of good $\alpha$ with $|\alpha|=d$ and $|V(\alpha)|=v$ is at most $n^{v-2}(2d)^{5(d-v+1)}$.
\end{lemma}

\begin{proof}
Starting with two isolated vertices $1$ and $2$, imagine building $\alpha$ by repeating the following operation: add a new edge whose endpoints are two existing vertices (which need not be distinct), and add some number (possibly 0) of degree-2 vertices along the new edge (so that this edge becomes a path). We also allow a variant of this operation where a ``lollipop'' (an edge with a self-loop at the end) rooted at some existing vertex is added, and again degree-2 vertices can be placed along the two edges; for counting purposes, we think of this as a case of the previous operation where one of the two endpoints of the new edge is not an existing vertex but rather one of the future vertices to be added along the new edge. We claim that any good $\alpha \ne 0$ can be produced by a sequence of these operations. This can be seen by working backwards from $\alpha$ as follows. At each step, choose a non-bridge edge (one whose removal won't disconnect $\bar\alpha$) and remove the maximal path containing this edge whose internal vertices have degree 2 (and are not vertices $1,2$). This results in a smaller graph and maintains the properties that $\bar\alpha$ is connected and all vertices have degree $\ge 2$ (except possibly $1,2$), allowing us to induct; the one exception is if the path we remove has both endpoints equal to some degree-3 vertex (now a degree-1 vertex), in which case we remove a lollipop.

Suppose $\alpha$ is produced from $T$ applications of the above operation, and let $\ell_t \ge 0$ be the number of degree-2 vertices added at step $t \in [T]$. Since each step adds $\ell_t$ vertices and $\ell_t+1$ edges, we must have $\sum_t \ell_t = v-2$ and $\sum_t (\ell_t+1) = d$, which implies $T = d-v+2$. The number of ways to choose the $\ell_t$ values for $t \in [T]$, subject to their sum being $v-2$, is $\binom{d-1}{d-v+1}$ using the ``stars and bars'' method. Each time we add an edge, there are at most $v^2$ choices for the two endpoints (which includes the possibility of choosing a ``future'' vertex to form a lollipop), for a total contribution of $v^{2T} = v^{2(d-v+2)}$. Finally, each new vertex must be assigned a label from $[n] \setminus \{1,2\}$, for a total contribution of $(n-2)^{v-2}$. In total, the count of $\alpha$'s is at most
\[ \binom{d-1}{d-v+1} v^{2(d-v+2)} (n-2)^{v-2} \le d^{d-v+1} v^{2(d-v+2)} n^{v-2}. \]
Define $\Delta := d-v+1$ and note that $\Delta \ge 0$ for good $\alpha$ (this is true for any graph without isolated vertices). The above becomes
\[ d^\Delta v^{2(\Delta+1)} n^{v-2} \le (2d)^\Delta (2d)^{2(\Delta+1)} n^{v-2} = (2d)^{3\Delta+2} n^{v-2}. \]
For $\Delta \ge 1$ we have $3\Delta+2 \le 5\Delta$, which gives the desired bound. For $\Delta = 0$, the only good $\alpha$ is a simple path from vertex $1$ to vertex $2$, (since only forests have $\Delta = 0$) so we can improve the bound to $n^{v-2}$, matching the desired bound.
\end{proof}

\begin{lemma}\label{lem:upper:delta}
For any good $\alpha$, we have $\delta(\bar\alpha) \le 6(|\alpha|-|V(\alpha)|+1)$.
\end{lemma}
\begin{proof}
The case $\alpha = 0$ can be checked by hand, so assume $\alpha \ne 0$. Since all vertices of $\bar\alpha$ have degree at least 2,
\[ \delta(\bar\alpha) = \sum_{\substack{v \in V(\bar\alpha) \\ \deg_{\bar\alpha}(v) \ge 3}} \deg_{\bar\alpha}(v) \le \sum_{v \in V(\bar\alpha)} 3(\deg_{\bar\alpha}(v) - 2) = 6(|\bar\alpha| - |V(\bar\alpha)|) = 6(|\alpha| + 1 - |V(\alpha)|), \]
as desired.
\end{proof}
\begin{proof}[Proof of Theorem~\ref{thm:main-wigner}(a) for $m=1$] For any good $\al\neq 0$ with $|V(\al)|=v$ and $|\al|=d$, combining Lemmas~\ref{lem:spiked:cumulant},~\ref{lem:f:upper:bound}, and \ref{lem:upper:delta} yields
\begin{equation}\label{eq:upper:kappa}
|\ka_{\al}|\leq \left(\frac{\lambda}{n}\right)^{(d+1)/2} (2d)^{d-v+1} \cdot M\big(6(d-v+1)\big).
\end{equation}
Thus, recalling  that $\Delta\equiv d-v+1 \ge 0$ for good $\al$,
\begin{align*}
\EE[x^2] \cdot \Corr^2_{\le D} &\le \sum_{\substack{1 \le |\alpha| \le D \\ \alpha \text{ good}}} \kappa_\alpha^2 \\
&\stackrel{(a)}{\le} \sum_{d=1}^D \sum_{v=2}^{d+1} n^{v-2}(2d)^{5(d-v+1)} \left[\left(\frac{\lambda}{n}\right)^{\frac{d+1}{2}} (2d)^{d-v+1} M\big(6(d-v+1)\big)\right]^2 \\
&= n^{-2} \sum_{d=1}^D \lambda^{d+1} \sum_{v=2}^{d+1} n^{-(d-v+1)}(2d)^{7(d-v+1)} M\big(6(d-v+1)\big)^2 \\
&= n^{-2} \sum_{d=1}^D \lambda^{d+1} \sum_{\Delta = 0}^{d-1} n^{-\Delta}(2d)^{7\Delta} M\big(6\Delta\big)^2 \\
&\stackrel{(b)}{\le} n^{-2} \sum_{d=1}^D \lambda^{d+1} \sum_{\Delta = 0}^{d-1} n^{-\Delta}(2D)^{7\Delta} (6c\Delta)^{12\nu\Delta} \\
&\le n^{-2} \left(\sum_{d=1}^D \lambda^{d+1}\right) \sum_{\Delta \ge 0} \left(\frac{(2D)^7 (6cD)^{12\nu}}{n}\right)^\Delta,
\end{align*}
where $(a)$ holds by Eq.~\eqref{eq:upper:kappa} and Lemma~\ref{lem:number:good:graphs:2}, and $(b)$ holds by our assumption that $\E|\pi|^{k}\leq (ck)^{\nu k}$ for all integers $k\geq 1$. The final sum over $\Delta$ is $O(1)$ provided $D\leq n^{\delta}$ where $\delta\equiv \delta(c,\nu)>0$. Finally note that $\EE[x^2] = \lambda m/n$ to complete the proof in the case $m=1$.
\end{proof}

\subsection{General $m$}

\begin{proof}[Proof of Theorem~\ref{thm:main-wigner}(a) for general $m$]
Now write $\kappa^{(m)}_\alpha$ to denote the cumulants for the rank-$m$ case. Above, we have bounded $\kappa^{(1)}_\alpha$.  Recall $\kappa^{(m)}_\alpha$ is the joint cumulant (see Section~2.5 of~\cite{SW-estimation})
\[ \kappa^{(m)}_\alpha = \kappa\left(\sqrt{\frac{\lambda}{n}} (UU^\top)_{ij} \,:\, (i,j) \in E(\bar\alpha)\right) \]
where $U$ is $n \times m$ with entries i.i.d.\ from $\pi$ and $\bar\alpha := \alpha + \one_{(1,2)}$. Joint cumulants are well known to have the following additivity property (e.g.,~\cite[Proposition~2.12]{SW-estimation}).

\begin{proposition}
If $X_1,\ldots,X_k$ and $Y_1,\ldots,Y_k$ are random variables with $\{X_i\}_{i \in [k]}$ independent from $\{Y_i\}_{i \in [k]}$, then
\[ \kappa(X_1+Y_1,\ldots,X_k+Y_k) = \kappa(X_1,\ldots,X_k) + \kappa(Y_1,\ldots,Y_k). \]
\end{proposition}

In our setting, we can write $UU^\top = \tilde{U}\tilde{U}^\top + uu^\top$ where $\tilde{U} \in \RR^{n \times (m-1)}$ contains the first $m-1$ columns of $U$ and $u \in \RR^n$ is the final column. Since $\tilde{U}$ and $u$ are independent from each other,
\begin{align*}
\kappa^{(m)}_\alpha &= \kappa\left(\sqrt{\frac{\lambda}{n}} (\tilde{U}\tilde{U}^\top)_{ij} + \sqrt{\frac{\lambda}{n}}(uu^\top)_{ij} \,:\, (i,j) \in E(\bar\alpha)\right) \\
&= \kappa\left(\sqrt{\frac{\lambda}{n}} (\tilde{U}\tilde{U}^\top)_{ij} \,:\, (i,j) \in E(\bar\alpha)\right) + \kappa\left(\sqrt{\frac{\lambda}{n}}(uu^\top)_{ij} \,:\, (i,j) \in E(\bar\alpha)\right) \\
&= \kappa^{(m-1)}_\alpha + \kappa^{(1)}_\alpha.
\end{align*}
By induction we conclude $\kappa^{(m)}_\alpha = m \cdot \kappa^{(1)}_\alpha$. As a result, our bound on $\EE[x^2] \cdot \Corr_{\le D}^2$ is simply $m^2$ times the bound for the $m=1$ case, namely
\[ \frac{\lambda m}{n} \cdot \Corr_{\le D}^2 \le O(1) \cdot \frac{m^2}{n^2} \sum_{d=1}^D \lambda^{d+1}, \]
completing the proof.
\end{proof}

\section{Stochastic Block Model}
\label{sec:sbm}

This section is devoted to the proof of Theorem~\ref{thm:main-sbm}(a). To begin, we consider a different estimand,
\[
x_{k,\ell}:=(\one_{\sigma^\star_1=k}-\pi_k)(\one_{\sigma^\star_2=\ell}-\pi_\ell), \qquad k,\ell\in [q].
\]
Then we have the following.
\begin{theorem}\label{thm:sbm}
Consider the stochastic block model with parameters $q,\pi,Q$ such that Eq.~\eqref{eq:degree:condition} holds. There exist constants $\delta,C>0$ depending only on $q,\pi,Q$ for which the following holds. If $D\leq n^\delta$ then for every $k_0,\ell_0 \in [q]$ we have
\[
\Corr_{\leq D}(k_0,\ell_0) := \sup_{f\in \R[Y]_{\leq D}} \frac{\EE[f(Y) \cdot x_{k_0,\ell_0}]}{\sqrt{\EE[f(Y)^2]}}\leq \sqrt{\frac{C}{n}\sum_{t=1}^{D}(d\la^2)^{t}}.
\]
\end{theorem}

\noindent We first prove that Theorem~\ref{thm:sbm} implies Theorem~\ref{thm:main-sbm}(a). In the later subsections, we prove Theorem~\ref{thm:sbm}.

\begin{proof}[Proof of Theorem~\ref{thm:main-sbm}(a)]
Note that 
\[
\sum_{k,\ell=1}^{q}(Q_{k,\ell}-d)x_{k,\ell}\equiv \sum_{k,\ell=1}^{q}(Q_{k,\ell}-d) (\one_{\sigma^\star_1=k}-\pi_k) (\one_{\sigma^\star_2=\ell}-\pi_\ell)=Q_{\sigma^\star_1,\sigma^\star_2}-d,
\]
where the last equality holds since $\sum_{\ell=1}^{q}(Q_{k,\ell}-d)\pi_{\ell}=0$ for every $k\in [q]$ by the condition~\eqref{eq:degree:condition}. Thus, it follows that
\begin{align*}
\Corr_{\le D} &\equiv \sup_{f\in \R[Y]_{\leq D}} \frac{\EE[f(Y) \cdot (Q_{\sigma^\star_1,\sigma^\star_2}-d)]}{\sqrt{\EE[f(Y)^2] \cdot \EE[(Q_{\sigma^\star_1,\sigma^\star_2}-d)^2]}}\\
&= \sup_{f\in \R[Y]_{\leq D}} \sum_{k,\ell=1}^q \frac{\EE[f(Y) \cdot (Q_{k,\ell}-d)x_{k,\ell}]}{\sqrt{\EE[f(Y)^2] \cdot \EE[(Q_{\sigma^\star_1,\sigma^\star_2}-d)^2]}} \\
&\le \sum_{k,\ell=1}^q |Q_{k,\ell}-d| \sup_{f\in \R[Y]_{\leq D}} \frac{\EE[f(Y) \cdot x_{k,\ell}]}{\sqrt{\EE[f(Y)^2] \cdot \EE[(Q_{\sigma^\star_1,\sigma^\star_2}-d)^2]}} \\
&= \sum_{k,\ell=1}^q \frac{|Q_{k,\ell}-d|}{\EE[(Q_{\sigma^\star_1,\sigma^\star_2}-d)^2]} \cdot \Corr_{\le D}(k,\ell) \\
&\le O(1) \cdot \sqrt{\frac{C}{n}\sum_{t=1}^{D}(d\la^2)^t},
\end{align*}
where the last step used Theorem~\ref{thm:sbm}.
\end{proof}

\noindent The rest of this section is devoted to proving Theorem~\ref{thm:sbm}.

\subsection{Orthonormal Basis}
For a graph $\al\in \{0,1\}^{\bcn}$, choose 
\[
\phi_{\al}(Y):=\left(Y-\frac{d}{n}\right)^{\al}.
\]
Consider $\be\in \{0,1\}^{\bcn}$ and $\ga \in [q]^{V(\beta)}$ the labels of the vertices of $\beta$. For $k,\ell\in[q]$, denote the set of edges between community $k$ and community $\ell$ by\[
E_{k,\ell}(\beta, \sigma^\star)=\{(i,j)\in E(\beta) \,:\, \sigma^\star_i = k, \, \sigma^\star_j = \ell\}
\]
where, recall, $i < j$ by convention. We encode $Y$ as follows. Let $X\equiv \left(X_{ij}(k,\ell)\right)_{i<j, k,\ell\in [q]}$ have independent entries with distributed according to \[
X_{ij}(k,\ell) \sim \Ber\left(\frac{Q_{k,\ell}}{n}\right).
\]
Then, we set $Y_{ij}=X_{ij}(\sigma^\star_i,\sigma^\star_j)$. Denoting $\sigma^\star_{V(\beta)}:=(\sigma^\star_i)_{i\in V(\beta)}$, consider
\[
\begin{split}
\psi_{\be\ga}(X,\sigma^\star)
:=& \one(\sigma^\star_{V(\be)}=\ga)\cdot \prod_{i\in V(\be)}\pi_{\ga_i}^{-1/2}\prod_{(i,j)\in \beta}\frac{Y_{ij}-\frac{Q_{\ga_i,\ga_j}}{n}}{\sqrt{\frac{Q_{\ga_i,\ga_j}}{n}(1-\frac{Q_{\ga_i,\ga_j}}{n})}}\\
=& \one(\sigma^\star_{V(\be)}=\ga)\cdot \prod_{i\in V(\be)}\pi_{\ga_i}^{-1/2}\prod_{k,\ell\in[q]} \; \prod_{(i,j)\in E_{k,\ell}(\beta,\sigma^\star)}\frac{X_{ij}(k,\ell)-\frac{Q_{k,\ell}}{n}}{\sqrt{\frac{Q_{k,\ell}}{n}(1-\frac{Q_{k,\ell}}{n})}}.
\end{split}
\]
We first prove that $\{\psi_{\be\ga}\}$ is orthonormal.
\begin{lemma}
    $\{\psi_{\be\ga} \,:\, \be\in \{0,1\}^{\bcn}, \, \ga \in [q]^{V(\be)}\}$ is orthonormal.
\end{lemma}
\begin{proof}
    Consider $\be,\be^\prime\in \{0,1\}^{\bcn}$ and $\ga\in [q]^{V(\be)}, \ga^\prime \in [q]^{V(\be^\prime)}$. Since $\E[X_{i,j}(k,\ell)-Q_{k,\ell}/n]=0$ and $\E[(X_{i,j}(k,\ell)-Q_{k,\ell}/n)^2]=(Q_{k,\ell}/n)(1-Q_{k,\ell}/n)$,
    \[
    \begin{split}
    &\E\left[\prod_{k,\ell\in[q]}\left\{\prod_{(i,j)\in E_{k,\ell}(\beta,\sigma^\star)}\frac{X_{ij}(k,\ell)-\frac{Q_{k,\ell}}{n}}{\sqrt{\frac{Q_{k,\ell}}{n}(1-\frac{Q_{k,\ell}}{n})}}\prod_{(i,j)\in E_{k,\ell}(\beta^\prime,\sigma^\star)}\frac{X_{ij}(k,\ell)-\frac{Q_{k,\ell}}{n}}{\sqrt{\frac{Q_{k,\ell}}{n}(1-\frac{Q_{k,\ell}}{n})}}\right\}\BBgiven \sigma^\star\right]\\
    &\quad=\one\{E_{k,\ell}(\beta,\sigma^\star)=E_{k,\ell}(\beta^\prime,\sigma^\star),~\forall k,\ell\in [q]\}.
    \end{split}
    \]
    Note that the last indicator equals $\one_{\beta=\beta^\prime}$.
    We thus have by the tower property that
    \[
    \E [\psi_{\be\ga}\psi_{\be^\prime \ga^\prime}]=\one_{\beta=\beta^\prime}\prod_{i\in V(\be)}\pi_{\ga_i}^{-1/2}\prod_{i\in V(\be^\prime)}\pi_{\ga^\prime_i}^{-1/2}\;\P(\sigma^\star_{V(\beta)}=\gamma,\, \sigma^\star_{V(\beta^\prime)}=\gamma^\prime)=\one_{\beta=\beta^\prime, \ga=\ga^\prime},
    \]
    which concludes the proof.
\end{proof}

\subsection{Removing ``Bad'' Terms}

Our next step is to apply Lemma~\ref{lem:disconn}. We define $\hat\cI$ and $\hat\cJ$ as follows.
\begin{definition}[``Good'' graphs for SBM]
\label{def:good:SBM}
    We define the set $\cIsbm\subseteq \{0,1\}^{\bcn}$ of ``good'' $\alpha$ for the SBM as follows. A graph $\alpha$ is considered ``good'' if either $\alpha = \empty$ or if $\alpha$ satisfies all of the following:
    \begin{itemize}
        \item $1,2 \in V(\alpha)$,
        \item $\bar\alpha := \alpha + \one_{(1,2)}$ is connected,
        \item for every $v \in V(\alpha)$, $\deg_{\bar\alpha}(v) \ge 2$.
    \end{itemize}
    We define $\cJsbm:=\{(\be,\ga) \,:\, \be\in \cIsbm,\, \ga\in [q]^{V(\be)}\}$ with the convention that if $\be=\emptyset$ then $\ga=\emptyset$.
\end{definition}

\begin{remark}\label{rmk:sbm}
Observe that the definition for $\cIsbm$ above is the same as the ``good'' graphs for spiked Wigner (Definition~\ref{def:good:SW}) except we restrict to simple graphs in the current SBM setting. In particular, the function $f(\al)$ from Lemma~\ref{lem:spiked:cumulant} is well-defined for $\al\in \cIsbm$, and satisfies the bound $f(\al)\leq (2|\al|)^{|\al|-|V(\al)|+1}$ from Lemma~\ref{lem:f:upper:bound}. Also, for $t \ge 1$ and $v \ge 2$, the number of $\al \in \cIsbm$ with $|\al|=t$ and $|V(\al)|=v$ is at most $n^{v-2}(2t)^{5(t-v+1)}$ by Lemma~\ref{lem:number:good:graphs:2}.
\end{remark}

We now verify the conditions of Lemma~\ref{lem:disconn} in the following proposition.
\begin{proposition}\label{prop:good:sbm}
For each $\al\notin \cIsbm$ there exist $\hat\al\in \cIsbm$ and $\mu \in \R$ such that $c_{\al}=\mu c_{\hat\al}$ and $M_{\be\ga,\al}=\mu M_{\be\ga,\hat\al}$ for all $(\be,\ga) \in \cJsbm$.
\end{proposition}

\noindent To prove Proposition~\ref{prop:good:sbm}, we first prove the following lemma.

\begin{lemma}\label{lem:sbm:simple}
Let $\al, \be\in \{0,1\}^{\bcn}$.
\begin{enumerate}
\item [(a)]For $W\subseteq V(\al)$, suppose there exists $i_0\in V(\al)\setminus W$ such that $\deg_{\al}(i_0)=1$. Denote $X^{-}:=(X_{ij})_{\{i,j\}\neq \{i_0,j_0\}}$ where $j_0\in V(\alpha)$ is the unique neighbor of $i_0$ in $\alpha$. Then, $\E[\phi_{\al}(Y)\given \sigma^\star_W, X^{-}]=0$ holds almost surely.
\item [(b)]$M_{\be\ga,\al}=0$ holds whenever $\beta \nleq \al$.
\end{enumerate}
\end{lemma}

\begin{proof}
For the first item, it suffices to consider the case $W=V(\alpha)\setminus \{i_0\}$ by the tower property. Note that
\[
\E[\phi_{\al}(Y)\given \sigma^\star_W,X^{-}]=\phi_{\al-\one_{(i_0,j_0)}}(Y)\E[Y_{i_0j_0}-d/n\given \sigma^\star_W, X^{-}].
\]
Since $Y_{i_0j_0}$ is measurable w.r.t.\ $(\sigma^\star_{i_0}, \sigma^\star_{j_0}, X_{i_0j_0})$, we have
\[
\E[Y_{i_0j_0}-d/n\given \sigma^\star_W, X^{-}]=\E[Y_{i_0j_0}-d/n\given \sigma^\star_{j_0}]=\frac{1}{n}\sum_{\ell=1}^{q}(Q_{\sigma^\star_{j_0},\ell}-d)\pi_{\ell}=0,
\]
where the last equality holds by our assumption~\eqref{eq:degree:condition}. Thus, $\E[\phi_{\al}(Y)\given \sigma^\star_W,X^{-}]=0$ holds.

For the second item, suppose that $(i,j)\in E(\beta)\setminus E(\al)$. Denote $X^{-}=(X_{i'j'})_{\{i',j'\}\neq \{i,j\}}$ as before, and let $\psi^{-}_{\be\ga}=\psi_{\be\ga}/\Big(X_{ij}(\ga_i,\ga_j)-\frac{Q_{\ga_i,\ga_j}}{n}/\sqrt{\frac{Q_{\ga_i,\ga_j}}{n}(1-\frac{Q_{\ga_i,\ga_j}}{n})}\Big)$. Then, $\phi_{\al}$ and $\psi^{-}_{\be\ga}$ are measurable w.r.t.\ $(\sigma^\star, X^{-})$, so
\[
\E[\phi_{\al}\psi_{\beta\gamma}\given \sigma^\star,X^{-}]=\phi_{\al}\psi^{-}_{\be\ga}\cdot \E\bigg[\frac{X_{ij}(\ga_i,\ga_j)-\frac{Q_{\ga_i,\ga_j}}{n}}{\sqrt{\frac{Q_{\ga_i,\ga_j}}{n}(1-\frac{Q_{\ga_i,\ga_j}}{n})}}\bbgiven \sigma^\star_i,\sigma^\star_j\bigg]=0.
\]
Therefore $M_{\be\ga,\al}=\E[\phi_{\alpha}\psi_{\be\ga}]=0$, concluding the proof.
\end{proof}

\begin{proof}[Proof of Proposition~\ref{prop:good:sbm}]
    Consider $\al\notin \cIsbm$. We divide into three cases. First, consider the case where $\al$ does not contain both vertices $1$ and $2$. By setting $\hat\al=\emptyset$, Lemma~\ref{lem:sbm:simple} shows that $M_{\be\ga,\al}=0=M_{\be\ga,\hat\al}$ for all $\be\neq \emptyset$. Moreover, if we assume $2\notin V(\al)$, then 
    \[
    c_{\al}= \E[\phi_{\al}(\one_{\sigma^\star_1=k_0}-\pi_{k_0})]\E[\one_{\sigma^\star_2=\ell_0}-\pi_{\ell_0}]=0.
    \]
    Thus, $c_{\al}=\mu c_{\emptyset}$ and $M_{\be\ga,\al}=\mu M_{\be\ga,\emptyset}$ for $\be\in \cIsbm, \ga\in [q]^{V(\be)}$ are trivially satisfied by setting $\mu=M_{\emptyset,\al}/M_{\emptyset,\emptyset}$. 

    Next, consider the case where there exists $i_0\neq 1,2$ such that $\deg_{\al}(i_0)=1$. We claim that $M_{\be\ga,\al}=0$ holds for $\be\in \cIsbm$ and $\ga\in [q]^{V(\be)}$. Indeed, this is true for $\be\nleq \al$ by Lemma~\ref{lem:sbm:simple}(b). For $\be\leq \al$, note that since $\be\in \cIsbm$, we must have $i_0\notin V(\beta)$. Thus, by Lemma~\ref{lem:sbm:simple}(a),  
    \[
    \E[\phi_{\al}\psi_{\be\ga}\given \sigma^\star_{V(\alpha)\setminus\{i_0\}}, X^{-}]=\psi_{\be\ga}\E[\phi_{\al}\given \sigma^\star_{V(\alpha)\setminus\{i_0\}}, X^{-}]=0,
    \]
    which yields $M_{\be\ga,\al}\equiv \E[\phi_{\al}\psi_{\be\ga}]=0$. Similarly, we have that $c_{\al}=0$ in this case. Consequently, setting $\mu=0$ gives the desired conclusion.

    Finally, consider the case where $1,2\in V(\al)$ and $\deg_{\alpha}(i)\geq 2$ holds for all $i\in V(\alpha)\setminus\{1,2\}$, but $\al+\bone_{(1,2)}$ is not connected. For $i=1,2$, let $\al_i$ denote the connected component of $\al$ containing vertex $i$, and set $\hat \al=\al_1\cup \al_2$. Then we have $\hat\al \in \cIsbm$. Also since $V(\al) \cap V(\al-\hat\al)=\emptyset$ holds, setting $\mu=\E[\phi_{\al-\hat\al}]$ gives the desired conclusion. 
\end{proof}

\subsection{Constructing $u$}
\label{subsec:SBM:u}
In light of Lemma~\ref{lem:disconn} and Proposition~\ref{prop:good:sbm}, our goal is to construct $u=(u_{\be\ga})_{\be\ga\in \cJsbm}$ such that 
\begin{equation*}
    \sum_{\be\ga \in \cJsbm} u_{\be\ga}M_{\be\ga,\al}=c_{\al} \qquad \forall\al\in \cIsbm
\end{equation*}
and such that $\|u\|$ is ``small enough.'' As opposed to the planted submatrix and planted dense subgraph models, where we constructed $u$ explicitly, we will construct $u$ implicitly by the following recursion:
\begin{enumerate}
    \item Set $u_{\empty\empty}=0$. 
    \item Having  determined $u_{\be\ga}$ for all $\be\in \cIsbm$ such that $|\be|<|\al|$, set
    \begin{equation}\label{eq:def:d}
   d_{\al}:=c_{\al}-\sum_{\be\ga\in \cJsbm:\be\lneq\al}u_{\be\ga}M_{\be\ga,\al}
    \end{equation}
    and
\begin{equation}\label{eq:sbm:u}
    u_{\al\ga}=\frac{M_{\al\ga,\al}}{\sum_{\ga'\in [q]^{V(\al)}}M_{\al \ga', \al}^2}d_{\al},~~~\ga\in [q]^{V(\al)}.
    \end{equation}
\end{enumerate}
The intuition behind this choice of $u_{\al\ga}$ is that given $(u_{\be\ga})_{\be\ga\in \cJsbm: |\be|<|\al|}$, the choice \eqref{eq:sbm:u} minimizes $\sum_{\ga\in [q]^{V(\al)}}u_{\al\ga}^2$ by the Cauchy--Schwarz inequality. Consequently, by Proposition~\ref{prop:master},
\begin{equation}\label{eq:sbm:corr:bound}
\Corr_{\leq D}(k_0,\ell_0)^2 \leq \|u\|^2=\sum_{\al\in \cIsbm:|\al|\leq D} \frac{d_\al^2}{\sum_{\ga\in [q]^{V(\al)}}M_{\al\ga,\al}^2}.
\end{equation}

\subsection{Upper bound for $\|u\|$}
Having \eqref{eq:sbm:corr:bound} in hand, we upper bound $d_{\al}^2$ and lower bound $\sum_{\ga\in [q]^{V(\al)}}M_{\al\ga,\al}^2$ separately. 

One of the challenges in upper bounding $|d_{\al}|$ is that unlike the other models we've considered, $c_{\al}$ and $M_{\be\ga,\al}$ do not admit simple expressions in terms of $|\al|, |\be|, |V(\al)|, |V(\be)|$. Rather, they depend on the entire topology of $\al$ and $\be$. Nevertheless, we show that the main contribution in the sum $\sum_{\al,\ga}u_{\al\ga}^2$ is from $\alpha$ that is a path connecting $1$ and $2$, and that $u_{\al\ga}$ for such path $\al$ can be explicitly calculated. In doing so, the following symmetric matrix $B=(B_{k,\ell})_{k,\ell \in [q]} \in \R^{q\times q}$ plays a central role:
\[
B:=\diag(\pi)^{-1/2}(T-\pi\bone^\top) \, \diag(\pi)^{1/2}=\frac{1}{d}\diag(\pi)^{1/2}Q \, \diag(\pi)^{1/2}-\sqrt{\pi}\sqrt{\pi}^\top.
\]
Observe that the eigenvalues of $B$ are $\lambda_2(T),\ldots, \lambda_q(T)$, so we have for each $t\geq 1$,
\begin{equation}\label{eq:sbm:simple}
\|B^t\|_{\infty}\equiv \max_{k,\ell}\big|(B^t)_{k,\ell}\big|\leq \|B^t\|_{\op}\leq \|B\|_{\op}^t=\lambda^t,
\end{equation}
which will be used multiple times throughout the proofs.
\begin{lemma}\label{lem:sbm:c:path}
    Let $\al$ be a (simple) path from vertex $1$ to vertex $2$ with length $t\geq 1$. Then,
    \[
    c_{\al}=\left(\frac{d}{n}\right)^{t}(B^t)_{k_0,\ell_0}\cdot \sqrt{\pi_{k_0}}\sqrt{\pi_{\ell_0}}.
    \]
\end{lemma}
\begin{proof}
Let $1=i_0,i_1,\ldots, i_{t-1},i_t=2$ denote the vertices of $\al$ in consecutive order. Since $\deg_{\al}(i)=1$ for $i\in \{1,2\}$, Lemma~\ref{lem:sbm:simple} shows that $\E[\phi_{\al}\given \sigma^\star_i]=0$ holds for $i\in \{1,2\}$. It follows that
   \[
   \E[\phi_{\al}]=\E[\phi_{\al}\one_{\sigma^\star_{1}=k_0}]=\E[\phi_{\al}\one_{\sigma^\star_{2}=\ell_0}]=0.
   \]
   Hence, we have
   \[
   c_{\al}\equiv \E[\phi_{\al} \cdot(\one_{\sigma^\star_1=k_0}-\pi_{k_0})(\one_{\sigma^\star_1=\ell_0}-\pi_{\ell_0})]=\E[\phi_{\al}\one_{\sigma^\star_1=k_0,\sigma^\star_2=\ell_0}].
   \]
   We can compute the RHS by first conditioning on $\sigma^\star$, which yields
   \[
   c_{\al}=\left(\frac{d}{n}\right)^{t}\E\left[\one_{\sigma^\star_1=k_0,\sigma^\star_2=\ell_0}\prod_{s=1}^{t} \left(\frac{Q_{\sigma^\star_{i_{s-1}}, \sigma^\star_{i_{s}}}}{d}-1\right)\right]=\left(\frac{d}{n}\right)^{t}(B^t)_{k_0,\ell_0}\cdot \sqrt{\pi_{k_0}}\sqrt{\pi_{\ell_0}}.
   \]
\end{proof}

Thus, as a consequence of Lemma~\ref{lem:sbm:c:path} and Eq.~\eqref{eq:sbm:simple}, if $\al$ is a path from $1$ to $2$,
\[
|c_{\al}|\leq \left(\frac{d\la}{n}\right)^{|\al|}.
\]
The next lemma shows that for cyclic $\al$'s, the values $|c_{\al}|$ satisfy a similar bound with an extra term that depends on the number of tree excess edges.

\begin{lemma}\label{lem:sbm:c:cycle}
For $\al \in \cIsbm$, we have
\[
|c_{\al}|\leq \left(\frac{d\la}{n}\right)^{|\al|} \cdot \left(\frac{q}{\pi_{\min}}\right)^{2(|\al|-|V(\al)|+2)},
\]
where $\pi_{\min} := \min_{k\in [q]}\pi_k$.
\end{lemma}

\noindent Lemma~\ref{lem:sbm:c:cycle} is a consequence of the following crucial estimate whose proof is deferred to Section~\ref{subsec:prop:sbm:crucial:estimate}.
\begin{proposition}\label{prop:sbm:crucial:estimate}
Let $\al\in \{0,1\}^{\bcn}$ be any non-empty graph. Consider a non-empty vertex set $W\subseteq V(\al)$ that intersects with every connected component of $\al$. Then, we have
\begin{equation}\label{eq:prop:sbm:crucial:estimate}
\max_{\tau\in [q]^W}\left|\E\big[\phi_{\al}(Y)\bgiven \sigma^\star_W=\tau\big]\right|\leq \left(\frac{d\la}{n}\right)^{|\al|}\left(\frac{q}{\pi_{\min}}\right)^{2(|\al|-|V(\al)|+|W|)}.
\end{equation}
\end{proposition}

\begin{proof}[Proof of Lemma~\ref{lem:sbm:c:cycle}]
Note that since $x_{k_0,\ell_0}\equiv (\one_{\sigma^\star_1=k_0}-\pi_{k_0})(\one_{\sigma^\star_2=\ell_0}-\pi_{\ell_0})$ is $(\sigma^\star_1,\sigma^\star_2)$-measurable and $|x_{k_0,\ell_0}|\leq 1$, we have $|c_{\al}|\leq \E\left|\E[\phi_{\al}(Y)\given \sigma^\star_1,\sigma^\star_2]\right|$. Moreover, since $\al\in \cIsbm$, the vertex set $\{1,2\}\subseteq V(\al)$ intersects with every connected component of $\al$. Consequently, applying Proposition~\ref{prop:sbm:crucial:estimate} with $W=\{1,2\}$ yields the desired conclusion.  
\end{proof}
As a next step, we show that $M_{\be\ga,\al}$ can be bounded in terms of $M_{\be\ga,\be}$. Here, Proposition~\ref{prop:sbm:crucial:estimate} again plays an important role.
\begin{lemma}\label{lem:sbm:bound:M}
Consider $\al,\be\in \cIsbm$ such that $\be\le \al$. We have that $M_{\be\ga,\be}\geq 0$ and
\[
|M_{\be\ga,\al}|\le M_{\be\ga,\be}\cdot \left(\frac{d\la}{n}\right)^{|\al|-|\be|}  \left(\frac{q}{\pi_{\min}}\right)^{2(|\al|-|\be|-|V(\al)|+|V(\be)|)}.
\]
\end{lemma}
\begin{proof}
The statement is trivial for $\be=\al$, so assume $\be\lneq \al$. We first compute $M_{\be\ga,\be}$. Note that
\[
\E\left[\phi_{\be}\psi_{\be\ga, \be}\given \sigma^\star_{V(\be)}\right]= \one_{\sigma^\star_{V(\be)}=\ga}\prod_{i\in V(\beta)}\pi_{\gamma_i}^{-1/2}\prod_{(i,j)\in \beta}\E\left[\left(X_{ij}(\ga_i,\ga_j)-\frac{d}{n}\right)\cdot \frac{X_{ij}(\ga_i,\ga_j)-\frac{Q_{\ga_i,\ga_j}}{n}}{\sqrt{\frac{Q_{\ga_i,\ga_j}}{n}(1-\frac{Q_{\ga_i,\ga_j}}{n})}}\right].
\]
Since $X_{ij}(\ga_i,\ga_j)\sim \Ber(Q_{\ga_i,\ga_j}/n)$, we can compute the rightmost expectation as
\begin{equation}\label{eq:sbm:M:self}
\E\left[\phi_{\be}\psi_{\be\ga, \be}\given \sigma^\star_{V(\be)}\right]= \one_{\sigma^\star_{V(\be)}=\ga}\prod_{i\in V(\beta)}\pi_{\gamma_i}^{-1/2}\prod_{(i,j)\in \beta}\sqrt{\frac{Q_{\ga_i,\ga_j}}{n}\left(1-\frac{Q_{\ga_i,\ga_j}}{n}\right)}\,.
\end{equation}
Taking expectation shows that $M_{\be\ga,\be}\geq 0$ holds. We next bound $|M_{\be\ga,\al}|$:
\begin{equation*}
M_{\be\ga,\al}=\E\Big[\E\big[\phi_{\al-\be}\phi_{\be}\psi_{\be\ga}\bgiven \sigma^\star_{V(\be)}\big]\Big]=\E\Big[\E\big[\phi_{\al-\be}\bgiven \sigma^\star_{V(\be)}\big]\cdot\E\big[\phi_{\be}\psi_{\be\ga}\bgiven \sigma^\star_{V(\be)}\big]\Big],
\end{equation*}
where the last equality holds because $\phi_{\al-\be}$ is conditionally independent from $\phi_{\be}\psi_{\be\ga}$ given $\sigma^\star_{V(\be)}$. Note that letting $W=V(\al-\be)\cap V(\be)$, we have $\E[\phi_{\al-\be}\given \sigma^\star_{V(\be)}]=\E[\phi_{\al-\be}\given \sigma^\star_{W}]$. Thus, it follows that
\begin{equation}\label{eq:sbm:technical}
\begin{split}
|M_{\be\ga,\al}|
&\leq \E\Big[\big|\E\big[\phi_{\al-\be}\bgiven \sigma^\star_{W}\big]\big|\cdot\E\big[\phi_{\be}\psi_{\be\ga}\bgiven \sigma^\star_{V(\be)}\big]\Big]\\
&\leq M_{\be\ga,\be}\cdot \max_{\tau\in [q]^W}\big|\E\big[\phi_{\al-\be}\bgiven \sigma^\star_{W}=\tau\big]\big|,
\end{split}
\end{equation}
where in the first inequality, we used the fact $\E[\phi_{\be}\psi_{\be\ga,\be}\given \sigma^\star_{V(\be)}]\geq 0$ a.s., which follows from~\eqref{eq:sbm:M:self}.

We next claim that $(\al-\be,W)$ satisfies the assumption of Proposition~\ref{prop:sbm:crucial:estimate}. Suppose by contradiction that there exists a connected component $\al'$ of $\al-\be$ such that $V(\al')\cap W=\empty$. Then $V(\al')\cap V(\be)=\empty$, so $\al'$ is a connected component of $\al$ not intersecting with $\be$. Since $\al,\be\in \cIsbm$ has at most $2$ connected components, this implies that $\al =\al'\sqcup \be$. Since $1,2\in V(\be)$, $\al'$ is a connected component of $\al$ not containing $1$ nor $2$, which contradicts that $\al\in \cIsbm$.

Consequently, we can apply Proposition~\ref{prop:sbm:crucial:estimate}, which yields
\begin{equation}\label{eq:sbm:not:important}
\max_{\tau\in [q]^W}\big|\E\big[\phi_{\al-\be}\bgiven \sigma^\star_{W}=\tau\big]\big|\leq \left(\frac{d\la}{n}\right)^{|\al-\be|}  \left(\frac{q}{\pi_{\min}}\right)^{2(|\al-\be|-|V(\al-\be)|+|W|)}.
\end{equation}
Note that $|\al-\be|=|\al|-|\be|$ and since $V(\al-\be)\setminus V(\be)=V(\al)\setminus V(\be)$, we have 
\[
|V(\al-\be)|=|V(\al-\be)\setminus V(\be)|+|W|=|V(\al)|-|V(\be)|+|W|.
\]
Therefore, combining Eq.~\eqref{eq:sbm:technical} and Eq.~\eqref{eq:sbm:not:important} implies the desired conclusion.
\end{proof}
\noindent We are now ready to bound $|d_{\al}|$. Recall the quantity $f(\al)$ from Lemma~\ref{lem:spiked:cumulant} (see also Remark~\ref{rmk:sbm}).
\begin{lemma}\label{lem:sbm:d:alpha}
    For a non-empty $\al\in \cIsbm$, we have
    \[
    |d_{\al}|\leq \left(\frac{d\la}{n}\right)^{|\al|}f(\al) \cdot \left(\frac{q}{\pi_{\min}}\right)^{2(|\al|-|V(\al)|+2)}.
    \]
\end{lemma}
\begin{proof}
    We proceed by strong induction on $|\al|$. If $|\al|=1$ and $\al\in \cIsbm$, then $\al=\bone_{(1,2)}$. In this case, $d_{\al}=c_{\al}$ by definition, and Lemma~\ref{lem:sbm:c:path} shows that $|d_{\al}|\leq \frac{d\la}{n}$. Since $f(\bone_{(1,2)})=1$, the conclusion holds for the base case $|\al|=1$. For the induction step, fix a good $\al$ and assume the conclusion holds for all $\beta\in \cIsbm$ such that $\be \lneq \al$. Recalling the definition of $d_{\al}$ in \eqref{eq:def:d}, we have
    \begin{equation}\label{eq:bound:d}
    |d_{\al}|\leq |c_{\al}|+\sum_{\be\ga\in \cJsbm:\be\lneq\al}|u_{\be\ga}|\cdot|M_{\be\ga,\al}|.
    \end{equation}
    First, we bound the sum in \eqref{eq:bound:d}. Since $u_{\empty\empty}=0$, we can restrict the sum to $\be\neq \empty$. Using Lemma~\ref{lem:sbm:bound:M}, 
    \[
    \sum_{\be\ga\in \cJsbm:\be\lneq\al}|u_{\be\ga}|\cdot|M_{\be\ga,\al}|\leq \sum_{\substack{\be\ga\in \cJsbm\\\empty \lneq \be\lneq\al}}|u_{\be\ga}| \cdot M_{\be\ga,\be}\cdot \left(\frac{d\la}{n}\right)^{|\al|-|\be|}  \left(\frac{q}{\pi_{\min}}\right)^{2(|\al|-|\be|-|V(\al)|+|V(\be)|)}.
    \]
    Note that for a fixed $\beta\in \cIsbm$ such that $\empty\lneq \beta\lneq \al$, our choice of $u_{\be\ga}$ in Eq.~\eqref{eq:sbm:u} dictates
    \[
    \sum_{\ga\in [q]^{V(\be)}}|u_{\be\ga}| \cdot M_{\be\ga,\be}=|d_{\be}|\leq \left(\frac{d\la}{n}\right)^{|\be|}f(\be) \cdot \left(\frac{q}{\pi_{\min}}\right)^{2(|\be|-|V(\be)|+2)},
    \]
    where the last inequality holds by the induction hypothesis. Thus, it follows that
    \[
      \sum_{\be\ga\in \cJsbm:\be\lneq\al}|u_{\be\ga}|\cdot|M_{\be\ga,\al}|\leq\left(\frac{d\la}{n}\right)^{|\al|}\left(\frac{q}{\pi_{\min}}\right)^{2(|\al|-|V(\al)|+2)} \sum_{\be\in \cIsbm: \empty\lneq \beta\lneq \al} f(\be).
    \]
    Plugging this estimate into \eqref{eq:bound:d} and using Lemma~\ref{lem:sbm:c:cycle} to bound $|c_{\al}|$, it follows that
    \[
    \begin{split}
    |d_{\al}|
    &\leq \left(\frac{d\la}{n}\right)^{|\al|}\left(\frac{q}{\pi_{\min}}\right)^{2(|\al|-|V(\al)|+2)} \left(1+\sum_{\be\in \cIsbm: \empty\lneq \beta\lneq \al} f(\be)\right)\\
    &=\left(\frac{d\la}{n}\right)^{|\al|}\left(\frac{q}{\pi_{\min}}\right)^{2(|\al|-|V(\al)|+2)} f(\al),
    \end{split}
    \]
    which concludes the proof.
\end{proof}

We next lower bound $\sum_{\ga\in [q]^{V(\al)}}M_{\al\ga,\al}^2$. By taking expectations in Eq.~\eqref{eq:sbm:M:self}, we have
\[
M_{\al\ga,\al}=\prod_{i \in V(\be)}\pi_{\ga_i}^{1/2}\prod_{(i,j)\in \al}\sqrt{\frac{Q_{\ga_i,\ga_j}}{n}\left(1-\frac{Q_{\ga_i,\ga_j}}{n}\right)}.
\]
Thus, we can express
\[
\sum_{\ga\in [q]^{V(\al)}}M_{\al\ga,\al}^2=\E\bigg[ \prod_{(i,j)\in\al} \bigg\{\frac{Q_{\sigma^\star_i,\sigma^\star_j}}{n}\bigg(1-\frac{Q_{\sigma^\star_i,\sigma^\star_j}}{n}\bigg)\bigg\}\bigg].
\]
Writing $\|Q\|_{\infty}:=\max_{k,\ell\in [q]}|Q_{k,\ell}|$, we lower bound
\begin{equation}\label{eq:sbm:M:lower:bound}
\sum_{\ga\in [q]^{V(\al)}}M_{\al\ga,\al}^2\geq \left(1-\frac{\|Q\|_{\infty}}{n}\right)^{|\al|} \E\bigg[\prod_{(i,j)\in \al}\frac{Q_{\sigma^\star_i,\sigma^\star_j}}{n}\bigg].
\end{equation}
We first calculate the above expected value when $\al$ is a tree.
\begin{lemma}\label{lem:sbm:M:tree}
    For any tree $\al\in \{0,1\}^{\bcn}$,
    \[
    \E\Big[\prod_{(i,j)\in \al}Q_{\sigma^\star_i,\sigma^\star_j}\Big]=d^{|\al|}.
    \]
\end{lemma}
\begin{proof}
    Consider a leaf of the tree $v\in V(\al)$ and let $e=(u,v)\in \al$ be the adjacent edge. Denote the community labels except $\sigma^\star_{v}$ by $\sigma^\star_{\sm v}:=(\sigma^\star_i)_{i\neq v}$. By revealing $\sigma^\star_{-v}$, we have
    \[
     \E\Big[\prod_{(i,j)\in \al}Q_{\sigma^\star_i,\sigma^\star_j}\Bgiven \sigma^\star_{\sm v}\Big]=\E\Big[\prod_{(i,j)\in \al-\bone_e}Q_{\sigma^\star_i,\sigma^\star_j}\Bgiven \sigma^\star_{\sm v}\Big]\cdot \E\big[Q_{\sigma^\star_u,\sigma^\star_v}\bgiven \sigma^\star_u\big]. 
    \]
    Recall that we assumed $\sum_{\ell=1}^{q}Q_{k,\ell}\pi_{\ell}=d$ for any $k\in [q]$ (see Eq.~\eqref{eq:degree:condition}), so $\E\big[Q_{\sigma^\star_u,\sigma^\star_v}\bgiven \sigma^\star_u\big]=d$ holds almost surely. Hence, taking expectation in the display above over $\sigma^\star_{\sm v}$ yields
    \[
      \E\Big[\prod_{(i,j)\in \al}Q_{\sigma^\star_i,\sigma^\star_j}\Big]=d\cdot \E\Big[\prod_{(i,j)\in \al-\bone_e}Q_{\sigma^\star_i,\sigma^\star_j}\Big].
    \]
    Since $\al-\bone_{e}$ is another tree with one fewer edge, the desired conclusion holds by induction.
\end{proof}
With Lemma~\ref{lem:sbm:M:tree} in hand, we lower bound $\sum_{\ga} M_{\al\ga,\al}^2$ as follows.
\begin{lemma}\label{lem:sbm:M:lower:bound}
Consider any non-empty graph $\alpha \in \{0,1\}^{\bcn}$, and let $K$ be its number of connected components. Writing $Q_{\min}:=\min_{k,\ell\in [q]}Q_{k,\ell}>0$, we have
\[
\sum_{\ga\in [q]^{V(\al)}}M_{\al\ga,\al}^2 \geq \left(\frac{d}{n}\right)^{|\al|}\left(1-\frac{\|Q\|_{\infty}}{n}\right)^{|\al|} \left(\frac{Q_{\min}}{d}\right)^{|\al|-|V(\al)|+K}.
\]
\begin{proof}
    Let $\al_1,\ldots, \al_K$ be the connected components of $\al$. Recalling the lower bound in Eq.~\eqref{eq:sbm:M:lower:bound}, we have by the independence of $\{\sigma^\star_{V(\al_i)} \,:\, i \in [K]\}$ that
    \[
    \sum_{\ga\in [q]^{V(\al)}}M_{\al\ga,\al}^2 \geq \left(\frac{d}{n}\right)^{|\al|}\left(1-\frac{\|Q\|_{\infty}}{n}\right)^{|\al|} \,\prod_{i=1}^{K} \bigg\{\E\bigg[\prod_{(u,v)\in \al_i}\frac{Q_{\sigma^\star_u,\sigma^\star_v}}{d}\bigg]\bigg\}.
    \]
    For $i\in [K]$, let $\hat{\alpha}_i$ be a spanning tree of $\alpha_i$. For $(u,v)\in \alpha_i-\hat{\alpha}_i$, we can crudely lower bound $Q_{\sigma^\star_u,\sigma^\star_v}\geq Q_{\min}$, which yields
    \[
    \E\bigg[\prod_{(u,v)\in \al_i}\frac{Q_{\sigma^\star_u,\sigma^\star_v}}{d}\bigg]\geq \left(\frac{Q_{\min}}{d}\right)^{|\al_i|-|\hat{\al}_i|} \E\bigg[\prod_{(u,v)\in \hat\al_i}\frac{Q_{\sigma^\star_u,\sigma^\star_v}}{d}\bigg]=\left(\frac{Q_{\min}}{d}\right)^{|\al_i|-|V(\al_i)|+1}.
    \]
    The equality holds because $|\hat\al_i|=|V(\hat\al_i)|-1=|V(\al_i)|-1$ and $\E\prod_{(u,v)\in \hat\al_i}Q_{\sigma^\star_u,\sigma^\star_v}=d^{|\hat\al_i|}$ by Lemma~\ref{lem:sbm:M:tree}. Combining the two displays above concludes the proof.
\end{proof}
\end{lemma}

Finally, we prove Theorem~\ref{thm:sbm}.
\begin{proof}[Proof of Theorem~\ref{thm:sbm}]
Recall the bound on $\Corr_{\leq D}(k_0,\ell_0)^2$ in Eq.~\eqref{eq:sbm:corr:bound}. Using Lemma~\ref{lem:sbm:d:alpha} to upper bound $d_{\al}^2$ and Lemma~\ref{lem:sbm:M:lower:bound} to lower bound $\sum_{\ga}M_{\al\ga,\al}^2$,
\[
\Corr_{\leq D}(k_0,\ell_0)^2\leq \sum_{\al\in \cIsbm:|\al|\leq D} \frac{\left(\frac{d\la}{n}\right)^{2|\al|}f(\al)^2 \left(\frac{q}{\pi_{\min}}\right)^{4(|\al|-|V(\al)|+2)}}{\left(\frac{d}{n}\right)^{|\al|}\left(1-\frac{\|Q\|_{\infty}}{n}\right)^{|\al|} \left(\frac{Q_{\min}}{d}\right)^{|\al|-|V(\al)|+2}},
\]
where we used that $\al\in \cIsbm$ has at most $2$ connected components. By Lemma~\ref{lem:f:upper:bound}, we have $f(\al)\leq (2|\al|)^{|\al|-|V(\al)|+1}$ and by Lemma~\ref{lem:number:good:graphs:2}, the number of good $\al$ with $|\al|=t$ and $|V(\al)|=v$ is at most $n^{v-2}(2d)^{5(d-v+1)}$. Write $\Delta=|\al|-|V(\al)|+1$ and note that $\Delta\geq 0$ for $\al\in \cIsbm$. Putting these all together,
\[
\begin{split}
\Corr_{\leq D}(k_0,\ell_0)^2
&\leq \sum_{t=1}^{D}\sum_{\Delta=0}^{t-1} \frac{C'}{n}\left(\frac{d \la^2}{1-\frac{\|Q\|_{\infty}}{n}}\right)^t \left(\frac{C'(2t)^7}{n}\right)^{\Delta}\\
&\leq \frac{C'}{n}\sum_{t=1}^{D}\left(\frac{d \la^2}{1-\frac{\|Q\|_{\infty}}{n}}\right)^t \sum_{\Delta\geq 0}\left(\frac{C'(2D)^7}{n}\right)^{\Delta},
\end{split}
\]
where we denoted $C':=\left(\frac{q}{\pi_{\min}}\right)^4\frac{d}{Q_{\min}}>0$, which is a constant that only depends on $q,\pi,Q$. The result now follows.
\end{proof}

\subsection{Proof of Proposition~\ref{prop:sbm:crucial:estimate}}
 \label{subsec:prop:sbm:crucial:estimate}
To facilitate the proof of Proposition~\ref{prop:sbm:crucial:estimate}, it is advantageous to consider an arbitrary finite graph $\alpha$ with vertex set $V(\alpha) \subseteq \mathbb{N}$, thereby removing the restriction that $V(\alpha) \subseteq [n]$. Note that the stochastic block model with vertices in $\mathbb{N}$ is well-defined by letting $(\sigma^\star_i)_{i \in \mathbb{N}}\iid \pi$ and defining $Y_{ij}\sim \Ber(\frac{Q_{k,\ell}}{n})$ whenever $\sigma^\star_i=k$ and $\sigma^\star_j=\ell$ for all $1 \le i,j < \infty$.
 
In this context, the parameter $n$ appears only in the edge probabilities $\frac{Q_{k,\ell}}{n}$. Our goal is to prove the estimate in Eq.~\eqref{eq:prop:sbm:crucial:estimate} for an arbitrary non-empty finite graph $\alpha \in \{0,1\}^{\{(i, j) \,:\, 1 \le i < j<\infty\}}$ and any subset $W \subseteq V(\alpha)$. We begin by considering the case where $\alpha$ is a tree.

\begin{lemma}\label{lem:sbm:crucial:estimate:tree}
Let $\al$ be a tree. For any $W\subseteq V(\al)$, we have
\begin{equation*}
\max_{\tau\in [q]^W}\left|\EE\big[\phi_{\al}(Y)\bgiven \sigma^\star_W=\tau\big]\right|\leq \left(\frac{d\la}{n}\right)^{|\al|}\left(\frac{q}{\pi_{\min}}\right)^{|W|-1}.
\end{equation*}
\end{lemma}
\begin{proof}
Throughout, $\tau=(\tau_i)_{i\in W}$ denotes an element of $[q]^W$.  By Lemma~\ref{lem:sbm:simple}, we may assume w.l.o.g.\ that every leaf of $\al$ is contained in $W$ since otherwise the LHS equals $0$. In particular, we have $|W|\geq 2$. Let us first consider the case where $W$ equals the set of leaves. Later, the case of general $W$ will follow from this case. 

\subsubsection*{Case 1: $W=\{v\in V(\al) \,:\, \deg_{\al}(v)=1\}$.}
Let $V_{\text{int}}= V(\al)\setminus W$ denote the set of internal vertices of $\al$. By a direct calculation, 
\[
\begin{split}
\E\big[\phi_{\al}(Y)\bgiven \sigma^\star_W=\tau\big]
&=\left(\frac{d}{n}\right)^{|\al|}\E\bigg[\prod_{(i,j)\in \al}\left(\frac{Q_{\sigma^\star_i,\sigma^\star_j}}{d}-1\right)\bbgiven \sigma^\star_W=\tau \bigg]\\
&=\left(\frac{d}{n}\right)^{|\al|}\sum_{\substack{\sigma\in [q]^{V(\al)}\\\sigma_j=\tau_j,~\forall j\in W}}~\prod_{i \in V_{\text{int}}}\pi_{\sigma_i}\prod_{(i,j)\in \al}\left(\frac{Q_{\sigma_i,\sigma_j}}{d}-1\right).
\end{split}
\]
Recalling that $B_{k,\ell}=\sqrt{\pi_k}\sqrt{\pi_{\ell}}(d^{-1}Q_{k,\ell}-1)$, we can simplify the RHS above as follows. Choose a root $\rho\in V(\al)$ for $\al$ and for $i\in V(\al)\setminus \{\rho\}$, let $p(i)$ denote the parent of $i$. Then,
\begin{equation}\label{eq:sbm:compute}
\E\big[\phi_{\al}(Y)\bgiven \sigma^\star_W=\tau\big]=\left(\frac{d}{n}\right)^{|\al|}\bigg(\prod_{i\in W}\pi_i\bigg)^{-1/2}\sum_{\substack{\sigma\in [q]^{V(\al)}\\\sigma_j=\tau_j,~\forall j\in W}}~\prod_{i\in V(\al)\setminus\{\rho\}}B_{\sigma_i,\sigma_{p(i)}}.
\end{equation}
Observe that if $\al$ is a path with end points $W=\{i_1,i_2\}$, the sum in the RHS of Eq.~\eqref{eq:sbm:compute} equals $(B^{|\al|})_{i_1,i_2}$. To deal with a general tree $\al$, we consider the set of branching points 
\[
V_{\branch}:=\{v\in V(\al) \,:\, \deg_{\al}(v)\geq 3\}.
\]
For $v\in V_{\branch}\cup W\setminus \{\rho\}$, consider the unique shortest path starting from $v$ and ending at $\bar{p}(v)\in V_{\branch}\cup W$, $\bar{p}(v)\neq v$, where $\bar{p}(v)$ is closer to $\rho$ than $v$. Such $\bar{p}(v)$ can be thought of as a `parent' branching point of $v$. Let $\ell_v$ denote the length of the segment $[v,\bar{p}(v)]$. Then, we have
\[
\sum_{\substack{\sigma\in [q]^{V(\al)}\\\sigma_j=\tau_j,~\forall j\in W}}~\prod_{i\in V(\al)\setminus\{\rho\}}B_{\sigma_i,\sigma_{p(i)}}=\sum_{\substack{\sigma\in [q]^{V_{\branch}\cup W}\\ \sigma_v=\tau_v,~\forall v\in W}}~\prod_{v\in V_{\branch}\cup W\setminus\{\rho\}} (B^{\ell_v})_{\sigma_v, \sigma_{\bar{p}(v)}}.
\]
Note that the RHS above can be bounded above in absolute value by
\[
q^{|V_{\branch}|}\prod_{v\in V_{\branch}\cup W\setminus\{\rho\}}\|B^{\ell_v}\|_{\infty}\leq q^{|V_{\branch}|}\lambda^{\sum_{v\in V_{\branch}\cup W\setminus\{\rho\}}\ell_v}=q^{|V_{\branch}|}\lambda^{|\al|},
\]
where the inequality is due to~\eqref{eq:sbm:simple} and the equality holds because the segments $[v,\bar{p}(v)]$ for $v\in V_{\branch}\cup W\setminus\{\rho\}$ constitute a partition of the tree $\al$. Consequently, plugging this estimate into Eq.~\eqref{eq:sbm:compute} yields
\[
\big| \EE\big[\phi_{\al}(Y)\bgiven \sigma^\star_W=\tau\big]\big|\leq \left(\frac{d\la}{n}\right)^{|\al|}\bigg(\prod_{i\in W}\pi_i\bigg)^{-1/2}q^{|V_{\branch}|}. 
\]
Finally, observe that since $V(\al)\setminus (V_{\branch}\cup W)$ is the set of degree-2 vertices, 
\[
2|\al|=\sum_{v\in V(\al)} \deg_{\al}(v)\geq |W|+2(|V(\al)|-|V_{\branch}|-|W|)+ 3|V_{\branch}|,
\]
so $|V_{\branch}|\leq |W|+2(|\al|-|V(\al)|)=|W|-2$. Therefore,
\[
\big| \EE\big[\phi_{\al}(Y)\bgiven \sigma^\star_W=\tau\big]\big|\leq \left(\frac{d\la}{n}\right)^{|\al|}\frac{q^{|W|-2}}{\pi_{\min}^{|W|/2}}\leq \left(\frac{d\la}{n}\right)^{|\al|}\left(\frac{q}{\pi_{\min}}\right)^{|W|-1},
\]
where we used $|W|/2 \leq |W|-1$ in the final inequality. This concludes the proof for the first case.

\subsubsection*{Case 2: General $W\supseteq \{v\in V(\al): \deg_{\al}(v)=1\}$.}
Observe that given a set $W$ containing the leaves of the tree $\al$, we can uniquely decompose $\al$ into $K$ trees $\al_{1},\ldots, \al_{K}$ where the leaves of $\al_i$ are given by $W\cap V(\al_i)$ for each $1\leq i \leq K$. That is, $(\al_{i})_{i\in [K]}$ are the subtrees of $\al$ obtained by splitting $\al$ according to the internal vertices of $\al$ that are contained in $W$. Since $W\cap V(\al_i)$ equals the leaves of $\al_i$, we can apply the results of Case 1 as follows. Writing $W_i= W\cap V(\al_i)$ and $\tau_i= (\tau_v)_{v\in W_i}$, we have for each $i \in [K]$,
\[
\big| \EE\big[\phi_{\al_i}(Y)\bgiven \sigma^\star_{W_i}=\tau_i\big]\big|\leq \left(\frac{d\la}{n}\right)^{|\al_i|}\left(\frac{q}{\pi_{\min}}\right)^{|W_i|-1}.
\]
Thus, it follows that
\[
\big| \EE\big[\phi_{\al}(Y)\bgiven \sigma^\star_{W}=\tau\big]\big|=\prod_{i=1}^{K} \big| \EE\big[\phi_{\al_i}(Y)\bgiven \sigma^\star_{W_i}=\tau_i\big]\big|\leq \prod_{i=1}^{K}\bigg\{\left(\frac{d\la}{n}\right)^{|\al_i|}\left(\frac{q}{\pi_{\min}}\right)^{|W_i|-1}\bigg\}.
\]
Finally, observe that $\sum_{i=1}^K|\al_i|=|\al|$ and
\[
\sum_{i=1}^{K}(|W_i|-1)=|W|-1.
\]
Indeed, if we root the tree $\al$ at $\rho$ and let $\rho_i$ be the leaf of $\al_i$ that is closest to $\rho$ (so $\rho_i=\rho$ for exactly one $i\in [K]$), then we have $W\setminus \{\rho\}=\sqcup_{i=1}^K W_i\setminus \{\rho_i\}$. This concludes the proof.
\end{proof}

We next generalize Lemma~\ref{lem:sbm:crucial:estimate:tree} to arbitrary connected graphs.
\begin{lemma}\label{lem:sbm:crucial:estimate:connected}
Let $\al$ be a non-empty connected graph. For any $W\subseteq V(\al)$, we have
\begin{equation}\label{eq:lem:sbm:crucial:estimate:connected}
 \max_{\tau\in [q]^W}\left|\EE\big[\phi_{\al}(Y)\bgiven \sigma^\star_W=\tau\big]\right|\leq \left(\frac{d\la}{n}\right)^{|\al|}\left(\frac{q}{\pi_{\min}}\right)^{2|\al|-2|V(\al)|+|W|+1}.
\end{equation}
\end{lemma}
\begin{proof}
We proceed by an induction over $|\al|-|V(\al)|$. The base case $|\al|-|V(\al)|=-1$ follows from Lemma~\ref{lem:sbm:crucial:estimate:tree}. Assume that our goal~\eqref{eq:lem:sbm:crucial:estimate:connected} holds for any $(\al,W)$ such that $|\al|-|V(\al)|= L-1$ where $L\geq 0$, and consider $(\al,W)$ such that $|\al|-|V(\al)|=L$. Since $|\al|-|V(\al)|\geq 0$, there exists $(v_1,v_2)\in \alpha$ such that $(v_1,v_2)$ belongs to a cycle in $\al$. Then consider the following graph $\al'$ created from $\al$: delete the edge $(v_1,v_2)$ from $\al$ and add an edge $(v',v_2)$, where $v'\notin V(\al)$ is a fresh vertex. That is, $\al'=\al-\bone_{(v_1,v_2)}+\bone_{(v',v_2)}$. An important observation is that the induction hypothesis can be applied to $\al'$ since $\al'$ is connected and $|\al'|-|V(\al')|=L-1$ holds because $|\al'|=|\al|$ while $|V(\al')|=|V(\al)|+1$. Moreover, for any given color $k\in [q]$, the conditional distribution of $(\phi_{\al}(Y),\sigma^\star_W)$ given $\sigma^\star_{v_1}=k$ equals that of $(\phi_{\al'}(Y), \sigma^\star_W)$ given $\sigma^\star_{v_1}=\sigma^\star_{v'}=k$. Consequently,
\[
\begin{split}
  \big| \EE\big[\phi_{\al}(Y)\bgiven \sigma^\star_{W}=\tau\big]\big|
  &\leq   \max_{k\in[q]}\big|\EE\big[\phi_{\al}(Y)\bgiven \sigma^\star_{W}=\tau,\; \sigma^\star_{v_1}=k\big]\big|\\
  &=\max_{k\in [q]}\big|\EE\big[\phi_{\al'}(Y)\bgiven \sigma^\star_{W}=\tau,\;\sigma^\star_{v_1}=\sigma^\star_{v'}=k\big]\big|.
\end{split}
\]
Using the induction hypothesis on the RHS, we have
\[
\begin{split}
\max_{\tau\in [q]^W}\big| \EE\big[\phi_{\al}(Y)\bgiven \sigma^\star_{W}=\tau\big]\big|
&\leq \left(\frac{d\la}{n}\right)^{|\al'|}\left(\frac{q}{\pi_{\min}}\right)^{2|\al'|-2|V(\al')|+|W|+3}\\
&=\left(\frac{d\la}{n}\right)^{|\al|}\left(\frac{q}{\pi_{\min}}\right)^{2|\al|-2|V(\al)|+|W|+1},
\end{split}
\]
which concludes the proof.
\end{proof}

\begin{proof}[Proof of Proposition~\ref{prop:sbm:crucial:estimate}]
Let $\al_1,\ldots, \al_K$ be the connected components of $\al$ and let $W_i\equiv W\cap V(\al_i)$. By assumption, $W_i\neq \empty$. Since $(\phi_{\al_i}(Y), \sigma^\star_{W_i})_{i\in [K]}$ are independent, we have
\[
\begin{split}
\max_{\tau\in [q]^W}\left|\EE\big[\phi_{\al}(Y)\bgiven \sigma^\star_W=\tau\big]\right|
&\leq \prod_{i=1}^{K}\max_{\tau_i\in [q]^{W_i}}\left|\EE\big[\phi_{\al_i}(Y)\bgiven \sigma^\star_{W_i}=\tau_i\big]\right|\\
&\leq \prod_{i=1}^{K}\left(\frac{d\la}{n}\right)^{|\al_i|}\left(\frac{q}{\pi_{\min}}\right)^{2|\al_i|-2|V(\al_i)|+|W_i|+1}\\
&\leq \left(\frac{d\la}{n}\right)^{|\al|}\left(\frac{q}{\pi_{\min}}\right)^{2(|\al|-|V(\al)|+|W|)},
\end{split}
\]
where we used Lemma~\ref{lem:sbm:crucial:estimate:connected} in the second inequality and $1+|W_i|\leq 2|W_i|$ in the final inequality.
\end{proof}

\section{Upper Bounds}
\label{sec:upper}

\subsection{Planted Submatrix}
\label{sec:subm-upper}

This section is devoted to the proof of Theorem~\ref{thm:main-subm}(b).

Let $\cT_k \subseteq \{0,1\}^{\bcn}$ denote the set of (simple) trees such that vertex 1 has exactly two neighbors, and the two subtrees rooted at these two neighbors each contain exactly $k$ edges. Each tree in $\cT_k$ has exactly $D := 2k+2$ edges in total. For some $k = k_n = \Theta(\log n)$ to be chosen later, consider the degree-$D$ polynomial
\[ f(Y) = \sum_{\alpha \in \cT_k} Y^\alpha. \]
By Cayley's tree formula, the number of spanning trees on $v$ (labeled) vertices is $v^{v-2}$, so
\begin{align*}
|\cT_k| &= \binom{n-1}{2} \binom{n-3}{k} \binom{n-3-k}{k} (k+1)^{2(k-1)} \\
&= \frac{(n-1)!}{2(k!)^2(n-2k-3)!} (k+1)^{2(k-1)} \\
&\ge \frac{1}{2}(n-2k-2)^{2k+2} (k!)^{-2} (k+1)^{2(k-1)} \\
&= (1-o(1)) (4 \pi k^3)^{-1} (en)^{2k+2},
\end{align*}
where the last step uses Stirling's approximation $k! = (1+o(1)) \sqrt{2\pi k} (k/e)^k$ as $k \to \infty$, along with the fact $((k+1)/k)^{2k} = (1+1/k)^{2k} \to e^2$.

Compute
\[ \EE[f(Y) \cdot \theta_1] = |\cT_k| \cdot \lambda^D \rho^{D+1}. \]
The main challenge will be to bound
\begin{equation}\label{eq:f-var}
\EE[f(Y)^2] = \sum_{\alpha,\beta \in \cT_k} \EE[Y^{\alpha+\beta}].
\end{equation}
For a pair $(\alpha,\beta)$ we will introduce various parameters. The first two that we need, $\tilde{m} = \tilde{m}(\alpha,\beta)$ and $\tilde{b} = \tilde{b}(\alpha,\beta)$ (closely related to $m,b$ used later) are defined as follows: $\tilde{m} \ge 0$ is the number of connected components in the forest $\alpha \cap \beta$, and $\tilde{b} = |\tilde{B}|$ where
\[ \tilde{B} := V(\alpha \cap \beta) \cap V(\alpha \triangle \beta). \]

\begin{lemma}\label{lem:Y-moments}
Let $\eta := \lambda^2$. For $\alpha,\beta \in \cT_k$, we have the following bounds:
\begin{itemize}
    \item If $\alpha = \beta$ then $\EE[Y^{\alpha+\beta}] \le (\eta \rho + 1)^{|\alpha|}$.
    \item If $\alpha \ne \beta$ then $\EE[Y^{\alpha+\beta}] \le \lambda^{|\alpha \triangle \beta|} \rho^{|V(\alpha \triangle \beta)|} (\eta + 1)^{\tilde{b}-\tilde{m}} (\eta \rho + 1)^{|\alpha \cap \beta| - (\tilde{b} - \tilde{m})}$.
\end{itemize}
\end{lemma}

\noindent We have introduced a new parameter $\eta := \lambda^2$ here for the purpose of reusing these calculation in planted dense subgraph, where $\eta$ will take a different value.

\begin{proof}
Compute
\begin{align}
\EE[Y^{\alpha+\beta}] &= \EE \prod_{(i,j) \in \alpha \cap \beta} (\lambda \theta_i \theta_j + Z_{ij})^2 \prod_{(i,j) \in \alpha \triangle \beta} (\lambda \theta_i \theta_j + Z_{ij}) \nonumber \\
&= \EE \prod_{(i,j) \in \alpha \cap \beta} (\lambda^2 \theta_i \theta_j + 1) \prod_{(i,j) \in \alpha \triangle \beta} \lambda \theta_i \theta_j \nonumber \\
&= \lambda^{|\alpha \triangle \beta|} \rho^{|V(\alpha \triangle \beta)|} \, \EE \left[\prod_{(i,j) \in \alpha \cap \beta} (\eta \theta_i \theta_j + 1) \;\Bigg|\; \theta_i = 1,~\forall i \in \tilde{B}\right].
\label{eq:eta-comparison}
\end{align}

\noindent We will bound the conditional expectation above, as follows. For each connected component of $\alpha \cap \beta$ (each component is a tree), choose a ``root'' vertex in that component that belongs to $\tilde{B}$ (one must always exist unless $\alpha = \beta$; in that case, choose any root vertex). Consider an arbitrary edge $(i,j) \in \alpha \cap \beta$, with $i$ closer to the root than $j$. If $j \in \tilde{B}$ we will upper bound the corresponding factor by $\eta \theta_i \theta_j + 1 \le \eta + 1$; if $j \notin \tilde{B}$ we will instead use the bound $\eta \theta_i \theta_j + 1 \le \eta \theta_j + 1$. In the case $\alpha \ne \beta$, this gives
\begin{align*}
\EE[Y^{\alpha+\beta}] &\le \lambda^{|\alpha \triangle \beta|} \rho^{|V(\alpha \triangle \beta)|} \, \EE \left[(\eta + 1)^{\tilde{b}-\tilde{m}} \prod_{j \in V(\alpha \cap \beta) \setminus \tilde{B}} (\eta \theta_j + 1)\right] \\
&= \lambda^{|\alpha \triangle \beta|} \rho^{|V(\alpha \triangle \beta)|} (\eta + 1)^{\tilde{b}-\tilde{m}} (\eta \rho + 1)^{|V(\alpha \cap \beta)| - \tilde{b}} \\
&= \lambda^{|\alpha \triangle \beta|} \rho^{|V(\alpha \triangle \beta)|} (\eta + 1)^{\tilde{b}-\tilde{m}} (\eta \rho + 1)^{|\alpha \cap \beta| + \tilde{m} - \tilde{b}}.
\end{align*}
A similar calculation applies in the case $\alpha = \beta$.
\end{proof}

We will break down the sum in~\eqref{eq:f-var} into three cases.

\subsubsection*{Case 1: $\alpha \cap \beta = \emptyset$.}

Let $w \ge 0$ denote the number of vertices aside from vertex 1 that are shared between $\alpha$ and $\beta$. We will give an upper bound on the number of $(\alpha,\beta)$ pairs that fall into Case~1 with a given value of $w$. To count these pairs, we need to first choose $w$ shared vertices. Then for each of $\alpha,\beta$ we need to choose 2 vertices to neighbor vertex 1 plus $2k-w$ additional vertices, then decide how the vertices are split among the two subtrees, and then finally choose the structure of the subtrees (using Cayley's tree formula to count spanning trees on $k+1$ vertices). The number of $(\alpha,\beta)$ pairs for a given $w \ge 0$ is at most
\begin{align*}
\binom{n}{w} \left[\binom{n}{2} \binom{n}{2k-w} \binom{2k}{k} (k+1)^{2(k-1)}\right]^2 &\le n^w \left[\frac{n^2}{2} \frac{n^{2k-w}}{(2k-w)!} \frac{(2k)!}{(k!)^2} (k+1)^{2(k-1)}\right]^2 \\
&\le n^w \left[\frac{n^{2k+2-w}}{2} \frac{(2k)^w}{(k!)^2} (k+1)^{2(k-1)}\right]^2 \\
&= \left(\frac{4k^2}{n}\right)^w \frac{n^{4k+4}}{4} (k!)^{-4} (k+1)^{4(k-1)} \\
&= (1+o(1)) \left(\frac{4k^2}{n}\right)^w \frac{n^{4k+4}}{4} \frac{1}{4\pi^2 k^6} e^{4k+4} \\
&= (1+o(1)) \left(\frac{4k^2}{n}\right)^w (16 \pi^2 k^6)^{-1} (en)^{4k+4},
\end{align*}
where $o(1)$ is uniform over $w$. Now by Lemma~\ref{lem:Y-moments} each of these $(\alpha,\beta)$ pairs has
\[ \EE[Y^{\alpha+\beta}] \le \lambda^{|\alpha \triangle \beta|} \rho^{|V(\alpha \triangle \beta)|} = \lambda^{2D} \rho^{2D+1-w}, \]
so the contribution from Case~1 to $\EE[f(Y)^2]$ is
\begin{align*}
\sum_{\substack{\alpha,\beta \in \cT_k \\ \textnormal{Case 1}}} \EE[Y^{\alpha+\beta}] &\le (1+o(1)) (16 \pi^2 k^6)^{-1} (en)^{4k+4} \lambda^{2D} \rho^{2D+1} \sum_{w \ge 0} \left(\frac{4k^2}{\rho n}\right)^w \\
&= (1+o(1)) (16 \pi^2 k^6)^{-1} (en)^{4k+4} \lambda^{2D} \rho^{2D+1},
\end{align*}
where we have used $k^2/(\rho n) = o(1)$ by our assumptions on $k,\rho$. We aim to show that this term dominates $\EE[f(Y)^2]$. To this end, define
\begin{equation}\label{eq:def-N}
N := k^{-6} (en)^{4k+4} \lambda^{2D} \rho^{2D+1} = k^{-6} \rho (en\lambda\rho)^{2D}.
\end{equation}

\subsubsection*{Case 2: $\alpha = \beta$.}

In this case we can directly compute, using Lemma~\ref{lem:Y-moments},
\[ \sum_{\substack{\alpha,\beta \in \cT_k \\ \textnormal{Case 2}}} \EE[Y^{\alpha+\beta}] \le |\cT_k| \, (\eta \rho + 1)^D \le (1-o(1))(16\pi k^3)^{-1} (en)^D (\eta \rho + 1)^D. \]
To show this is $o(N)$, we need to show
\[ k^3 \rho^{-1} \left(\frac{\eta\rho+1}{en \lambda^2 \rho^2}\right)^D = o(1). \]
This holds using the assumption on $\lambda$ along with $\eta\rho = o(1)$, and choosing $D$ to be a large enough multiple of $\log n$.

\subsubsection*{Case 3: $\alpha \cap \beta \ne \emptyset$ and $\alpha \ne \beta$.}

We classify the $(\alpha,\beta)$ pairs in this case based on the following parameters. Let $\ell := |\alpha \cap \beta| \ge 1$. We refer to the edges $\alpha \cap \beta$ as the ``core,'' which is always a forest. We define the ``core vertices'' $V_\core := V(\alpha \cap \beta) \cup \{1\}$, with vertex 1 always included by convention. Let $m \ge 1$ denote the number of connected components in the core, including $\{1\}$ in the case that vertex 1 is isolated. Vertices at which the core meets the remaining edges $\alpha \triangle \beta$ are called ``branch points'' and the number of them will be denoted
\begin{equation}\label{eq:def-b}
b := |V(\alpha \triangle \beta) \cap V_\core| \ge m.
\end{equation}
Also, the number of vertices shared by $\alpha$ and $\beta$ outside the core will be denoted
\begin{equation}\label{eq:def-w}
w := |(V(\alpha) \cap V(\beta)) \setminus V_\core|.
\end{equation}
We will bound the number of $(\alpha,\beta)$ pairs for a given tuple $(\ell,m,b,w)$. For the time being, we will forget the specific structure of trees in $\cT_k$ and overcount by allowing $\alpha,\beta$ to be any $D$-edge trees that contain vertex 1. It is a standard consequence of Cayley's formula that the number of \emph{rooted} forests on $v$ (labeled) vertices is exactly $(v+1)^{v-1}$, which is an upper bound on the number of forests. Since the core is a forest on $\ell+m$ vertices including vertex 1, the number of ways to choose the core is at most
\begin{align*}
\binom{n}{\ell+m-1} (\ell+m+1)^{\ell+m-1}
&\le \left(\frac{en}{\ell+m-1}\right)^{\ell+m-1} (\ell+m+1)^{\ell+m-1} \\
&= (en)^{\ell+m-1} \left(1 + \frac{2}{\ell+m-1}\right)^{\ell+m-1} \\
&\le e^2 (en)^{\ell+m-1}.
\end{align*}
Here we have used the standard bound $\binom{n}{k} \le (en/k)^k$, valid for $1 \le k \le n$. (In the case $\ell+m-1 = 0$, the final conclusion above can be verified by hand, despite the division by 0 in intermediate steps.) Next we need to choose $w$ shared vertices, along with $u := D+1-\ell-m-w$ additional vertices for each of $\alpha,\beta$. Then choose $b$ branch points among the $\ell+m$ core vertices. Now $\alpha \setminus \beta$ is a forest on at most $v := D+1-\ell-m+b$ vertices, so the number of ways to choose $\alpha \setminus \beta$ is a most $(v+1)^{v-1}$, and the same holds for $\beta \setminus \alpha$. Putting it together, the number of $(\alpha,\beta)$ pairs for a given tuple $(\ell,m,b,w)$ is at most
\[ e^2 (en)^{\ell+m-1} \binom{n}{w} \binom{n}{u}^2 \binom{\ell+m}{b} (v+1)^{2(v-1)} \le e^2 (en)^{\ell+m-1} n^w (\ell+m)^b \left(\frac{en}{u}\right)^{2u} (v+1)^{2(v-1)}. \]
Focusing on a particular portion of the above,
\begin{align*}
\frac{(v+1)^{2(v-1)}}{u^{2u}} &\le \left(\frac{v+1}{u}\right)^{2u} (v+1)^{2(v-u)-2} \\
&= \left(1 + \frac{v-u+1}{u}\right)^{2u} (v+1)^{2(v-u)-2} \\
&\le e^{2(v-u+1)} (v+1)^{2(v-u)-2} \\
&= e^{2(b+w+1)} (v+1)^{2(b+w)-2},
\end{align*}
so the above count of $(\alpha,\beta)$ pairs becomes
\begin{align*}
&e^2 (en)^{\ell+m-1+2u} n^w (\ell+m)^b e^{2(b+w+1)} (v+1)^{2(b+w)-2} \\
&\quad \le e^2 (en)^{\ell+m-1+2u} n^w (D+1)^b e^{2(b+w+1)} (D+2)^{2(b+w)} \\
&\quad = e^4 (en)^{\ell+m-1+2(d+1-\ell-m-w)} n^w (D+1)^b e^{2(b+w)} (D+2)^{2(b+w)} \\
&\quad = e^4 \left(e^2(D+1)(D+2)^2\right)^b \left(\frac{(D+2)^2}{n}\right)^w (en)^{2D+1-\ell-m} \\
&\quad \le e^4 \left(e^2(D+2)^3\right)^b \left(\frac{(D+2)^2}{n}\right)^w (en)^{2D+1-\ell-m}.
\end{align*}
By Lemma~\ref{lem:Y-moments}, and observing $\tilde{b}-\tilde{m} = b-m$, each of these $(\alpha,\beta)$ pairs has (recalling $\alpha \ne \beta$)
\begin{align*}
\EE[Y^{\alpha+\beta}] &\le \lambda^{|\alpha \triangle \beta|} \rho^{|V(\alpha \triangle \beta)|} (\eta+1)^{b-m} (\eta\rho+1)^{\ell-(b-m)} \\
&= \lambda^{2(D-\ell)} \rho^{2(D+1-\ell-m)+b-w} (\eta+1)^{b-m} (\eta\rho+1)^{\ell-(b-m)},
\end{align*}
so the contribution from Case~3 to $\EE[f(Y)^2]$ is
\begin{align*}
\sum_{\substack{\alpha,\beta \in \cT_k \\ \textnormal{Case 3}}} \EE[Y^{\alpha+\beta}] &\le \sum_{\ell,m,b,w} e^4 (en\lambda\rho)^{2D} \left(\frac{\eta\rho+1}{en\lambda^2\rho^2}\right)^\ell \left(\frac{1}{en \rho^2}\right)^{m-1} \left(e^2(D+2)^3 \rho\right)^b \left(\frac{(D+2)^2}{n \rho}\right)^w \left(\frac{\eta+1}{\eta\rho+1}\right)^{b-m} \\
&= (1+o(1)) e^5 n\rho^2 (en\lambda\rho)^{2D} \sum_{\ell,m,b} \left(\frac{\eta\rho+1}{en\lambda^2\rho^2}\right)^\ell \left(\frac{e^2(D+2)^3 \rho}{en\rho^2}\right)^m \left(\frac{e^2(D+2)^3 \rho(\eta+1)}{\eta\rho+1}\right)^{b-m} \\
&\le (1+o(1)) N e^5 k^6 n \rho \sum_{\ell,m,b} (1-\delta)^\ell \left(\frac{e(D+2)^3}{n\rho}\right)^m \left(\frac{e^2(D+2)^3 \rho(\eta+1)}{\eta\rho+1}\right)^{b-m}
\end{align*}
for a constant $\delta > 0$ depending on $\epsilon$. Here we have used the assumptions $D^2/(n\rho) = o(1)$ and $\eta\rho = o(1)$. We aim to show that the above quantity is $o(N)$, where $N$ was defined in~\eqref{eq:def-N}. It will now be important to use the specific structure of trees in $\cT_k$ to see which $(\ell,b)$ pairs are actually realizable. Note that $b \ge m$ always, but it will additionally be crucial to know that for $m=1$ we must either have $b \ge 2$ or large $\ell$.

\begin{lemma}\label{lem:b-ell}
Every pair $\alpha,\beta \in \cT_k$ with $\alpha \cap \beta \ne \emptyset$ must either have $b \ge 2$ or $\ell \ge k+1$.
\end{lemma}

\begin{proof}
First consider the case where vertex 1 (the ``root'') is not a branch point. This means the two root-incident edges of $\alpha$ match those of $\beta$. If both subtrees of $\alpha$ contain a branch point then $b \ge 2$. Otherwise, one subtree has no branch point, but this means the entire subtree must be contained in the core, so $\ell \ge k+1$.

Now consider the case where the root is a branch point. Since $\alpha \cap \beta \ne \emptyset$ by assumption, (at least) one of the two subtrees of $\alpha$ has an edge in the core (where the root-incident edge is included). If this subtree has another branch point (aside from the root) then $b \ge 2$. Otherwise, the entire subtree must be contained in the core, so $\ell \ge k+1$.
\end{proof}

Choose $k = \Theta(\log n)$ large enough so that $(1-\delta)^{k+1} \le n^{-1}$. Continuing from above and isolating the $m=1$ term,
\begin{align*}
\sum_{\substack{\alpha,\beta \in \cT_k \\ \textnormal{Case 3}}} \EE[Y^{\alpha+\beta}] &\le O(N D^6 n \rho) \sum_{\ell,b} \Bigg[ (1-\delta)^\ell \left(\frac{e(D+2)^3}{n\rho}\right) \left(\frac{e^2(D+2)^3 \rho(\eta+1)}{\eta\rho+1}\right)^{b-1} \\
&\qquad + \sum_{m \ge 2} (1-\delta)^\ell \left(\frac{e(D+2)^3}{n\rho}\right)^m \left(\frac{e^2(D+2)^3 \rho(\eta+1)}{\eta\rho+1}\right)^{b-m} \bigg] \\
&\le O(N D^6 \rho^{-1} (\eta+1)^{-1}) \sum_{\ell,b} (1-\delta)^\ell \, O(D^3 \rho (\eta+1))^b \\
&\qquad + O(N D^{12} (n \rho)^{-1}) \sum_{m \ge 2} \sum_{b \ge m} O(D^3 (n \rho)^{-1})^{m-2} \, O(D^3 \rho (\eta+1))^{b-m},
\intertext{and now using Lemma~\ref{lem:b-ell},}
&\le O\left(\frac{N D^6}{\rho(\eta+1)}\right) \left[\sum_{\ell \ge 1, b \ge 2} (1-\delta)^\ell \, O(D^3 \rho (1+\eta))^b + \sum_{\ell \ge k+1, b \ge 1} (1-\delta)^\ell \, O(D^3 \rho (1+\eta))^b\right] + o(N) \\
&\le O\left(\frac{N D^6}{\rho(\eta+1)}\right) \Big[O(D^6 \rho^2 (1+\eta)^2) + O(n^{-1} \cdot D^3 \rho (1+\eta))\Big] + o(N) \\
&= o(N).
\end{align*}

\begin{proof}[Proof of Theorem~\ref{thm:main-subm}(b)]
Now combining the three cases above, we conclude
\[ \EE[f(Y)^2] \le (1+o(1)) (16 \pi^2 k^6)^{-1} (en)^{4k+4} \lambda^{2D} \rho^{2D+1}. \]
Also recall
\[ \EE[f(Y) \cdot \theta_1] = |\cT_k| \cdot \lambda^D \rho^{D+1} \ge (1-o(1)) (4\pi k^3)^{-1} (en)^{2k+2} \lambda^D \rho^{D+1}, \]
which gives $\Corr_{\le C \log n} = 1-o(1)$ as desired.

To recap, the assumptions we needed were
\begin{equation}\label{eq:subm-upper-final-cond}
\lambda \ge (1+\eps)(\rho\sqrt{en})^{-1}, \quad n \rho = \omega(\log^{12} n), \quad \rho(\eta+1) = o(\log^{-12} n). 
\end{equation}
Specializing to $\eta = \lambda^2$, we will show that the last assumption can be replaced by $\rho = o(\log^{-12} n)$. Note that $\Corr_{\le D}$ is increasing in $\lambda$ (with all else held fixed) since a random polynomial can simulate additional noise; see~\cite[Claim~A.2 \& Fact~1.1]{SW-estimation}. This means we can shrink $\lambda$ as necessary to assume without loss of generality that the first assumption holds with equality: $\lambda = (1+\eps)(\rho\sqrt{en})^{-1}$. Now the requirement $\rho\lambda^2 = o(\log^{-12} n)$ reduces to $n \rho = \omega(\log^{12} n)$, which is the second assumption.
\end{proof}

\subsection{Planted Dense Subgraph}

This section is devoted to the proof of Theorem~\ref{thm:main-pds}(b).

The estimator and analysis will be similar to the case of planted submatrix in Section~\ref{sec:subm-upper}. Define the set of trees $\cT_k$ as in Section~\ref{sec:subm-upper} and consider the estimator
\[ f(Y) = \sum_{\alpha \in \cT_k} \prod_{(i,j) \in \alpha} \frac{Y_{ij}-p_0}{\sqrt{p_0(1-p_0)}}. \]
Recall that $\alpha \in \cT_k$ has $|\alpha| = D := 2k+2$. Compute
\[ \EE[f(Y) \cdot \theta_1] = |\cT_k| \cdot \rho^{D+1} \left(\frac{p_1-p_0}{\sqrt{p_0(1-p_0)}}\right)^D. \]
The main challenge will be to bound
\[ \EE[f(Y)^2] = \sum_{\alpha,\beta \in \cT_k} \EE \left(\prod_{(i,j) \in \alpha} \frac{Y_{ij}-p_0}{\sqrt{p_0(1-p_0)}}\right) \left(\prod_{(i,j) \in \beta} \frac{Y_{ij}-p_0}{\sqrt{p_0(1-p_0)}}\right) =: \sum_{\alpha,\beta \in \cT_k} R_{\alpha \beta}. \]
For a pair $(\alpha,\beta)$, recall the parameters $\tilde{m}$ and $\tilde{b}$ from Section~\ref{sec:subm-upper}.

\begin{lemma}
\label{lem:PDS:estimate}
Suppose $p_1 \ge p_0$ and define
\begin{equation}\label{eq:subg-subst}
\eta := \frac{p_1-p_0}{p_0(1-p_0)} \ge 0 \qquad \text{and} \qquad \lambda := \frac{p_1-p_0}{\sqrt{p_0(1-p_0)}}.
\end{equation}
For $\alpha,\beta \in \cT_k$, we have the following bounds:
\begin{itemize}
    \item If $\alpha = \beta$ then $R_{\alpha\beta} \le (\eta \rho + 1)^{|\alpha|}$.
    \item If $\alpha \ne \beta$ then $R_{\alpha\beta} \le \lambda^{|\alpha \triangle \beta|} \rho^{|V(\alpha \triangle \beta)|} (\eta + 1)^{\tilde{b}-\tilde{m}} (\eta \rho + 1)^{|\alpha \cap \beta| - (\tilde{b} - \tilde{m})}$.
\end{itemize}
\end{lemma}

\begin{proof}
With $\tilde{B}$ defined as in Section~\ref{sec:subm-upper}, compute
\begin{align*}
R_{\alpha\beta} &= \EE \prod_{(i,j) \in \alpha \cap \beta} \left(\frac{Y_{ij}-p_0}{\sqrt{p_0(1-p_0)}}\right)^2 \prod_{(i,j) \in \alpha \triangle \beta} \left(\frac{Y_{ij}-p_0}{\sqrt{p_0(1-p_0)}}\right) \\
&= \rho^{|V(\alpha \triangle \beta)|} \left(\frac{p_1-p_0}{\sqrt{p_0(1-p_0)}}\right)^{|\alpha \triangle \beta|} \EE \left[\prod_{(i,j) \in \alpha \cap \beta} \left(\frac{Y_{ij}-p_0}{\sqrt{p_0(1-p_0)}}\right)^2 \;\Bigg|\; \theta_i = 1,~\forall i \in \tilde{B}\right] \\
&= \rho^{|V(\alpha \triangle \beta)|} \lambda^{|\alpha \triangle \beta|} \, \EE \left[\prod_{(i,j) \in \alpha \cap \beta} \frac{Y_{ij} - 2 p_0 Y_{ij} + p_0^2}{p_0(1-p_0)} \;\Bigg|\; \theta_i = 1,~\forall i \in \tilde{B}\right] \\
&= \rho^{|V(\alpha \triangle \beta)|} \lambda^{|\alpha \triangle \beta|} \, \EE \left[\prod_{(i,j) \in \alpha \cap \beta} \frac{(1 - 2 p_0)[p_0 + (p_1-p_0)\theta_i\theta_j] + p_0^2}{p_0(1-p_0)} \;\Bigg|\; \theta_i = 1,~\forall i \in \tilde{B}\right] \\
&= \rho^{|V(\alpha \triangle \beta)|} \lambda^{|\alpha \triangle \beta|} \, \EE \left[\prod_{(i,j) \in \alpha \cap \beta} \big((1-2p_0)\eta \theta_i \theta_j + 1\big) \;\Bigg|\; \theta_i = 1,~\forall i \in \tilde{B}\right].
\end{align*}
Taking absolute values on both sides, we can bound
\[
\begin{split}
|R_{\al,\be}|
&\leq \rho^{|V(\alpha \triangle \beta)|} \lambda^{|\alpha \triangle \beta|} \, \EE \left[\prod_{(i,j) \in \alpha \cap \beta} \Big|(1-2p_0)\eta \theta_i \theta_j + 1\Big|\;\Bigg|\; \theta_i = 1,~\forall i \in \tilde{B}\right]\\
&\leq \rho^{|V(\alpha \triangle \beta)|} \lambda^{|\alpha \triangle \beta|} \, \EE \left[\prod_{(i,j) \in \alpha \cap \beta} (\eta \theta_i \theta_j + 1) \;\Bigg|\; \theta_i = 1,~\forall i \in \tilde{B}\right],
\end{split}
\]
where the last inequality holds because $\big|(1-2p_0)\eta \theta_i\theta_j+1\big|\leq \eta \theta_i\theta_j+1$ by noting that $\eta,\theta_i,\theta_j$ are all non-negative and $p_0\in [0,1]$. Comparing the RHS to~\eqref{eq:eta-comparison}, we now follow the proof of Lemma~\ref{lem:Y-moments} to obtain the result.
\end{proof}

\begin{proof}[Proof of Theorem~\ref{thm:main-pds}(b)]
Our situation is now identical to that of planted submatrix in Section~\ref{sec:subm-upper}, with the substitutions~\eqref{eq:subg-subst}. From the proof in Section~\ref{sec:subm-upper} we conclude $\Corr_{\le C \log n} = 1-o(1)$ under the assumptions~\eqref{eq:subm-upper-final-cond}, namely
\[ \lambda \ge (1+\epsilon)(\rho\sqrt{en})^{-1}, \quad n \rho = \omega(\log^{12} n), \quad \rho(\eta+1) = o(\log^{-12} n). \]
We will show that the last condition can be replaced by $\rho = o(\log^{-12} n)$ and $p_0 = \omega(n^{-1} \log^{24} n)$. Given an input graph $Y$, consider the operation that selects each edge independently with some probability $s$, and then resamples the selected edges as Bernoulli$(p_0)$. This has the effect of decreasing $p_1$ to some value in the range $[p_0,p_1]$ while leaving all other parameters unchanged. Since a random polynomial can implement this operation, $\Corr_{\le D}$ is increasing in $p_1$; see~\cite[Claim~A.1 \& Fact~1.1]{SW-estimation}. This means we can shrink $p_1$ as necessary to assume without loss of generality that the first assumption holds with equality: $\lambda = (1+\eps)(\rho\sqrt{en})^{-1}$. Since $\eta = \frac{ \lambda}{\sqrt{p_0(1-p_0)}}$, the requirement $\rho\eta = o(\log^{-12} n)$ is implied by $p_0 = \omega(n^{-1} \log^{24} n)$.
\end{proof}

\subsection{Spiked Wigner}

This section is devoted to the proof of Theorem~\ref{thm:main-wigner}(b).

Note that $\Corr_{\le D}$ is increasing in $\lambda$ (with all else held fixed) since a random polynomial can simulate additional noise; see~\cite[Claim~A.2 \& Fact~1.1]{SW-estimation}. This means we can shrink $\lambda$ as necessary to assume without loss of generality that $\lambda = 1+\eta$ for a constant $\eta \in (0,1)$.

Throughout, we let $K\equiv \E\pi^4<\infty$ and assume $D=\Theta(\log n)$, with the specific scaling of $D$ to be chosen later. In addition, we denote the rows of the matrix $U\in \R^{n\times m}$ by $x_i \in \R^m$ for $i \in [n]$, so the entries $(x_{i,k})_{k \in [m]}$ of $x_i$ are drawn i.i.d.\ from $\pi$.

Let $\cS_D\subseteq \{0,1\}^{\bcn}$ denote the set of self-avoiding (simple) paths from vertex $1$ to vertex $2$ with length $D$ (i.e., $D$ edges). Consider the estimator
\begin{equation}\label{eq:f:saw}
f(Y)=\frac{1}{|\cS_D|}\sum_{\al \in \cS_D} Y^{\al}.
\end{equation}
Note that
\[
\E[f(Y)\cdot x]=\left(\frac{\la}{n}\right)^{(D+1)/2}\E\Big[\prod_{(i,j)\in \al+\one_{(1,2)}}\langle x_i,x_j\rangle\Big]=m\left(\frac{\la}{n}\right)^{(D+1)/2},
\]
where $\al\in \cS_D$ is arbitrary. Here, the first equality holds by first taking expectation w.r.t.\ the noise $Z$, and the second equality holds since we assumed $\E\pi=0$, $\E\pi^2=1$. Thus, the main challenge is to bound
\begin{equation}\label{eq:express:f:sec:mo}
\E[f(Y)^2]=\frac{1}{|\cS_D|^2}\sum_{\al,\be\in \cS_D} \E[Y^{\al+\be}].
\end{equation}
Our first task is to compute $\E[Y^{\al+\be}]$. The following lemma will be useful here.
\begin{lemma}\label{lem:compute:path}
Suppose $\ga$ is a self-avoiding path joining two distinct vertices. Then,
\[
\E\bigg[\prod_{(i,j)\in \ga}\langle x_i,x_j\rangle^2\bigg]=m^{|\ga|}\cdot \left(1+\frac{K-1}{m}\right)^{|\ga|-1},
\]
which is at most $(Km)^{|\ga|}$.
\end{lemma}
\begin{proof}
Expanding the LHS in terms of the entries $(x_{i,k})_{i \in [n], k \in [m]}$ yields
\[
\E\bigg[\prod_{(i,j)\in \ga}\langle x_i,x_j\rangle^2\bigg]=\sum_{k=(k_{ij}), \ell=(\ell_{ij})\in [m]^{\ga}}\E\bigg[\prod_{(i,j)\in \ga}x_{i,k_{ij}}x_{i,\ell_{ij}}x_{j,k_{ij}}x_{j,\ell_{ij}}\bigg].
\]
Because $\E \pi=0$, we may restrict this sum to $k,\ell \in [m]^{\ga}$ that satisfy the following property: for any $k_{ij}$ (resp.\ $\ell_{ij}$) where $(i,j)\in \ga$, either $k_{ij}=\ell_{ij}$ or there exists $j'\neq j$ such that $(i,j')\in \ga$ and either $k_{ij'}$ or $\ell_{ij'}$ equals $k_{ij}$ (resp.\ $\ell_{ij}$). Moreover, since $\ga$ is a self-avoiding path, it can be easily seen that the only $k,\ell$ satisfying this property are $k,\ell$ such that $k=\ell$ (start from the one endpoint of $\ga$ and apply the property). Consequently, recalling the notation $K\equiv \E \pi^4$, the above sum equals
\[
\sum_{k\in [m]^{\ga}}\E\bigg[\prod_{(i,j)\in \ga}x_{i,k_{ij}}^2 x_{j,k_{ij}}^2\bigg]=m^{|\ga|}\cdot \E\bigg[K^{\sum_{t=1}^{|\ga|-1}\one_{U_t=U_{t+1}}}\bigg],
\]
where $U_1,\ldots, U_{|\ga|}$ are i.i.d.\ $\textnormal{Unif}([m])$ random variables. Finally, note that the random variables $(\one_{U_t=U_{t+1}})_{1 \leq t\leq |\ga|-1}$ are independent and distributed according to $\Ber(1/m)$. Therefore, the RHS equals $m^{|\ga|}(1+(K-1)/m)^{|\ga|-1}$, which concludes the proof.
\end{proof}

\begin{lemma}\label{lem:compute:Y:al:plus:be}
Let $\al,\be\in \cS_D$ and denote $v:=|V(\al)\cap V(\be)|$ and $\ell:=|\al\cap \be|$. Then,
\[
0\leq \EE[Y^{\al+\be}]\leq (Km)^{v-\ell-1}\cdot \left(1+\frac{Km\la}{n}\right)^{\ell}\left(\frac{\la}{n}\right)^{D-\ell}.
\]
\end{lemma}
\begin{proof}
Compute $\E[Y^{\al+\be}]$ by taking expectation w.r.t.\ $Z$ first:
\begin{equation}\label{eq:compute:Y:al:be}
\E[Y^{\al+\be}]=\left(\frac{\la}{n}\right)^{|\al\tri \be|/2}\E\left[\prod_{(i,j)\in \al\cap \be}\left(\frac{\la}{n}\langle x_i,x_j\rangle^2+1\right)\cdot\prod_{(i,j)\in \al\tri\be}\langle x_i,x_j\rangle\right].   
\end{equation}
Note that since $\al,\be$ are self-avoiding paths joining vertices $1$ and $2$, $\al\tri \be$ can be uniquely decomposed into cycles where each cycle contains exactly $2$ vertices in $V(\al)\cap V(\be)$. Let $\al \tri \be =\sqcup_{\ell=1}^{L} \ga_\ell$ denote this decomposition and denote $\{i_{\ell},j_{\ell}\}= V(\ga_\ell)\cap V(\al)\cap V(\be)$ for $1\leq \ell \leq L$. Here, $L$ is determined by $v\equiv |V(\al)\cap V(\be)|$ and $\ell\equiv |\al\cap \be|$ as $L=v-\ell-1$. This is because we may double count the multi-set $\{1,i_1,j_1,\ldots, i_L,j_L, 2\}$ as $2L+2$ or the number of non-internal vertices in $\al\cap \be$, i.e., vertices with degree $1$ in $\al\cap\be$, plus twice the number of isolated vertices $(V(\al)\cap V(\be))\setminus V(\al\cap \be)$. For example, if $V(\al)\cap V(\be)=\{1,i,2\}$ for some $3\leq i\leq n$, then $L=1$ if $\al\cap \be = \one_{(1,i)}$ or $\al\cap \be=\one_{(i,2)}$, and $L=2$ if $\al\cap \be=\emptyset$. 

We compute the RHS of~\eqref{eq:compute:Y:al:be} by first conditioning on $(x_i)_{i\in V(\al)\cap V(\be)}$, i.e., averaging only w.r.t.\ vertices of $\al\tri \be$ except for $\{i_1,j_{1},\ldots, i_{L},j_L\}$. Note that the first product in the RHS of \eqref{eq:compute:Y:al:be} is $(x_i)_{i\in V(\al)\cap V(\be)}$-measurable because $V(\al\cap\be)\subseteq V(\al)\cap V(\be)$, and for the second product,
\[
\E\bigg[\prod_{(i,j)\in \al\tri \be}\langle x_i,x_j\rangle\bbgiven \left(x_i\right)_{i\in V(\al)\cap V(\be)}\bigg]=\prod_{\ell=1}^{L}\E\bigg[\prod_{(i,j)\in \ga_\ell}\langle x_i,x_j\rangle \bbgiven x_{i_{\ell}}, x_{j_{\ell}}\bigg]=\prod_{\ell=1}^{L}\langle x_{i_{\ell}}, x_{j_{\ell}}\rangle ^2,
\]
where the first equality holds since two different cycles $\ga_{\ell}$ and $\ga_{\ell'}$ can only share vertices among $V(\al)\cap V(\be)$ and the second equality can be seen by expanding $\prod_{(i,j)\in \ga_\ell} \sum_{k \in [m]} x_{i,k},x_{j,k}$ and using $\E \pi=0$, $\E\pi^2=1$. Thus, if we let $\zeta=\sum_{\ell=1}^{L}\one_{(i_{\ell},j_{\ell})}$ be the path joining $(i_1,j_1),\ldots ,(i_L,j_L)$, then $|\zeta|=L=v-\ell-1$ and 
\[
\begin{split}
\E[Y^{\al+\be}]
&=\left(\frac{\la}{n}\right)^{|\al\tri \be|/2}\E\left[\prod_{(i,j)\in \al\cap \be}\left(\frac{\la}{n}\langle x_i,x_j\rangle^2+1\right)\cdot\prod_{(i,j)\in \zeta}\langle x_i,x_j\rangle^2\right]\\
&=\left(\frac{\la}{n}\right)^{D-\ell}\sum_{0\leq \gamma\leq \al \cap\be} \left(\frac{\la}{n}\right)^{|\ga|}\E\left[\prod_{(i,j)\in \ga+\zeta}\langle x_i,x_j\rangle^2\right],
\end{split}
\]
where we used $|\al\tri \be|=2D-2|\al\cap\be|$ and that $\al\cap\be\cap \zeta=\empty$ in the last equality. From this expression, we have $\E[Y^{\al+\be}]\geq 0$. Further, note that since $\al\cap \be$ is a disjoint union of paths and $\zeta$ is a path, $\ga+\zeta$ must be a disjoint union of paths for any $\ga\le \al\cap\be$. As a consequence, we can upper bound the final expectation using Lemma~\ref{lem:compute:path} as $\E \prod_{(i,j)\in \ga+\zeta}\langle x_i,x_j\rangle^2\leq (Km)^{|\ga|+|\zeta|}$. Hence, 
\[
\E[Y^{\al+\be}]\leq \left(\frac{\la}{n}\right)^{D-\ell}(Km)^{|\zeta|}\sum_{0\leq \gamma\leq \al\cap\beta}\left(\frac{Km\la}{n}\right)^{|\ga|}=\left(\frac{\la}{n}\right)^{D-\ell}(Km)^{v-\ell-1}\cdot\left(1+\frac{Km\la}{n}\right)^{\ell},
\]
which concludes the proof.
\end{proof}

Our next task is to bound the number of $(\al,\be)\in \cS_D$ such that $|\al\cap\be|=\ell$ and $|V(\al)\cap V(\be)|=v$. Note that if $\ell=D$, then we must have $\al=\be$ and $v=D+1$, and the number of such $(\alpha,\beta)$ pairs is $|\cS_D|$. Otherwise, $\ell\leq D-1$ and in this case, we must have $v\geq \ell+2$. In the following, we let $(m)_{k}\equiv \prod_{\ell=0}^{k-1}(m-\ell)$ denote the falling factorial.

\begin{lemma}\label{lem:combinatorics:saw}
Fix $\al\in \cS_D$, $0\leq \ell \leq D-1$, and $\ell+2\leq v\leq D$. The number of $\be\in \cS_D$ such that $|\al\cap\be|=\ell$ and $|V(\al)\cap V(\be)|=v$ is at most $\binom{\ell+2}{2}\cdot D^{3(v-\ell-2)}\cdot n^{D+1-v}$. Moreover, if $v=\ell+2$, the number of such $\beta\in \cS_D$ equals $(\ell+1)\cdot(n-D-1)_{D-\ell-1}$.
\end{lemma}

\begin{proof}
Throughout, we fix $\al\in \cS_D$. We construct $\be\in \cS_D$ by first choosing $\al\cap \be$ and $(V(\al)\cap V(\be))\setminus V(\al\cap\be)$, and then constructing $\be\setminus \al$ attached to $V(\al)\cap V(\be)$.

Note that $\al \cap \be$ is a disjoint union of paths contained in $\al$, and $(V(\al)\cap V(\be))\setminus V(\al\cap \be)$ is a set of isolated vertices in $\al$. We may view these altogether as another disjoint union of paths, where each isolated vertex in $(V(\al)\cap V(\be))\setminus V(\al\cap \be)$ is considered a length-0 ``path'' with 1 vertex. Let us write $(\al\cap\be)_+$ as the corresponding union of paths and call its internal boundary (i.e., isolated vertices and degree-1 vertices) the \textit{end-points} of $(\al\cap\be)_+$. An important observation is that given $\al$, the end-points and the choice of isolated vertices among them determines $(\al\cap\be)_{+}$.

Consequently, we construct $(\al\cap\be)_{+}$ as follows. First, we choose the two paths in $(\al\cap\be)_{+}$ that contain vertex $1$ and vertex $2$ respectively (these are distinct, since $\ell < D$). Given $\al$, these paths are determined by their lengths, which are non-negative and sum to at most $\ell$, so the number of ways to choose the paths is at most $\binom{\ell+2}{2}$. Second, we choose the end-points of the rest of the paths in $(\al\cap\be)_{+}$. Note that the number of paths in $(\al\cap\be)_{+}$ is $v-\ell$, since $(\al\cap\be)_{+}$ is a union of paths having $v$ vertices and $\ell$ edges in total. Thus, the number of paths that don't contain vertex $1$ nor $2$ is $v-\ell-2$, so the number of end-points contained in these paths is $2(v-\ell-2)-w$ where $w$ is the number of isolated vertices among these paths. As a result, the number of ways to choose the end-points of these paths is at most $\binom{D-1}{2(v-\ell-2)-w}$. Finally, the number of ways to choose the isolated vertices among the available $2(v-\ell-2)-w$ end-points is at most $\binom{2(v-\ell-2)-w}{w}$. In total, the number of ways to construct $(\al\cap \be)_{+}$ is at most
\[
\begin{split}
\binom{\ell+2}{2}\sum_{w\geq 0}\binom{D-1}{2(v-\ell-2)-w}\binom{2(v-\ell-2)-w}{w}
&\leq \binom{\ell+2}{2}\sum_{w\geq 0}(D-1)^{2(v-\ell-2)-w}\binom{2(v-\ell-2)}{w}\\
&=\binom{\ell+2}{2} D^{2(v-\ell-2)}.
\end{split}
\]
Having constructed $(\al\cap\be)_{+}$, we now construct $\be\setminus \al$. Recall that $\al\tri \be$ consists of $v-\ell-1$ cycles where each cycle contains exactly $2$ vertices in $V(\al)\cap V(\be)$ (see the proof of Lemma~\ref{lem:compute:Y:al:plus:be}). Thus, $\be\setminus \al$ consists of $v-\ell-1$ paths, where each path has length at least $2$ and is attached to $2$ of the end-points of $(\al\cap \be)_{+}$. Since $|V(\be)\setminus V(\al)|=D+1-v$, the number of ways to assign lengths to $v-\ell-1$ paths (or equivalently, the number of internal vertices) is $\binom{D-v}{v-\ell-2}$ by the ``stars and bars'' method. For each assignment of lengths to these paths, the number of ways to construct $\be\setminus \al$ is at most the falling factorial $(n-D-1)_{D+1-v}$. Therefore, having chosen $(\al\cap\be)_{+}$, the number of ways to construct $\be\setminus \al$ is at most $\binom{D-v}{v-\ell-2}\cdot (n-D-1)_{D+1-v}\leq D^{v-\ell-2}\cdot (n-D-1)_{D+1-v}$. This establishes the desired bound in the case $v\geq \ell+3$.

If $v=\ell+2$, then $(\al \cap \be)_{+}$ consist of two distinct paths that contain vertex $1$ and vertex $2$ respectively. Given $\al$, $(\al\cap\be)_{+}$ is determined by the lengths of these two paths, which must sum to $\ell$. Therefore, the number of ways to choose $(\al \cap \be)_{+}$ is $\ell+1$. Having constructed $(\al\cap \be)_{+}$, the number of ways to construct $\be\setminus \al$ is $(n-D-1)_{D+1-v}=(n-D-1)_{D-\ell-1}$. Thus, the desired claim holds for $v=\ell+2$.
\end{proof}

\begin{proof}[Proof of Theorem~\ref{thm:main-wigner}(b)]
Consider the estimator $f \in \R[Y]_{\leq D}$ from Eq.~\eqref{eq:f:saw}, 
with $D = \lceil \frac{C}{\eta}\log (n/m) \rceil$ where we take $C>2$. Recalling the expression \eqref{eq:express:f:sec:mo}, we can bound $\E[f(Y)^2]$ by using Lemma~\ref{lem:compute:Y:al:plus:be} and~\ref{lem:combinatorics:saw} as
\begin{equation}\label{eq:bound:sec:mo}
\begin{split}
\E[f(Y)^2]
&\leq\frac{1}{|\cS_D|}\sum_{\ell=0}^{D-1}\sum_{v=\ell+3}^{D-1}\binom{\ell+2}{2}D^{3(v-\ell-2)}n^{D+1-v}\cdot (Km)^{v-\ell-1}\left(1+\frac{Km\la}{n}\right)^{\ell}\left(\frac{\la}{n}\right)^{D-\ell}\\
&\quad\quad+\frac{1}{|\cS_D|}\sum_{\ell=0}^{D-1}(\ell+1)n^{D-\ell-1}\cdot Km\left(1+\frac{Km\la}{n}\right)^{\ell}\left(\frac{\la}{n}\right)^{D-\ell}+\frac{1}{|\cS_D|}\left(1+\frac{Km\la}{n}\right)^D,
\end{split}
\end{equation}
where the last term accounts for the contribution from identical copies, i.e., $\ell=D$ and $v=D+1$. Since $|\cS_D|=(n-2)_{D-1}=n^{D-1}(1+o(1))$, the first sum equals
\[
\begin{split}
&\left(1+o(1)\right)Km\left(\frac{\la}{n}\right)^{D}\cdot \sum_{\ell=0}^{D-1}\binom{\ell+2}{2}\left(1+\frac{Km\la}{n}\right)^{\ell}\la^{-\ell}\sum_{v=\ell+3}^{D-1}\left(\frac{KmD^3}{n}\right)^{v-\ell-2}\\
&\leq \left(1+o(1)\right)Km\left(\frac{\la}{n}\right)^{D}\cdot \sum_{\ell=0}^{D-1}\binom{\ell+2}{2}\left(\frac{1}{\la}+\frac{Km}{n}\right)^{\ell}\cdot \frac{KmD^3}{n}\\
&\leq \left(1+o(1)\right)\frac{KmD^3}{n}\cdot Km\left(\frac{\la}{n}\right)^{D}\cdot \sum_{\ell=0}^{\infty}\binom{\ell+2}{2}\left(\frac{1}{1+\eta/2}\right)^{\ell}\\
&\leq o(1)\cdot m\left(\frac{\la}{n}\right)^{D},
\end{split}
\]
where the first and third inequality holds since $\frac{KmD^3}{n} = \frac{CK^3}{\eta}\cdot \frac{m}{n}\log^3(\frac{n}{m})= o(1)$ by our assumption that $m=o(n)$, and the second inequality holds since $\frac{Km}{n}=o(1)$. Similarly, the second sum in~\eqref{eq:bound:sec:mo} equals
\[
\begin{split}
\left(1+o(1)\right)Km\left(\frac{\la}{n}\right)^{D}\cdot \sum_{\ell=0}^{D-1}(\ell+1)\left(\frac{1}{\la}+\frac{Km}{n}\right)^{\ell}
&\leq \left(1+o(1)\right)Km\left(\frac{\la}{n}\right)^{D}\cdot \sum_{\ell=0}^{\infty}(\ell+1)\left(\frac{1}{1+\eta/2}\right)^{\ell}\\
&= \left(1+o(1)\right)Km\left(\frac{\la}{n}\right)^{D}\cdot  \frac{(1+2\eta)^2}{\eta^2}.
\end{split}
\]
Thus, using these bounds in \eqref{eq:bound:sec:mo} shows 
It follows that
\[
\E[f(Y)^2]\leq \big(1+o(1)\big)m\left(\frac{\la}{n}\right)^D\left(\frac{K(1+2\eta)^2}{\eta^2}+\frac{n}{m}\left(\frac{1}{\la}+\frac{Km}{n}\right)^D\right)=\big(1+o(1)\big)m\left(\frac{\la}{n}\right)^D\frac{9K}{\eta^2},
\]
where the last equality holds because $\eta<1$ and for large enough $n$,
\[
\frac{n}{m}\left(\frac{1}{\la}+\frac{Km}{n}\right)^D\leq \frac{n}{m}\left(\frac{1}{1+\eta/2}\right)^{-D}\leq \frac{n}{m}e^{-\frac{C}{2}\log(n/m)}= o(1),
\]
since we chose $C>2$ and $m=o(n)$ holds.
Also, recall that $\E[f(Y)\cdot x]=m\left(\frac{\la}{n}\right)^{(D+1)/2}$, and a direct calculation shows $\E[x^2]=\frac{\la}{n}\E[\langle x_1,x_2\rangle^2]=\frac{m\la}{n}$. Therefore, $\Corr_{\leq D}^2\geq \frac{\eta^2}{9K}-o(1)$, which concludes the proof.
\end{proof}

\subsection{Stochastic Block Model}

This section is devoted to the proof of Theorem~\ref{thm:main-sbm}(b), which is based on~\cite{HS-bayesian}.

\begin{proof}[Proof of Theorem~\ref{thm:main-sbm}(b)]
Consider the following estimator introduced in \cite[Section 2]{HS-bayesian}. For a simple graph $\al\in \{0,1\}^{\bcn}$, let $\phi_{\al}(Y):=\prod_{i<j}(Y_{ij}-\frac{d}{n})^{\al_{ij}}$. As in the previous section, let $\cS_D$ denote the set of self-avoiding paths from vertex $1$ to $2$ of length $D$. Consider $f\in \R[Y]_{\le D}$ where
\[
f(Y)=\frac{1}{|\cS_D|}\sum_{\al\in \cS_D}\phi_{\al}(Y)\equiv \frac{1}{|\cS_D|}\sum_{\al\in \cS_D}\prod_{i<j}\left(Y_{ij}-\frac{d}{n}\right)^{\al_{ij}}.
\]
It was shown in Lemma 4.9 in the arxiv version of \cite{HS-bayesian} that if $d\la^2>1+\eps$ and $D\geq C \log n$ for some $C=C(\eps,q,\pi,Q)$,
\[
\E \big[f(Y)^2\big]\leq c\left(\frac{d}{n}\right)^{2D}\E \Bigg[\prod_{(i,j)\in \al_1\cup \al_2}\bigg(\frac{Q_{\sigma^\star_i,\sigma^\star_j}}{d}-1\bigg)\Bigg], 
\]
where $c > 0$ depends only on $q,\pi,Q$ and $\al_1,\al_2\in \mathcal{A}_D$ are two self-avoiding paths that only share vertices $1$ and $2$. Since $\al_1\cup\al_2$ is a cycle of length $2D$, we can directly compute
\[
\E \Bigg[\prod_{(i,j)\in \al_1\cup \al_2}\bigg(\frac{Q_{\sigma^\star_i,\sigma^\star_j}}{d}-1\bigg)\Bigg]=\Tr\left(\frac{1}{d}\,\diag(\pi)\,Q-\pi \bone^{\top}\right)=\sum_{k=2}^{q}\la_k(T)^{2D},
\]
where $T$ is defined in Eq.~\eqref{eq:T}. Let $s$ be the sum of the multiplicities of the eigenvalues that equals $\la_2(T)$ up to potential sign change. Then, the sum in the RHS is $s\la^{2D}(1+o(1))$ for $D\geq C\log n$. 
Moreover, Lemma~\ref{lem:sbm:c:path} yields that
\[
\E\big[f(Y)(Q_{\sigma^\star_1,\sigma^\star_2}-d)\big]=\left(\frac{d}{n}\right)^{D}d\cdot \E \Bigg[\prod_{(i,j)\in \al+\bone_{(1,2)}}\bigg(\frac{Q_{\sigma^\star_i,\sigma^\star_j}}{d}-1\bigg)\Bigg]=\frac{d^{D+1}}{n^D}\cdot \sum_{k=2}^{q}\la_k(T)^{D+1}.
\]
Thus, $|\EE[f(Y)(Q_{\sigma^\star_1,\sigma^\star_2}-d)]|\geq \frac{(d\la)^{D+1}}{n^{D}}s(1-o(1))$. It follows that
\[
\EE[x^2] \cdot \Corr_{\leq 2C\log n}^{2}\geq  \frac{\frac{(d\la)^{2D+2}}{n^{2D}} \cdot s^2(1-o(1))}{c\left(\frac{d\la}{n}\right)^{2D}\cdot s(1+o(1))}=\frac{d^2\la^2 s}{c}(1-o(1)).
\]
Finally, note that $\EE[x^2] \equiv \EE[(Q_{\sigma^\star_1,\sigma^\star_2}-d)^2]$ is of constant order.
\end{proof}

\addcontentsline{toc}{section}{Acknowledgments}
\section*{Acknowledgments}
We are grateful to Jonathan Niles-Weed for discussions that helped inspire the framework presented in Section~\ref{sec:overview-proof}. We are grateful to Elchanan Mossel for numerous discussions and valuable insights. We thank Lenka Zdeborov{\'a} for introducing us to the spiked Wigner model with growing rank (also known as symmetric matrix factorization) and providing the reference~\cite{extensive-rank}. We thank Daniel Guti\'errez Espinoza for correcting a factor of $1/2$ in Section~\ref{sec:subm-upper}.

\addcontentsline{toc}{section}{References}
\bibliographystyle{alpha}
\bibliography{main}

\end{document}